\documentclass[10pt, a4paper,centertags]{article}
\usepackage[french, english]{babel}
\usepackage{todonotes}
\usepackage[utf8,latin1]{inputenc}
\usepackage{amsmath}
\usepackage{amssymb}
\usepackage{amsthm}
\usepackage{latexsym,tikz}
\usepackage{hyperref}
\usepackage{enumerate}
\usepackage{enumitem}
\usepackage[all]{xy}
\usepackage{mathrsfs}
\usepackage{dsfont}
\usepackage{tikz}
\usetikzlibrary{decorations.pathmorphing}
\usepackage{appendix}
\usetikzlibrary{decorations.pathreplacing}
\usepackage{pdfpages} 
\usepackage{bm}

\usepackage[nottoc]{tocbibind}

\usepackage[margin=2.9cm]{geometry}
\newtheorem{thm}{Theorem}[section]
\newtheorem{cor}[thm]{Corollary}
\newtheorem{lem}[thm]{Lemma}
\newtheorem{prop}[thm]{Proposition}
\newtheorem{defn}[thm]{Definition}

\newtheorem{rem}[thm]{Remark}

\newtheorem{crit}[thm]{Criterion}
\newtheorem{nota}[thm]{Notation}
\newtheorem{fact}[thm]{Fact}

\def\fA{\mathfrak{A}}
\def\rA{\mathrm{A}}
\def\bA{\mathbb{A}}

\def\fB{\mathfrak{B}}
\def\bC{\mathbb{C}}

\def\det{\mathrm{det}}

\def\F{\mathbb{F}}

\def\rG{\mathrm{G}}

\def\GL{\mathrm{GL}}

\def\Hom{\mathrm{Hom}}

\def\ind{\mathrm{ind}}

\def\rL{\mathrm{L}}

\def\max{\text{max}}

\def\rM{\mathrm{M}}

\def\rN{\mathrm{N}}

\def\rP{\mathrm{P}}
\def\cP{\mathcal{P}}

\def\rR{\mathrm{R}}

\def\res{\mathrm{res}}

\def\SL{\mathrm{SL}}

\def\rU{\mathrm{U}}

\begin{document}





\hypersetup{							
pdfauthor = {Peiyi Cui},			
pdftitle = {mod l blocks of SL_n},			
}					
\title{$\ell$-modular blocks of $\mathrm{Rep}_k(\mathrm{SL_n}(F))$}
\author{Peiyi Cui \footnote{peiyi.cui@amss.ac.cn, Morningside Center of Mathematics, No.55, Zhongguancun East Road, Beijing, 100190, China.}}

\date{\vspace{-2ex}}
\maketitle

 \begin{abstract}
Let $F$ be a non-archimedean local field with a residual characteristic $p$, and $k$ an algebraically closed field with characteristic $\ell$, where $\ell\neq p$. Let $\mathrm{Rep}_k(\mathrm{SL}_n(F))$ be the category of smooth $k$-representations of $\mathrm{SL}_n(F)$. In this work, we establish the block decomposition of $\mathrm{Rep}_k(\mathrm{SL}_n(F))$ under the condition that $p$ does not divide the order of the Weyl group of $\mathrm{SL}_n(F)$.
 \end{abstract}

\tableofcontents

\section{Introduction}
\subsection{Backgrounds}
\subsubsection{History}
The block decomposition, also known as Bernstein decomposition, has been established in \cite{BD84}. It provides a finest decomposition of the category of smooth complex representations of a $p$-adic group into a direct product of some full subcategories. The involved full subcategories are called the blocks of this category. The structure of the blocks and their center attract a lot of interest: each block is equivalent to the category of modules of an affine Hecke algebra, and the center is isomorphic to the center of this algebra.

In the late 1980s, Vignéras suggested studying the $\ell$-modular representations of $p$-adic groups, which are representations defined over a vector space of an algebraically closed field with positive characteristic different from $p$. Due to the similarity between complex representations and $\ell$-modular representations in many of their properties, it was always considered that the latter is a generalization of the former. However, with the emergence of cuspidal but non-supercuspidal representation in \cite{V89}, people gradually realized that they have many differences. One of the most significant differences is the block decomposition of the category of smooth $\ell$-modular representations, which has been unknown for general $p$-adic groups so far. It has been established for general linear groups by Vignéras in \cite{V2}, and for their inner forms by Sécherre and Stevens in \cite{SS}. In both cases, the blocks have a similar formulation to the complex setting. With the discovery of an irreducible $\ell$-modular representation whose supercuspidal support is not unique in \cite{Da}, and the construction of projective generators of blocks in \cite{Da12} and \cite{Helm}, we realized that the formulation in \cite{V2} and \cite{SS} is impossible for general $p$-adic groups. Beyond that, there is essential difference comparing to complex setting. A block is equivalent to the category of the modules of an algebra, which is an affine Hecke algebra for complex representations, but is more complicates in $\ell$-modular setting, even for the two known cases in \cite{V2} and \cite{SS}. 

In this work, we study the blocks of the category of smooth $\ell$-modular representations of special linear groups under the tameness condition. Let $F$ be a non-archimedean local field with a residual characteristic $p$, and let $k$ be an algebraically closed field with characteristic $\ell$ such that $\ell \neq p$. Denote $\rG$ as a Levi subgroup of $\mathrm{GL}_n(F)$, and $\rG' = \rG \cap \mathrm{SL}_n(F)$. We assume that $p$ does not divide the order of the Weyl group $\vert W_{\rG'}\vert$ of $\rG'$ (the tameness condition). Let $\mathrm{Rep}_k(\rG')$ be the category of smooth $k$-representations of $\rG'$. In the last example of \cite{C3}, we already know that the formulation of a block of $\mathrm{Rep}_k(\rG')$ differs both from the complex setting and the case of $\mathrm{Rep}_k(\rG)$. In this work, we will provide a complete construction of the blocks of $\mathrm{Rep}_k(\rG')$, and we will construct a projective generator for each block. 

It is worth noticing, projective generators play an important role in two aspects. On one hand, when $\ell\neq 0$, a block is still equivalent to the category of modules of the endomorphism algebra of a projective generator. The structure of this algebra is currently a mystery. A good choice of projective generator can greatly simplify computations for this algebra. On the other hand, in past decade, an idea called ``reduction to depth zero"  has been widely recognised as a promising direction on the study of blocks, aiming on giving a bijection from blocks of a $p$-adic group to the depth zero blocks of some other $p$-adic groups. This article uses a different approach, but is very relevant to this scenario. Due to the absence of endo-equivalence in the case of $\rG'$, we cannot state our result in the language of reduction to depth zero at the moment, but for each block, we can see via its projective generator a natural candidate of the depth zero block that it could correspond.

\subsubsection{Decomposition of category}
Let $\rA$ be a $p$-adic group, i.e., $\rA=\mathbf{A}(F)$, where $\mathbb{A}$ is a connected reductive group defined over $F$. We fix a Borel subgroup. Let $\rL$ be a standard Levi subgroup of $\rA$, and $\rP=\rL\rN$ the standard parabolic subgroup. We denote by $i_{\rL}^{\rA}$ the normalised parabolic induction with respect to $\rP$. Consider an irreducible $k$-representation $\pi$ of $\rA$. We say that $\pi$ is:
\begin{itemize}
\item \textbf{cuspidal}, if $\pi$ is neither a sub-representation nor a quotient-representation of $i_{\rL}^{\rA}\tau$ for any proper $\rL$ and $k$-representation $\tau$ of $\rL$.
\item \textbf{supercuspidal}, if $\pi$ is not a subquotient of $i_{\rL}^{\rA}\tau$.
\end{itemize}

Let $\tau$ be supercuspidal (resp. cuspidal), and $\pi$ irreducible. We say that $(\rL,\tau)$ belongs to the \textbf{supercuspidal (resp. cuspidal) support} of $\pi$, if $\pi$ is a subquotient (resp. a sub-representation or a quotient-representation) of $i_{\rL}^{\rA}\tau$. Consider $(\rL, \tau)$ as a supercuspidal pair, where $\tau$ is supercuspidal of $\rL$. Denote by $[\rL, \tau]$ the supercuspidal class of $(\rL,\tau)$, which is an $\rA$-conjugacy class up to twist of unramified characters of $\rL$. Denote by $\mathrm{Rep}_k(\rA)_{[\rL,\tau]}$ the full-subcategory generated by $[\rL,\tau]$. The supercuspidal support of irreducible subquotients of its objects are contained in $[\rL, \tau]$. Let $\mathcal{SC}$ be the set of supercuspidal classes, and $C$ a finite subset of $\mathcal{SC}$. We define the full-subcategory $\mathrm{Rep}_k(\rA)_C$ generated by $C$: it contains objects that the supercuspidal support of irreducible subquotients belonging to $C$.

When $\ell=0$, the well known result of the block decomposition states that:
$$\mathrm{Rep}_{\bC}(\rA)\cong\prod_{\mathcal{SC}}\mathrm{Rep}_{\bC}(\rA)_{[\rL,\tau]}.$$

When $\ell\neq 0$ and $\rA=\rG$ (see \cite{V2}) or the inner forms of $\rG$ (see \cite{SS}), it is proved that 
$$\mathrm{Rep}_{k}(\rA)\cong\prod_{\mathcal{SC}}\mathrm{Rep}_{k}(\rA)_{[\rL,\tau]},$$
However, the above decomposition does not exist in general. When $\rA=\rG'$, we show (Theorem \ref{thmblocksdecom}) the block decomposition should be
$$\mathrm{Rep}_{k}(\rG')\cong\prod\mathrm{Rep}_{k}(\rG')_{C}.$$
A block is generated by several supercuspidal classes rather than a single one in general. We determine these classes and give a projective generator for each block (Theorem \ref{thmblocksdecom}).


\subsection{Main results}
Let $\rL'$ be a Levi subgroup of $\rG'$, and $\tau'$ be a supercuspidal representation of $\rL'$. We give more details on the supercuspidal classes that determine the block containing $\mathrm{Rep}_k(\rG')_{[\rL',\tau']}$. We begin by outlining the advances in type theory, which forms a piece of the puzzle of the main result (Theorem \ref{thmblocksdecom}).

\subsubsection{Type theory}
Type theory is the primary technical support. In this article, we establish the cover theory of $\ell$-modular representations of $\rG'$, in order to understand the relation between parahoric induction and parabolic induction. To introduce the role it plays, we start from supercuspidal types, also known as maximal simple types, which are established to construct supercuspidal representations of $p$-adic groups from finite groups. A supercuspidal type is a pair consisting of an open compact subgroup and a special irreducible representation of it. In particular when $\rA=\rL$, it is a pair $(J,\lambda)$ with the following characteristics:
\begin{enumerate}
\item $J$ is an open compact subgroup of $\rL$.
\item $\lambda\cong\kappa\otimes\sigma$ is a special irreducible representation of $J$, such that:
\begin{itemize}
\item  The quotient $\mathbb{L}:=J\slash J^1$ is a finite reductive group, where $J^1$ is the pro-$p$ radical.
\item $\sigma$ is inflated from a supercuspidal representation of $\mathbb{L}$. We refer to this as \textbf{the moderate part}.
\item $\kappa$ is more technical, referred to as \textbf{the wild part}. This pair $(J,\kappa)$ is also called a \textbf{wild pair} in $\rL$.
\end{itemize}
\end{enumerate}

We have similar properties for a supercuspidal type $(J',\lambda')$ of $\rL'$. In particular, under the tameness condition, it can be obtained from a supercuspidal type $(J,\lambda)$ of $\rL$ (see Proposition \ref{proptamecond} for more details), where $\rL$ satisfies $\rL\cap\rG'=\rL'$. This yields:
\begin{enumerate}
\item $J'$ is an open compact subgroup of $\rL'$, and $\mathbb{L}':=J'\slash J^{1'}$ is a finite reductive group, where $J'=J\cap \rL'$ and $J^{1'}=J^1\cap\rL'$.
\item $\lambda'\cong\kappa'\otimes\sigma'$, where the wild part $\kappa'$ is determined by $\kappa$, and the moderate part $\sigma'$ is inflated from a supercuspidal representation of $\mathbb{L}'$.
\end{enumerate}

Each supercuspidal representation contains a supercuspidal type. After a unicity property, we can identify the supercuspidal class $[\rL',\tau']$ with the conjugacy class of its type $(J',\lambda')$. In this work, we show that the block containing $\mathrm{Rep}_k(\rG')_{[\rL',\tau']}$ is constructed from a union $\ell$-blocks of a finite reductive group.  This finite reductive group is determined by the wild part of $(J,\lambda)$, and this union of $\ell$-blocks is determined by the moderate part of $(J',\lambda')$.

This finite reductive group is found as a quotient of an open compact subgroup inside $\rL_{\text{max}}'$, which is  a Levi subgroup generally larger than $\rL'$, obtained from the endo-equivalence class of $[\rL',\tau']$. The latter is an equivalent relation on $\mathcal{SC}$ of $\rG$, which has been established in \cite{BuKuI}. Its connected components are endo-equivalence classes (The absence of endo-equivalence theory for $\mathcal{SC}$ of $\rG'$ is the main reason we transit to $\rG$ in this work). The wild pair $(J,\kappa)$ defines the endo-equivalence class of $[\rL,\tau]$. In other words, endo-equivalence is also defined on the set of wild pairs of Levi subgroups of $\rG$. Understanding endo-equivalence from the Galois side provides a more intrinsic perspective, and the relation on both sides has been studied in \cite{BH}. Let $\rL_{\text{max}}$ be the maximal standard Levi subgroup containing a wild pair $(J_{\text{max}},\kappa_{\text{max}})$ that is endo-equivalent to $(J,\kappa)$. We refer to $\rL_{\text{max}}$ as \textbf{the homogeneous Levi} of $(J,\kappa)$. It serves as the bound of cuspidal representations arising from $(\rL,\tau)$ in the following sense: for any Levi subgroup $\rM$ larger than $\rL_{\text{max}}$, there is not cuspidal subquotient of $i_{\rL}^{\rM}\tau$. The  group $\rL_{\text{max}}'$ is the intersection $\rL_{\text{max}}\cap\rG'$. 

Let $J_{\text{max}}'=J_{\text{max}}\cap\rG'$, and $J_{\text{max}}^{1'}=J_{\text{max}}^1\cap\rG'$, where $J_{\text{max}}^1$ is the pro-$p$ radical of $J_{\text{max}}$. The quotient $\mathbb{L}_{\text{max}}':=J_{\text{max}}'\slash J_{\text{max}}^{1'}$ is the finite reductive group we are seeking.

\subsubsection{The blocks}
The moderate part $(J',\sigma')$ defines a full-subcategory $\mathcal{B}$ of $\mathrm{Rep}_k(\mathbb{L}_{\text{max}}')$ (a product of $\ell$-blocks of $\mathbb{L}_{\text{max}}'$). Meanwhile, the tensor product $(J',\kappa'\otimes\rho')$ is a supercuspidal type for any supercuspidal $\rho'$ of $\mathbb{L}'$. Now consider all supercuspidal types by taking $\rho'$ in the supercuspidal support of an irreducible representation in $\mathcal{B}$. Let $[\tau',\sim]$ be the union of the associated supercuspidal classes, and $\mathrm{Rep}_k(\rG')_{[\tau',\sim]}$ be the full-subcategory generated by $[\tau',\sim]$.

\begin{thm}[Theorem \ref{thmblocksdecom}]
The full-subcategory $\mathrm{Rep}_k(\rG')_{[\tau',\sim]}$ is the block containing $\mathrm{Rep}_k(\rG')_{[\rL',\tau']}$.
\end{thm}

\subsubsection{Projective generators}
We introduce a projective generator for $\mathrm{Rep}_k(\rG')_{[\tau',\sim]}$. The wild pairs $(J_{\text{max}}',\kappa_{\text{max}}')$ in the above section are not unique. $(J_{\text{max}},\kappa_{\text{max}})$ gives finitely many wild pairs of $\rL_{\text{max}}'$: $(J_{max,b}',\kappa_{max,b}'),b\in B$ with an index set $B$ (see \ref{defnB}). Let $\cP(\mathcal{B})$ be a projective generator of $\mathcal{B}$. We have:

\begin{thm}[Theorem \ref{thmblocksdecom}]
The direct sum
$$\cP:=\bigoplus_{b\in B}i_{\rL_{\text{max}}'}^{\rG'}\ind_{J_{max,b}'}^{\rL_{\text{max}}'}\kappa_{max,b}'\otimes\cP(\mathcal{B}),$$
is a projective generator of $\mathrm{Rep}_k(\rG')_{[\tau',\sim]}$.
\end{thm}

One notable feature of this theorem is that the aforementioned projective generator is induced from a specifically determined Levi subgroup, which may be useful for various applications.This is the first time we have constructed a projective generator in this form. The idea of the projective cover of a supercuspidal type was initially suggested by Vignéras in \cite{V1}. Helm further showed that it can be decomposed as $\kappa\otimes\cP$ in \cite{Helm}, where $\cP$ is projective for $\mathbb{L}$. These works are inspiration for the formulation presented in the above theorem.

\subsection{Examples of $\mathrm{SL}_2(\mathbb{Q}_p)$}
\label{depthzeroexample}
This can be regarded as an application of Theorem \ref{thmblocksdecom}. 

We introduce these blocks and their projective generators to highlight the differences compared to the complex setting. Recall two elementary facts below. We illustrate how they influence the structure of blocks in $\ell$-modular setting (see Example \ref{example1}, \ref{example2} for Fact 2, and Example \ref{example3} for Fact 1), in order to assure the readers that the technical operations in this work are necessary.
\begin{fact}
\label{twodiffs}
\begin{enumerate}
\item For finite groups, indecomposable reducible $\ell$-modular representations exist in general; 
\item For $p$-adic groups or finite groups of Lie type, cuspidal $\ell$-modular representations are not always supercuspidal (see \cite{V89}). 
\end{enumerate}
\end{fact}

For $\rG'=\mathrm{SL}_2(\mathbb{Q}_p)$, there are two maximal open compact subgroups that are not $\rG'$-conjugate: $K_1=\mathrm{SL}_2(\mathbb{Z}_p)$ and $K_2=xK_1x^{-1}$, where $x=(p,1)$ is a diagonal matrix in $\rG=\mathrm{GL}_2(\mathbb{Q}_p)$. The groups $K_1,K_2$ correspond to two vertices of a chamber of the Bruhat-Tits building of $\rG'$. Let $K_i^{1}$ be the pro-$p$ radical of $K_i$ for $i=1,2$, then the quotient $K_i\slash K_i^1\cong\mathrm{SL}_2(\mathbb{F}_p)$. For an irreducible $k$-representation $\pi$ of $\mathrm{SL}_2(\mathbb{F}_p)$, denote by $\cP_{\pi}$ its projective cover. We use the following notation:
\begin{itemize}
\item $T$ for the diagonal torus in $\rG'$,
\item $T_{0}$ for the open compact subgroup of $T$ with values in $\mathbb{Z}_p$, 
\item $T_p$ for the diagonal torus in $\mathrm{SL}_2(\mathbb{F}_p)$.
\end{itemize}

\subsubsection{Example 1}
\label{example1}
For a ramified character $\tau$ of $T$, the subcategory $\mathrm{Rep}_k(\rG')_{[T,\tau]}$ forms a block of $\mathrm{Rep}_k(\rG')$. Define $\tau_0=\tau\vert_{T_0}$. Denote by $\mathcal{P}_{\tau_0}$ the projective cover of $\tau_0$. Then, 
$$i_{T}^{\rG'}\ind_{T_0}^{T}\cP_{\tau_0}$$
is a projective generator of $\mathrm{Rep}_k(\rG')_{[T,\tau]}$. It is worth noting that the representation $i_{T}^{\rG'}\ind_{T_0}^{T}\tau_0$ is not projective, which is different from the complex setting.

\subsubsection{Example 2}
\label{example2}
The subcategory $\mathrm{Rep}_k(\rG')_{[T,\mathds{1}]}$ is a block of $\mathrm{Rep}_k(\rG')$. Due to the existence of cuspidal but non-supercuspidal representations in $\mathrm{Rep}_k(\rG')_{[T,\mathds{1}]}$ when $\ell\vert p+1$ (see \cite{V89}), the structure of a projective generator is more intricate. 

There are three possibilities for the set of irreducible subquotients of $\ind_{T_p}^{\mathrm{SL}_2(\mathbb{F}_p)}\mathds{1}$:
\begin{enumerate}[label=(\roman*)]
\item $\{\mathds{1},\mathrm{St}\}$, where $\mathrm{St}$ is the Steinberg representation.
\item $\{\mathds{1},\sigma,\mathrm{sgn}\}$, where $\sigma$ is cuspidal, and $\mathrm{sgn}$ is a non-trivial character. Define $\cP:=\cP_{\mathds{1}}\oplus\cP_{\sigma}\oplus\cP_{\mathrm{St}}$.
\item $\{\mathds{1},\sigma_1,\sigma_2\}$, where $\sigma_1,\sigma_2$ are cuspidal. Define $\cP:=\cP_{\mathds{1}}\oplus\cP_{\sigma_1}\oplus\cP_{\sigma_2}$.
\end{enumerate}

For case $(i)$, there is not cuspidal representation in $\mathrm{Rep}_k(\rG')_{[T,\mathds{1}]}$, and $i_{T}^{\rG'}\ind_{T_0}^{T}\cP_{\mathds{1}}$ is a projective generator of $\mathrm{Rep}_k(\rG')_{[T,\mathds{1}]}$.

For case $(ii)$ and $(iii)$,
$$\ind_{K_1}^{\rG'}\cP\oplus\ind_{K_2}^{\rG'}\cP$$
is a projective generator of $\mathrm{Rep}_k(\rG')_{[T,\mathds{1}]}$.

Cuspidal but non-supercuspidal representation can be obtained by considering subquotient of parabolic induction from supercuspidals of smaller Levi subgroups. However, their appearance is accompanied by new morphisms which cannot be obtained from smaller Levi subgroups.
\begin{rem}[case $(ii)$ and $(iii)$]
\begin{itemize}
\item The projective representation $i_{T}^{\rG'}\ind_{T_0}^{T}\cP_{\mathds{1}}$ does not give enough morphisms. There exists cuspidal but non-supercuspidal representations in $\mathrm{Rep}_k(\rG')_{[T,\mathds{1}]}$. Any object induced from $T$ cannot map surjectively to cuspidal representations.

\item The direct sum is necessary. There exist two cuspidal but non-supercuspidal representations, $\tau_1,\tau_2$, in $\mathrm{Rep}_k(\rG')_{[T,\mathds{1}]}$, induced from $K_1$ and $K_2$, respectively. An object induced from $K_1$ cannot map surjectively to $\tau_2$.
\end{itemize}
\end{rem}

\subsubsection{Example 3}
\label{example3}
We have discussed blocks generated by a single supercuspidal class, and now we will consider a block generated by two supercuspidal classes. It has been discussed in \cite{C3}, and we will provide additional details here.

Let $\tau$ be a supercuspidal representation of $\rG'$, equivalent to $\ind_{K_i}^{\rG'}\sigma$ for either $i=1$ or $i=2$, where $\sigma$ is a supercuspidal representation of $\mathrm{SL}_2(\mathbb{F}_p)$. When $\ell\neq2$ and $p\neq2$, the irreducible subquotients of the projective cover $\mathrm{Irr}(\cP_{\sigma})$ may include $\{\sigma,\sigma'\}$ (as seen in the last example in \cite{C3}), where $\sigma\neq\sigma'$.

Let's assume $i=1$. We have four supercuspidal representations:
\begin{itemize}
\item $\tau=\tau_1\cong\ind_{K_1}^{\rG'}\sigma$,
\item $\tau_2:=\ind_{K_1}^{\rG'}\sigma'$,
\item $\tau_3:=\ind_{K_2}^{\rG'}\sigma$,
\item $\tau_4:=\ind_{K_2}^{\rG'}\sigma'$.
\end{itemize}
They are $\mathrm{GL}_2(\mathbb{Q}_p)$-conjugate, and they define four distinct supercuspidal classes.

Define $\mathcal{R}_1:=\mathrm{Rep}_k(\rG')_{[\rG',\tau_1]\cup[\rG',\tau_2]}$ as the full-subcategory containing object $\Pi$ such that the supercuspidal support of any irreducible subquotient of $\Pi$ is contained in $[\rG',\tau_1]\cup[\rG',\tau_2]$. Similarly, we define $\mathcal{R}_2:=\mathrm{Rep}_k(\rG')_{[\rG',\tau_3]\cup[\rG',\tau_4]}$. These are two blocks generated by two supercuspidal classes. 
\begin{itemize}
\item $\ind_{K_1}^{\rG'}\cP_{\sigma}\oplus\cP_{\sigma'}$ is a projective generator of $\mathcal{R}_1$.
\item $\ind_{K_2}^{\rG'}\cP_{\sigma}\oplus\cP_{\sigma'}$ is a projective generator of $\mathcal{R}_2$.
\end{itemize}

\subsection{Notations and the strategy}
Let $\rG$ be a Levi subgroup of $\mathrm{GL}_n(F)$, and $\rG'=\rG\cap\mathrm{SL}_n(F)$ a Levi subgroup of $\mathrm{SL}_n(F)$. Fix a maximal split torus. The intersection with $\rG'$ gives a bijection between the set of standard Levi of $\rG$ and those of $\rG'$. Let $\rL'$ be a standard Levi subgroup of $\rG'$, and $\rL$ a standard Levi of $\rG$ such that $\rL\cap\rG'=\rL'$. Let $\rP=\rL\rN$ be a parabolic subgroup, with unipotent radical $\rN$. We write $\rL\cong\prod_{i\in I}\mathrm{GL}_{n_i}(F),n_i\in\mathbb{N}$ for a finite index set $I$. Let $K$ be a subgroup of $\rG$, we always denote by $K'$ the intersection $K\cap\rG'$.

The strategy to establish the blocks of $\mathrm{Rep}_k(\rG')$ is by applying Morita's method, which is a criterion for decomposing a category:

\begin{thm}[Morita's equivalence]
\label{ThmMorita}
Suppose there are two projective objects $\cP_1$ and $\cP_2$ in an abelian category $\mathcal{A}$, satisfying the following conditions:
\begin{enumerate}
\item The sets of simple subquotients of $\mathrm{Irr}(\cP_1)$ and $\mathrm{Irr}(\cP_2)$ are disjoint.
\item The union of $\mathrm{Irr}(\cP_1)\cup\mathrm{Irr}(\cP_2)$ is equal to the set of simple objects in $\mathcal{A}$.
\item For each $\pi_i\in\mathrm{Irr}(\cP_i)$, there is a surjective morphism $\cP_i\rightarrow \pi_i$, for $i=1,2$.
\end{enumerate}
Then there exists two full-subcategories $\mathcal{A}_1$ and $\mathcal{A}_2$, such that
$$\mathcal{A}\cong\mathcal{A}_1\times\mathcal{A}_2,$$
where $\Pi\in\mathcal{A}_i$ if and only if its irreducible subquotients belong to $\mathrm{Irr}(\cP_i)$ for $i=1,2$. Moreover, when $\cP_i$ are \textbf{faithfully projective} (see \cite{Helm}[Definition 2.2]), the functor
\begin{equation}
\label{equaMoritaHom}
\Pi\rightarrow \mathrm{Hom}(\cP_i,\Pi)
\end{equation}
is an equivalence of the category $\mathcal{A}_i$ to the category of $\mathrm{End}(\cP_i)$-modules. In particular, the centre of $\mathcal{A}_i$ is the centre of $\mathrm{End}(\cP_i)$. In this case, $\cP_i$ is called a projective generator of $\mathcal{A}_i$ for $i=1,2$.
\end{thm}

The idea is to establish a family of projective objects verifying the above $3$ conditions, enabling the decomposition of the category. The second step is to show the involved full-subcategories are non-split.

As we have explained in the preceding section, the representations in the form of $i_{\rL_{\text{max}}'}^{\rG'}\ind_{J_{max,b}'}^{\rL_{max,b}'}\kappa_{max,b}'\otimes\cP(\mathcal{B})$ are building pieces of a projective generator of a block of $\mathrm{Rep}_k(\rG')$. This construction encompasses a moderate part $\cP(\mathcal{B})$ and a wild part $\kappa_{max,b}$. We appoach these components separately.

In Section \ref{sectionfinite} we study the $\ell$-blocks of some special finite reductive groups, which is crucial in the study of the moderate part. We introduce the notion of $\ell$-parablocks, that are unions of $\ell$-blocks compatible with parabolic induction. These $\ell$-parablocks play a key role in defining $\mathcal{B}$. Moving on to $p$-adic groups, we first recall some essential results of type theory of $\rG$ in Section \ref{sectionGcover}, including cuspidal types and cover theory. Subsequently, in Section \ref{sectionG'covers}, we return to $\rG'$. In the initial subsection, we establish cover theory in $\ell$-modular setting. The following subsection is the most technical part of this work, which involves intricate operations on wild pairs. One primary goal is to provide compatibility between cover theory and parabolic induction, allowing us to describe the irreducible components of $i_{\rL_{\text{max}}'}^{\rG'}\ind_{J_{max,b}'}^{\rL_{max,b}'}\kappa_{max,b}'\otimes\cP(\mathcal{B})$, and morphisms mapping from it. After these discussions, we verify that the constructed projective objects satisfy the conditions of Morita's criterion and establish a decomposition in Section \ref{sectiondecomp}. We show it is the block decomposition by proving the involved full-subcategories are non-split. At the end of this work, we give description of supercuspidal classes generating the block containing $\mathrm{Rep}_k(\rG')_{[\rL',\tau']}$.

\subsection{Acknowledgements}
This work originated from the author's PhD project, which has been going on for many years. The author received assistance from various individuals throughout this process. Firstly, I would like to express my gratitude to my PhD supervisor, Anne-Marie Aubert, who guided me through this project and introduced me to various aspects of $\ell$-modular representations. I would like to thank Shaun Stevens for frequent discussions and helpful suggestions on type theory during my stay in Norwich. Without these discussions, this work might have taken much longer to complete. I thank Shaun Stevens and Thomas Lanard for pointing out two mistakes in an early version. I appreciate the support and academic advice provided by Alberto Mínguez, Vincent Sécherre, Dmitry Gourevitch, and Erez Lapid. I would like to thank Guy Henniart and Marie-France Vignéras for their continued interest in this problem, which has been a great motivation for the author. This work was mostly written when I worked in University of East Anglia (funded by EPSRC grant EP/V061739/1) and was finished in Weizmann Institute of Science (funded by the ISF 1781/23 grant BSF 2019724). I would like to thank their nice working environments.

\section{Finite reductive groups}
\label{sectionfinite}

\subsection{$\ell$-parablocks}

Let $\bA$ be the group of $k_F$-rational points of a finite reductive group defined over $k_F$. In this section we introduce some full-subcategories of $\bA$ that we call $\ell$-parablocks, and we study the projective cover of irreducible representations of $\bA$.

Recall that there exists block decomposition of the category $\mathrm{Rep}_k(\bA)$, and the involved full-subcategories are known as blocks of $\mathrm{Rep}_k(\bA)$.  In this work, to emphasis $\bA$ is a finite group, we always add a prefix and call them \textbf{$\ell$-blocks}. Let $R$ be a full-subcategory of $\mathrm{Rep}_k(\bA)$. Denote by $\mathrm{Irr}(R)$ the set of isomorphism classes of irreducible representations in $R$, by $\mathbf{SC}(R)$ the set of supercuspidal support of objects in $\mathrm{Irr}(R)$, and by $\mathbf{Cusp}(R)$ the set of cuspidal support of objects in $\mathrm{Irr}(R)$. For a object $\Pi\in\mathrm{Rep}_k(\bA)$, we denote by $\mathrm{Irr}(\Pi)$ the set of isomorphism classes of irreducible subquotients of $\Pi$, and $\mathbf{SC}(\Pi)$ the set of supercuspidal support of subquotients of $\Pi$.

\begin{lem}
\label{finlem001}
Let $\mathrm{B}$ be an $\ell$-block of $\mathrm{Rep}_k(\bA)$, and $\pi_1,\pi_2\in\mathrm{Irr}(\mathrm{\bA})$. We say they are linked (i.e $\pi_1\leftrightarrow\pi_2$) if them are subquotients of a same indecomposable projective object. Denote by $\sim_{\ell}$ the equivalence relation generated from this link relation. Then $\mathrm{Irr}(\mathrm{B})$ is a connected component of $\sim_{\ell}$.
\end{lem}

\begin{proof}
Let $\cP_{\pi}$ be the projective cover of $\pi$, for $\pi\in\mathrm{Irr}(\bA)$. Denote by $\cP(\bA)$ the set $\{\cP_{\pi}\}_{\pi\in\mathrm{Irr}(\bA)}$. We say $\cP_{\pi_1}$ is linked with $\cP_{\pi_2}$ (i.e. $\cP_{\pi_1}\leftrightarrow\cP_{\pi_2}$) if they contain a same irreducible subquotient. Let $\sim_{\cP}$ be the equivalence relation generated by this link relation. Let $\{\cP^i, i\in I\}$ be a set of connected components of $\cP(\bA)$ via $\sim_{\cP}$. Define $\cP_i=\sum_{\cP\in\cP^i}\cP$ for each $i\in I$. Let $\mathrm{Irr}_i$ be the set of equivalence classes of irreducible subquotients of $\cP_i$. We have:
\begin{itemize}
\item $\cup_{i\in I}\mathrm{Irr}_i=\mathrm{Irr}(\bA)$.
\item $\mathrm{Irr}_i\cap\mathrm{Irr}_j=\emptyset$ if $i\neq j$.
\item For $\pi\in\mathrm{Irr}_i$ there is a surjective map from $\cP_i$ to $\pi$.
\end{itemize}
Hence $\{\cP_i,i\in I\}$ defines a family of full-subcategories $\{\mathrm{R}_i,i\in I\}$ of $\mathrm{Rep}_k(\bA)$ via Morita's method, such that $\mathrm{Rep}_k(\bA)\cong\prod_{i\in I}\mathrm{R}_i$. Hence $\mathrm{Irr}(\mathrm{B})\subset \mathrm{Irr}_i$ for an $i\in I$. On the other hand, for each $i$ and for any non-trivial partition on $\mathrm{Irr}_i=\mathrm{Irr}_{i_1}\sqcup\mathrm{Irr}_{i_2}$ there exists an indecomposable projective object $\cP\in\cP^i$, such that neither $\mathrm{Irr}(\cP)\cap\mathrm{Irr}_{i_1}$ nor $\mathrm{Irr}(\cP)\cap\mathrm{Irr}_{i_2}$ is trivial. Hence there is not finer decomposition for each $\mathrm{R}_i$ and they are blocks of $\mathrm{Rep}_k(\bA)$. Hence $\mathrm{Irr}(\mathrm{B})= \mathrm{Irr}_i$. By definition $\mathrm{Irr}_i$ is a connected component by $\sim_{\ell}$ of $\mathrm{Irr}(\bA)$.
\end{proof}

From now on, we assume that the supercuspidal support of irreducible $k$-representations of Levi subgroups of $\bA$ is \textbf{unique} up to conjugacy.

\begin{defn}
\label{defn2.5}
Consider an equivalence relation on $\ell$-blocks of $\bA$ generated by linking $\mathrm{B}_1$ and $\mathrm{B_2}$ when $\mathbf{SC}(\mathrm{B}_1)\cap\mathbf{SC}(\mathrm{B}_2)\neq\emptyset$. The product of $\ell$-blocks in a connected component is called an $\ell$-parablock.
\end{defn}

\begin{lem}
\label{lemparalblock}
Let $\pi_1,\pi_2$ be two irreducible $k$-representations of $\bA$. Suppose that $(\mathbb{M},\rho_i)$ belongs to the cuspidal support of $\pi_i$ for $i=1,2$, and that $\rho_1,\rho_2$ are in the same $\ell$-parablock. Then $\pi_1$ and $\pi_2$ are in the same $\ell$-parablock.
\end{lem}

\begin{proof}
Let $\cP_{\rho_1}$ be the projective cover of $\rho_1$, hence $\pi_1$ is a quotient of $i_{\mathbb{M}}^{\mathbb{A}}\cP_{\rho_1}$. First we show that the irreducible subquotients of $i_{\mathbb{M}}^{\bA}\cP_{\rho_1}$ are in the same $\ell$-parablock. We have
\begin{equation}
\label{finindproj}
i_{\mathbb{M}}^{\bA}\cP_{\rho_1}\cong\oplus_{j=1 }^m\cP_j,
\end{equation}
where $\cP_j$ are indecomposable projective, and one of them is the projective cover of $\pi_1$. For each $j$, there is a non-trivial morphism from $\cP_{\rho_1}$ to $\bar{r}_{\mathbb{M}}^{\bA}\cP_j$, where $\bar{r}_{\mathbb{M}}^{\bA}$ is the opposite parabolic restriction functor. Hence there is an irreducible subquotient $\sigma$ of $\cP_j$, such that $(\mathbb{M},\rho_1)$ is an irreducible subquotient of $\bar{r}_{\mathbb{M}}^{\bA}\sigma$, which implies that the supercuspidal support of $\sigma$ contains that of $\rho_1$. We deduce that all the irreducible subquotients of $\cP_j$ are in the same $\ell$-parablock of $i_{\mathbb{M}}^{\bA}\rho_1$.

The proof of Lemma \ref{finlem001} and the definition of $\ell$-parablock imply that there exists a series of indecomposable projective objects
$$\{\cP_{\rho_1}=\cP_0,\cdots,\cP_a,\cdots,\cP_n=\cP_{\rho_2}, n\in \mathbb{N}\},$$
such that either $\cP_a$ and $\cP_{a-1}$ are in a same $\ell$-block, or $\mathbf{SC}(\cP_a)\cap\mathbf{SC}(\cP_{a-1})\neq\emptyset$.
By the analysis above, we conclude that the irreducible subquotients of $i_{\mathbb{M}}^{\bA}\cP_a$ are in the same $\ell$-parablock for each $1\leq a\leq n$, so the same for $\pi_1$ and $\pi_2$.
\end{proof}


\begin{prop}
\label{fin003}
Let $\mathcal{B}$ be an $\ell$-parablock. We fix an element for each $\bA$-conjugacy class in $\mathbf{Cusp}(\mathcal{B})$. Denote by $\mathrm{Cusp}(\mathcal{B})$ the set of these representatives. Then the projective object
$$\cP(\mathcal{B})=\oplus_{(\mathbb{M},\rho)\in\mathrm{Cusp}(\mathcal{B})}i_{\mathbb{M}}^{\bA}\cP_{\rho}$$
is a projective generator of $\mathcal{B}$. 
\end{prop}

\begin{proof}
By definition, for each $\pi\in\mathrm{Irr}(\mathcal{B})$ there is a pair $(\mathbb{M}_0,\rho_0)\in\mathrm{Cusp}_{\mathcal{B}}$ and surjective morphism from $i_{\mathbb{M}_0}^{\bA}\cP_{\rho_0}$ to $\pi$. Meanwhile, by the proof of Lemma \ref{lemparalblock}, the irreducible subquotients of $i_{\mathbb{M}_0}^{\bA}\cP_{\rho_0}$ are in the same $\ell$-parablock of $\pi$.

\end{proof}

\subsection{Restriction of scalars groups}
\label{RSG}

Let $\mathbb{F}_q$ be the finite field of $q$-elements, where $q$ is a power of $p$, and $k_0\slash \F_q$ a finite field extension. Denote by $\overline{\mathbb{F}}_q$ an algebraic closure of $\mathbb{F}_q$. Let $\mathbf{A}(k_0)$ be the group of $k_0$-rational points of a reductive group $\mathbf{A}$ defined over $k_0$. \textbf{The restriction of scalars group} $\rR(\mathbf{A})$ of $\mathbf{A}$ to $\F_q$ is reductive defined over $\F_q$, such that $\rR(\mathbf{A}_0)(O)\cong\mathbf{A}(O\otimes_{\F_q}k_0)$ for any $\F_q$-algebra $O$. In particular, when $O=k_0$, we have $\rR(\mathbf{A})(\F_q)\cong\mathbf{A}(k_0)$. For any $\alpha:\mathbf{A}_1\rightarrow\mathbf{A}_2$, there is a morphism $\rR(\alpha):\rR(\mathbf{A}_1)\rightarrow\rR(\mathbf{A}_2)$ such that $\rR(\alpha)\cong\alpha$ as group morphism from $\rR(\mathbf{A}_1)(\mathbb{F}_q)$ to $\rR(\mathbf{A}_2)(\mathbb{F}_q)$.

The restriction of scalars group $\mathbf{G}$ of $\mathrm{GL}_n\slash k_0$ exists and is connected for all $n\in\mathbb{N}$. We list some basic properties:
\begin{itemize}
\item The canonical determinant morphism $\det_n:\mathrm{GL}_n\rightarrow\mathbb{G}_m$ gives a morphism $\rR(\det_n): \mathbf{G}\rightarrow\rR(\mathbb{G}_m)$, where $\mathbb{G}_m$ is the multiplicative group over $k_0^{\times}$. To be more precise, we have $\rR(\mathbb{G}_n)(\F_q)\cong k_0^{\times}$. Denote by $\iota:\rR(\mathbf{G})(\F_q)\cong\mathrm{GL}_n(k_0)$, then $\rR(\det_n)(x)=\det_n(\iota(x))$ for $x\in \rR(\mathbf{G})(\F_q)$. The kernel $\mathrm{ker}(\rR(\det_n))$ is the restriction of scalars group of $\mathrm{SL}_n$ from $k_0$.
\item There is a morphism $N_{k_0\slash \F_q}:\rR(\mathbb{G}_m)\rightarrow \mathbb{G}_{m}\slash\F_q$, such that its action on $k_0^{\times}$ is equivalent with norm mapping of $k_0^{\times}\slash\F_q^{\times}$. The kernel $\mathrm{ker}(N_{k_0\slash\F_q})$ is a connected reductive group.
\item There is an isomorphism: $\mathbf{G}(\overline{\F}_q)\cong\mathrm{GL}_n(\overline{\F}_q)^{[k_0:\F_q]}$. The elements in $\mathrm{Gal}(k_0\slash\F_q)$ act on $\mathbf{G}(\overline{\F}_q)$, which is a $1$-cocyle.  Embedding $\mathrm{GL}_n(\overline{\F}_q)^{[k_0:\F_q]}$ as a subgroup diagonal by blocks of $\mathrm{GL}_{n[k_0:\F_q]}(\overline{\F}_q)$, we have
$$N_{k_0\slash\F_q}\circ\rR(\det_n)(x)\cong\det_{n[k_0:\F_q]}(x),$$
for $x\in \mathbf{G}(\overline{\F}_q)$.
\item The Levi subgroups of $\mathbf{G}$ are the restriction of scalars group of Levi subgroups of $\mathrm{GL}_{n}$ to $\F_q$. In particular, the Levi subgroups of $\mathbf{G}(\F_q)$ are Levi subgroups of $\mathrm{GL}_n(k_0)$.

\end{itemize}

Now we consider a finite index set $I$. For each $i\in I$ we fix integers $n_i, e_i\in \mathbb{N}$ and a finite field extension $k_i\slash\F_q$. Let $\mathbf{G}_i$ be the restriction of scalars group of $\mathrm{GL}_{n_i}$ from $k_i$,  and denote by $\mathbf{G}_I=\prod_{i\in I}\mathbf{G}_i$. 
\begin{itemize}
\item We denote by $\iota_i: \mathbf{G}_i(\F_q)\cong\mathrm{GL}_{n_i}(k_i)$ the canonical isomorphism.
\item Define $\det_i:=N_{k_i\slash\F_q}\circ\rR(\det_{n_i})$, where $\det_{n_i}$ the canonical determinant morphism on $\mathrm{GL}_{n_i}$. Hence $\det_i(x)=N_{k_i\slash\F_q}\circ\det_{n_i}(\iota_i(x))$ for $x\in\mathbf{G}_i(\F_q)$.
\item Denote by $\det_{I}^{e}=\otimes_{i\in I}\det_i^{e_i}:\prod_{i\in I}\mathbf{G}_i\rightarrow\mathbb{G}_m$, and by $\mathbf{G}_{I}^{e}=\mathrm{ker}(\det_{I}^{e})$, which is a reductive subgroup of $\mathbf{G}_I$ but \textbf{disconnected} in general (a disconnected algebraic group is \textbf{reductive} if its identity component is reductive as in the usual sense). 
\item Define $\rR(\det)_I:=(\rR(\det_{n_i}))_{i\in I}$. Denote by $\mathbf{G}_I^0=\mathrm{ker}(\rR(\det)_I)$, which is the product of restriction of scalars group of $\mathrm{SL}_{n_i}\slash k_i$, hence is connected. It is obvious that $\mathbf{G}_I^{0}$ is the derived group of $\mathbf{G}_I$, and
$$\mathbf{G}_I^0\subset\mathbf{G}_I^{e}\subset\mathbf{G}_I.$$
\end{itemize}

\begin{defn}
A parabolic subgroup of $\mathbf{G}_{I}^{e}$ is the intersection of a parabolic subgroup of $\mathbf{G}_I$ with $\mathbf{G}_{I}^{e}$. 
\end{defn}


We have the following properties.
\begin{prop}
\label{propLevibij}
\begin{itemize}
\item The Levi subgroups of $\mathbf{G}_I(\F_q)$ are the products of Levi subgroups of $\mathrm{GL}_{m_i}(k_i)$.
\item Intersection with $\mathbf{G}_I^0$ (resp. $\mathbf{G}_I^e$) gives a bijection between parabolic subgroups of $\mathbf{G}_I$ and $\mathbf{G}_I^0$ (resp. $\mathbf{G}_I^e$). In particular, the unipotent radical of a parabolic subgroup of $\mathbf{G}_I$ is the unipotent radical of the corresponding parabolic subgroups of $\mathbf{G}_I^0$ and $\mathbf{G}_I^e$.
\end{itemize}
\end{prop}
\begin{proof}
These statements can be checked by definition.
\end{proof}

Denote by $\mathbb{G}=\mathbf{G}_I(\F_q)$ and by $\mathbb{G}'=\mathbf{G}_I^{e}(\F_q)$, by $\mathbb{G}^0=\mathbf{G}_I^0(\F_q)$.

\begin{cor}
\label{coruniqsupcuspL'}
Let $\rho'$ be an irreducible $k$-representation of $\mathbb{G}'$, then its supercuspidal support is unique up to $\mathbb{G}'$-conjugation.
\end{cor}

\begin{proof}
The quotient $\mathbb{G}\slash\mathbb{G}'$ is a finite abelian group, hence each irreducible representation $\rho'$ of $\mathbb{G}'$ is a direct component of $\rho\vert_{\mathbb{G}'}$ where $\rho$ is irreducible of $\mathbb{G}_I$. We call such $\rho$ a lifting of $\rho'$. Since $\rho\vert_{\mathbb{G}^{0}}$ is multiplicity-free, so is $\rho\vert_{\mathbb{G}'}$. The direct components of $\rho\vert_{\mathbb{G}'}$ are $\mathbb{G}$-conjugate. Let $(\mathbb{M}',\pi')$ be an element in the supercuspidal support of $\rho'$, where $\mathbb{M}'$ is a Levi subgroup, and $\mathbb{M}^{0},\mathbb{M}$ be the Levi subgroups of $\mathbb{G}^{0}$ and $\mathbb{G}$ under the bijection in Proposition \ref{propLevibij}. Then $\rho'$ is a subquotient of $i_{\mathbb{M}'}^{\mathbb{G}'}\pi'$, and we write $\rho'\textless i_{\mathbb{M}'}^{\mathbb{G}'}\pi'$. Let $\rho^{0}$ be an irreducible direct component of $\rho'\vert_{\mathbb{G}^0}$, and $\pi$ be a lifting of $\pi'$ to $\mathbb{M}$. Since $\ind_{\mathbb{G}'}^{\mathbb{G}}i_{\mathbb{M}'}^{\mathbb{G}'}\pi'\cong i_{\mathbb{M}}^{\mathbb{G}}\ind_{\mathbb{M}'}^{\mathbb{M}}\pi'$, we have 
$$\ind_{\mathbb{G}'}^{\mathbb{G}}\rho'\textless i_{\mathbb{M}}^{\mathbb{G}}\ind_{\mathbb{M}'}^{\mathbb{M}}\pi'.$$
The irreducible subquotients of the lefthand side is equivalent to $\rho\otimes\chi$ where $\chi$ is a character. By the uniqueness of supercuspidal support of $\rho\otimes\chi$ and the fact that $\rR(\det_I)(\mathbb{G}')=\rR(\det_I)(\mathbb{M}')$, we deduce that the supercuspidal support of $\rho'$ is contained in $[\mathbb{M}',\pi']_{\mathbb{G}}$ which is the $\mathbb{G}$-conjugacy class of $(\mathbb{M}',\pi')$. Suppose there exists $g\in \mathbb{M}$ such that $g(\pi')\neq \pi'$ and $\rho'\textless i_{\mathbb{M}'}^{\mathbb{G}'}g(\pi')$. By Mackey's theory we have 
$$\rho^{0}\textless \res_{\mathbb{G}^{0}}^{\mathbb{G}'}i_{\mathbb{M}'}^{\mathbb{G}'}\pi'\cong i_{\mathbb{M}^{0}}^{\mathbb{G}^{0}}\res_{\mathbb{M}^{0}}^{\mathbb{M}^{'}}\pi',$$
and
$$\rho^{0}\textless \res_{\mathbb{G}^{0}}^{\mathbb{G}'}i_{\mathbb{M}'}^{\mathbb{G}'}g(\pi')\cong i_{\mathbb{M}^{0}}^{\mathbb{G}^{0}}\res_{\mathbb{M}^{0}}^{\mathbb{M}^{'}}g(\pi').$$
Since the direct components of $\res_{\mathbb{M}^{0}}^{\mathbb{M}^{'}}\pi'$ and $\res_{\mathbb{M}^{0}}^{\mathbb{M}^{'}}g(\pi')$ are never equivalent, which contradicts with the uniqueness of supercuspidal support of $\rho^0$. Hence we conclude that the supercuspidal support of $\rho'$ is  the $\mathbb{G}'$-conjugacy class of $(\mathbb{M}',\pi')$.

\end{proof}

\begin{cor}
\label{cor2.3}
Let $\rho'$ be irreducible and cuspidal of $\mathbb{G}'$, and $\cP_{\rho'}$ its projective cover. Then the supercuspidal support of the irreducible subquotients of $\cP_{\rho'}$ is contained in the $\mathbb{G}$-conjugacy class $[\mathbb{M}',\pi']_{\mathbb{G}}$. 
\end{cor}

\begin{proof}
Let $\mathbb{M}$ be the Levi subgroup of $\mathbb{G}$ as above. Let $\rho$ be a lifting of $\rho'$ to $\mathbb{G}$, of which the supercuspidal support is the $\mathbb{G}$-conjugacy class $[\mathbb{M},\pi]$, where $\pi$ is a lifting of $\pi'$ to $\mathbb{M}$. Denote by $\cP_{\rho}$ the projective cover of $\rho$. The supercuspidal support of an irreducible subquotient of $\cP_{\rho}$ is $[\mathbb{M},\pi]$ as well. Since $\cP_{\rho'}\hookrightarrow \cP_{\rho}\vert_{\mathbb{G}'}$, the supercuspidal support of irreducible subquotients of $\cP_{\rho'}$ are $\mathbb{G}$-conjugate to that of $\rho'$, hence they are contained in $[\mathbb{M}',\pi]'_{\mathbb{G}}$ by Corollary \ref{coruniqsupcuspL'}.
\end{proof}

\subsection{$\ell$-parablocks of $\mathbb{G}'$}

Thanks to the uniqueness of supercuspidal support, we can apply Definition \ref{defn2.5} and study the $\ell$-parablocks of $\mathbb{G}'$. Let $\rho'$ be an irreducible representation of $\mathbb{G}'$, and $\mathcal{B}_{\rho'}$ be the $\ell$-parablock containing $\rho'$.


\begin{prop}
\label{propconjparabl}
Let $\mathcal{B}_1$ and $\mathcal{B}_2$ be two $\ell$-parablocks of $\mathbb{G}'$. Suppose there exist irreducible representations $\rho_i\in\mathcal{B}_i$ for $i=1,2$ that are $\mathbb{G}$-conjugate. Let $\mathcal{P}(\mathcal{B}_1)$ be a projective generator. There exists $x\in\mathbb{M}$ such that $x(\mathcal{P}(\mathcal{B}_1))$ is a projective generator of $\mathcal{B}_2$.
\end{prop}
\begin{proof}
Up to a conjugation in $\mathbb{G}'$, we can assume that $\rho_1$ and $\rho_2$ are conjugate by an element $x\in\mathbb{M}$. Let $(\mathbb{M}',\tau_1)$ be in the supercuspidal support of $\rho_1$. Recall that $\mathbb{M}$ is a Levi of $\mathbb{G}$ as above. We have the following two facts:
\begin{enumerate}
\item Let $\mathcal{P}_{1}$ be a projective cover of $\rho_1$, we have $x(\mathcal{P}_{1})$ is the projective cover of $\rho_2$.
\item $\rho_2\leq i_{\mathbb{M}'}^{\mathbb{G}'}x(\tau_1)$.
\end{enumerate}
We conclude that the conjugation of $x$ sends the $\ell$-block $\mathcal{B}_{1}$ containing $\rho_1$ to the $\ell$-block $\mathcal{B}_2$ containing $\rho_2$, and $x(\mathbf{SC}(\mathcal{B}_1))=\mathbf{SC}(\mathcal{B}_2)$. By the equivalence relation in Definition \ref{defn2.5}, the above analysis implies that the conjugation of $x$ gives a bijection between the $\ell$-blocks equivalent to $\mathcal{B}_1$ to those equivalent to $\mathcal{B}_2$, which gives the result.
\end{proof}




The last example in \cite{C3} shows that the supercuspidal representations in an $\ell$-block of $\mathrm{SL}_2(\F_q)$ can be different, but are conjugate in $\mathrm{GL}_2(\F_q)$. Now we consider a more general setting. We give a description on the supercuspidal pairs appearing in the supercuspidal support of irreducible representations in an $\ell$-parablock.

\begin{prop}
\label{finprop004}
Let $\mathcal{B}$ be an $\ell$-parablock of $\mathbb{G}'$. The supercuspidal supports of irreducible representations in $\mathcal{B}$ are contained in the $\mathbb{G}$-conjugacy class of a supercuspidal pair of $\mathbb{G}'$.
\end{prop}

\begin{proof}
Let $\mathrm{B}$ be an $\ell$-block contained in $\mathcal{B}$. For $\rho'\in\mathrm{Irr}(\mathrm{B})$, let $(\mathbb{M}',\tau')$ be in the cuspidal support of $\rho'$, and $(\mathbb{L}',\sigma')$ in the supercuspidal support of $\tau'$. The projective cover $\mathcal{P}_{\rho'}$ is an indecomposable direct summand of $i_{\mathbb{M}'}^{\mathbb{G}'}\cP_{\tau'}$. By Corollary \ref{coruniqsupcuspL'} and Corollary \ref{cor2.3}, the supercuspidal support of irreducible subquotients of $\cP_{\rho'}$ is contained in $[\mathbb{L}',\sigma']_{\mathbb{G}}$. We deduce from Lemma \ref{finlem001} that the supercuspidal support of an element in $\mathrm{Irr}(\mathrm{B})$ is contained in $[\mathbb{L}',\sigma']_{\mathbb{G}}$. Then the result can be obtained directly from Definition \ref{defn2.5}.

\end{proof}

\section{The group $\rL_{\text{max}}$ and cover theory}
\label{sectionGcover}
Recall that $F$ is a non-archimedean local field, of which $\mathfrak{o}_F$ is the integer ring, $\mathfrak{p}_F$ is the maximal ideal of the integer ring, and $\varpi_{F}$ is a uniformiser. From now on, we denote by $\mathrm{G}$ a Levi subgroup of $\mathrm{GL}_n(F)$, and by $\mathrm{G}'$ the intersection of $\rG$ and $\mathrm{SL}_n(F)$. Without loss of generality, we may assume that $\rG$ contains a maximal split torus of diagonal matrices. We fix a Borel subgroup of upper triangular matrices. A classical way to study maximal simple types (also known as cuspidal types) of $\mathrm{G}'$ is by constructing them from $\mathrm{G}$. In this section, we recall some background knowledgement of type theory of $\mathrm{G}$ in \cite{BuKuI}, \cite{BuKu99}, \cite{MS} and \cite{SS}. Let $\rP$ be a standard parabolic subgroup with Levi subgroup $\rL$. After recalling the structure of a supercuspidal type of $\mathrm{L}$, we introduce homogeneous Levi subgroup $\mathrm{L}_{\text{max}}$, which is bigger than $\mathrm{L}$ in general and is determined by fixing a supercuspidal type. Then we recall the structure of $\rL_{\text{max}}$-cover and $(\rL_{\text{max}},\alpha)$-cover theory.  We then generalise cover theory to $\mathrm{G}'$ in the next section.

\subsection{Maximal simple $k$-types}
\label{subsecmaxsimpG}
A \textbf{maximal simple $k$-type} of $\mathrm{GL}_n(F)$ is a pair $(J,\lambda)$ with technical conditions. In particular, each maximal simple $k$-type is defined from a \textbf{maximal simple stratum} $(\fA,0,\beta)$ for $\beta\in\mathrm{GL}_n(F)$. It gives a field extension $E=F[\beta]$ and a maximal $\mathfrak{o}_{E}$-hereditary order $\mathfrak{B}$ with respect to the same lattice chain in the definition of $\fA$, where $\mathfrak{o}_E$ is the integer ring of $E$. Sometimes, we also denote it as $\mathfrak{B}_{\beta}$. A maximal simple stratum also gives compact open subgroups $J,H$ and $\rU(\fA)$ such that $H\subset J\subset \rU(\fA)$, and it associates an integer $\nu_{\fA}(\beta)$, called the valuation of $\beta$ with respect to $\fA$ (see \cite[1.1.3]{BuKuI}). Let $H^1, J^1$ be the  pro-$p$ radical of $H$ and $J$. A simple character $\theta$ defined from $(\fA,0,\beta)$ is a $k$-character of $H^1$ with special conditions. Denote by $\mathcal{C}(\fA,0,\beta)$ the set of \textbf{simple characters} defined from $(\fA,0,\beta)$, and by $\mathcal{C}_{n}$ the union of simple characters defined from simple stratum of $\mathrm{GL}_n(F)$. For each $\theta$, there is a unique irreducible $k$-representation $\eta$ of $J^1$ of which the restriction to $H^1$ is a multiple of $\theta$, which is called the \textbf{Heisenberg representation} of $\theta$. We know $\eta$ can be extended to $J$, and part of them verify some technical conditions, which we call \textbf{wild-extensions} of $\theta$. Meanwhile, the quotient $J\slash J^1\cong\mathrm{GL}_m(\mathbb{F}_q)$, where $m$ divides $n$ and $\mathbb{F}_q$ is the residual field $k_E$ of $E$, hence is a finite extension of the residual field $k_F$ of $F$. There is a wild-extension $\kappa$ of a simple character $\theta$, such that:
$$\lambda\cong\kappa\otimes\sigma,$$
where $\sigma$ is inflated from an irreducible cuspidal $k$-representation of $\mathrm{GL}_m(\mathbb{F}_q)$.

Now for a standard Levi subgroup $\rL$ of $\rG$, a maximal simple stratum $(\fA_{\rL},0,\beta_{\rL})$ is defined to be $\fA_{\rL}\cong\prod_{i\in I}\fA_i$ and $\beta_{\rL}=(\beta_i)_{i\in I}$ where each $(\fA_i,0,\beta_i)$ is a maximal simple stratum of $\mathrm{GL}_{n_i}(F)$. It gives compact open subgroups $J_{\rL}\cong\prod_{i\in I}J_i, H_{\rL}\cong\prod_{i\in I}H_i$. Denote by $E_{\rL}$ the product $\prod_{i\in I}E_i$ where $E_i=F[\beta_i]$, and $\mathfrak{B}_{\rL}:=\prod_{i\in I}\mathfrak{B}_{i}$ where $\mathfrak{B}_i$ is the associated maximal $\mathfrak{o}_{E_i}$-hereditary order. We also call $\mathfrak{B}_{\rL}$ an $\mathfrak{o}_{E_{\rL}}$-hereditary order. A simple character $\theta_{\rL}\cong\prod_{i\in I}\theta_i$, where $\theta_i\in\mathcal{C}(\fA_i,0,\beta_i)$. It gives compact open subgroups $J_{\rL}\cong\prod_{i\in I}J_i,\lambda_{\rL}\cong\otimes_{i\in I}\lambda_i$. Then for a $\theta_{\rL}$, we define $\eta_{\rL}\cong\prod_i\eta_i$ and we call it the \textbf{Heisenberg representation} of $\theta_{\rL}$. We call an extension $\kappa_{\rL}$ a \textbf{wild-extension} (it is also called a $\beta$-extension in \cite{BuKuI}) of $\theta_{\rL}$ if $\kappa_{\rL}\cong\prod_i\kappa_i$ where each $\kappa_i$ is a wild-extension of $\eta_i$. We call $(J_{\rL},\kappa_{\rL})$ a \textbf{wild pair} of $\rL$. Denote by $\mathbb{L}$ the quotient $J_{\rL}\slash J_{\rL}^1$ where $J_{\rL}^1$ is the pro-$p$ radical, which is equivalent to $\prod_i\mathrm{GL}_{m_i}(\mathbb{F}_{q_i})$ and a cuspidal representation of $\mathbb{L}$ is a tensor product of cuspidals of $\mathrm{GL}_{m_i}(\mathbb{F}_{q_i})$ for each $i$. A maximal simple $k$-types of $\rL$ is a pair $(J_{\rL},\lambda_{\rL})$ such that 
$$\lambda_{\rL}\cong\kappa_{\rL}\otimes\sigma_{\rL},$$
where $(J_{\rL},\kappa_{\rL})$ is a wild pair and $\sigma_{\rL}$ is inflated from a cuspidal of $\mathbb{L}$. We also call a maximal simple $k$-type a \textbf{cuspidal $k$-type}. Notice that in $\ell$-modular setting, cuspidality is not equivalent to supercuspidality. We call $(J_{\rL},\lambda_{\rL})$ a \textbf{supercuspidal $k$-type} when $\sigma_{\rL}$ is supercuspidal. Denote by $[J_{\rL},\lambda_{\rL}]$ its $\mathrm{G}$-conjugacy class. 

A maximal simple $k$-type of $\rL'$ is defined from one of $\rL$. Let 
$$\tilde{J}_{\rL}:=\{g\in\rU(\fA_{\rL}), g(J_{\rL})=J_{\rL},g(\lambda_{\rL})\cong\lambda_{\rL}\otimes\chi\circ\det, \textit{ where }\chi \textit{ is a }k\textit{-quasicharacter of }F^{\times}\}.$$
It is the group of projective normaliser of $(J_{\rL},\lambda_{\rL})$ defined in \cite{BuKuII} (also in \cite{C1} for modular setting), which is open compact containing $J_{\rL}$. The induced representation $\tilde{\lambda}_{\rL}:=\ind_{J_{\rL}}^{\tilde{J}_{\rL}}\lambda_{\rL}$ is irreducible. Now let $\tilde{J}_{\rL}'=\tilde{J}_{\rL}\cap\rL'$ and $\tilde{\lambda}_{\rL}'$ be an irreducible direct component of the semisimple representation $\tilde{\lambda}_{\rL}\vert_{\tilde{J}_{\rL}'}$. A pair in the form of $(\tilde{J}_{\rL}',\tilde{\lambda}_{\rL}')$ is a \textbf{maximal simple $k$-type} or equivalently a \textbf{cuspidal $k$-type} of $\rL'$. We call it a \textbf{supercuspidal $k$-type} when $\sigma_{\rL}$ is supercuspidal. Write $[\tilde{J}_{\rL}',\tilde{\lambda}_{\rL}']$ its $\rG'$-conjugacy class and $[\tilde{J}_{\rL}',\tilde{\lambda}_{\rL}']_{\rG}$ its $\rG$-conjugacy class. Denote by $\mathcal{ST}_{\rG'}$ the set of $[\tilde{J}_{\rL}',\tilde{\lambda}_{\rL}']$.

For $\rL'=\rL\cap\rG'$, let $(\rL',\tau')$ be a cuspidal pair of $\rG'$, and $[\rL',\tau']$ its cuspidal class of $\rG'$. By \cite{C1}, there exists a cuspidal $k$-type $(\tilde{J}_{\rL}',\tilde{\lambda}_{\rL}')$ such that the restriction of $\tau'\vert_{\tilde{J}'}$ contains $\tilde{\lambda}_{\rL}'$ as a sub-quotient, which gives a mapping from $[\rL',\tau']$ to $[\tilde{J}_{\rL}',\tilde{\lambda}_{\rL}']$. Let $\mathcal{SC}_{\rG'}$ be the set of supercuspidal classes of $\rG'$. The above mapping is a bijection between $\mathcal{SC}_{\rG'}$ and $\mathcal{ST}_{\rG'}$.

\begin{lem}
\label{lem 001}
Let $(\tilde{J}_{\rL}',\tilde{\lambda}_{\rL}')$ be a cuspidal $k$-type of $\rL'$, then the supercuspidal supports of irreducible subquotients of $\ind_{\tilde{J}_{\rL}'}^{\rL'}\tilde{\lambda}_{\rL}'$ are in the same supercuspidal class of $\rL'$.
\end{lem}

\subsection{$\rL_{\text{max}}$-covers and $(\rL_{\text{max}},\alpha)$-covers}
\label{Lmaxcovers}
Let $(J_{\rL},\lambda_{\rL})$ a cuspidal type of $\rL$. We introduce the homogeneous Levi subgroup $\rL_{\text{max}}$ determined by $\theta_{\rL}$, which is the simple character contained in $\kappa_{\rL}$. As we have explained, the wild extension of $\theta_{\rL}$ is not unique, so neither is the decomposition $\lambda_{\rL}\cong\kappa_{\rL}\otimes\sigma_{\rL}$. In this section, we determine a good choice of $\kappa_{\rL}$. We start from the endo-class determined by $\theta_{\rL}$ (see \cite{BH} or \cite{BuKuI} for endo-class of general simple characters), which gives a wild pair $(J_{\text{max}},\kappa_{\text{max}})$ on $\rL_{\text{max}}$. There is a unique wild extension of $\theta_{\rL}$ that verifies a compatibility property with $\kappa_{\text{max}}$. With this choice of $\kappa_{\rL}$, there exist two pairs of open compact subgroups and their irreducible representations. One is an $\rL_{\text{max}}$-cover of $(J_{\rL},\lambda_{\rL})$, the other verifies an induction equivalence (Equation \ref{alphaindequa}), and we call it an $(\rL_{\text{max}},\alpha)$-cover. 

\begin{rem}[Warning] We use different notations comparing to \cite[\S 5]{MS}. For an example, our pair $(J_{\rL},\kappa_{\rL})$ corresponds to their $(J_{\max,\alpha},\kappa_{\max,\alpha})$. We will see in Section \ref{sectionK'Levi} that notation system here is more convenient for our use. 
\end{rem}

A $\rG$-cover of $(J_{\rL},\lambda_{\rL})$ is a pair $(J_{\rP},\lambda_{\rP})$ where $J_{\rP}$ is open compact in $\rG$ and $\lambda_{\rP}$ is an irreducible representation of $J_{\rP}$, that verify the following conditions (see the last corollary in \cite{Blon05}):
\begin{crit}[for $\rG$-covers]
\label{critcover}
Let $\rN$ be the unipotent radical of $\rP$ and $\bar{\rN}$ the unipotent radical of the opposite parabolic with Levi subgroup $\rL$. 
\begin{enumerate}
\item $J_{\rP}\cap\rL=J_{\rL}$ and $J_{\rP}=J_{\rP}\cap\bar{\rN}\cdot J_{\rL}\cdot J_{\rP}\cap\rN$.
\item $\lambda_{\rP}\vert_{J_{\rL}}\cong\lambda_{\rL}$, and both of $\lambda_{\rP}\vert_{J_{\rP}\cap\rN},\lambda_{\rP}\vert_{J_{\rP}\cap\bar{\rN}}$ are trivial. 
\item Denote by $r_{\rL}^{\rG}$ the parabolic restriction with respect to $\rP$. For any irreducible representation $\pi$ of $\rG$, the mapping
$$\Hom_{J_{\rP}}(\lambda_{\rP},\pi)\rightarrow \Hom_{J_{\rL}}(\lambda_{\rL},r_{\rL}^{\rG}(\pi)),$$
sends $f\in\Hom_{J_{\rP}}(\lambda_{\rP},\pi)$ to $r_{\rL}^{\rG}\circ f$ is injective.
\item For every irreducible representation $\pi$ of $\rG$, the composition with $r_{\rL}^{\rG}$ gives an injective mapping from the isotypic part $(\pi)^{\lambda_{\rP}}$ to $(r_{\rL}^{\rG}\pi)^{\lambda}$.
\end{enumerate}
\end{crit}

\begin{rem}
\begin{itemize}
\item When $(J_{\rP},\lambda_{\rP})$ is a cover, we have
$$\ind_{J_{\rP}}^{\rG}\lambda_{\rP}\cong i_{\rL}^{\rG}\ind_{J_{\rL}}^{\rL}\lambda_{\rL}.$$
\item The above criterion is also valid to check $\rG'$-cover of cuspidal types of $\rL'$, via replacing $(J_{\rL},\lambda_{\rL})$ by $(\tilde{J}_{\rL}',\tilde{\lambda}_{\rL}')$, and considering homorphism set with irreducible representations of $\rG'$.
\end{itemize}
\end{rem}

Denote by $\mathcal{C}_{\mathrm{GL}}$ the union $\bigcup_{n\in \mathbb{N}}\mathcal{C}_{n}$. \textbf{Endo-equivalence} is an equivalent relation defined on $\mathcal{C}_{\mathrm{GL}}$, and we call a connected component under this relation an endo-equivalent class (see \cite{BH96} for definiton). In other words, it defines an equivalence between simple characters of $\mathrm{GL}$-group of different ranks. Suppose $(J_{\rL},\lambda_{\rL})$ defined from a maximal simple stratum $(\fA_{\rL},0,\beta_{\rL})$ on $\rL$ and a simple character $\theta_{\rL}\cong\prod_i\theta_i$. We can give a partition $I=\cup_{s\in S}I_s$ with respect to endo-classes. In particular, $([\fA_i,0,\beta_i],\theta_i),([\fA_j,0,\beta_j],\theta_j)$ are in the same endo-class if and only if $i,j$ are in the same part under the partition. We may assume that $\beta_i,i\in I_s$ are all equal to a single element $\beta_s$. In particular, $E_i$'s are identical among $i\in I_s$. From now on we always simplify $\beta_{\rL}$ as $\beta$. 

Denote by $\rL_{\text{max}}$ the standard Levi subgroup $\prod_{s\in S}\mathrm{GL}_{n_s}(F)$ where $n_s=\sum_{i\in I_s}n_i$, and $\rP=\rL_{\text{max}}\rN$ be the standard parabolic subgroup with Levi $\rL_{\text{max}}$. With the notation in \cite{MS}, $\rL_{max}$ is the maximal standard Levi subgroup such that $(J_{\rL},\lambda_{\rL})$ is  ``homogène" in $\rL_{\text{max}}$. In this work, we call $\rL_{\text{max}}$ the \textbf{homogeneous Levi subgroup} of the pair $([\fA_{\rL},0,\beta_{\rL}],\theta_{\rL})$. The following properties are studied in \cite{BuKu99} (see Common approximation). On $\rL_{\text{max}}$, there are two simple strata $[\fA,0,\beta]$ and $[\fA_{max},0,\beta]$, where the latter is maximal and we have $\rU(\fA)\subset\rU(\fA_{max})$ (see more details in \cite[\S 5.1]{SS}). Write $\beta=(\beta_s)_{s\in S}$, then $\fA\cong\prod_{s\in S}\fA_s$ and $\fA_{max}\cong\prod_{s\in S}\fA_{max,s}$. They give two families of open compact subgroups $\{H,H^1,J,J^1\}$,  and $\{H_{max},H_{max}^1,J_{max},J_{max}^1\}$ respectively. Denote by $\mathfrak{B}_{\rL}$ the $\mathfrak{o}_{E_{\rL}}$-hereditary order defined from $(\fA_{\rL},0,\beta)$ (resp. $(\fA,0,\beta)$), and by $\mathfrak{B}$ (resp. $\mathfrak{B}_{max}$) the $\mathfrak{o}_{E_{\rL}}$-hereditary order defined from $(\fA,0,\beta)$ (resp. $(\fA_{max},0,\beta)$). In particular, $\fA_{max}$ being maximal is equivalent to $\fB_{max}$ being maximal, which means that by writing $\fB_{max}\cong \prod_{s\in S}\fB_{max,s}$ we have $\rU(\fB_{max,s})\slash\rU^1(\fB_{max,s})\cong\mathrm{GL}_{m_s}(k_{E_i})$ where $m_s=n_s\slash[E_i:F]$ for each $i\in I_s$. We list the following properties and definitions:
\begin{enumerate}
\item There is a bijection between any two of the set $\mathcal{C}(\fA_{s},0,\beta_s),\mathcal{C}(\fA_i,0,\beta_i)$ and $\mathcal{C}(\fA_{max,s},0,\beta_s)$ for $i\in I_s$. These bijections determine the unique simple character in $\mathcal{C}(\fA_{s},0,\beta_s)$ and $\mathcal{C}(\fA_{max,s},0,\beta_s)$ that endo-equivalent to $\theta_i$ for $i\in I_s$. Let $\theta_{max,s}$ and $\theta_s$ be the image under the above bijections (or in the language of \cite{SS} we say $\theta_{max,s},\theta_s$ are transferred from $\theta_i$). In particular, $\theta_s$ is the common approximation of $\otimes_{i\in I_s}\theta_{i}$ (see Main Theorem of \cite{BuKu99}). Define $\theta=\otimes_{s\in S}\theta_s$, and $\theta_{max}=\otimes_{s\in S}\theta_{max,s}$.
\item Let $J_{\rP}:=(H^1\cap\bar{\rN})(J\cap\rL)(J^1\cap\rN)=(J\cap\rP)H^1$, where $J_{\rP}\cap\rL=J_{\rL}$. There is an equivalence $J_{\rP}\slash J_{\rP}^1\cong J_{\rL}\slash J_{\rL}^1$, where $(\cdot)^1$ denotes the pro-$p$ radical.
\item Define $J_{max,\alpha}:=U(\mathfrak{B})J_{\text{max}}^1$.
\end{enumerate}

Now we determine a wild extension for each of $\theta_{\rL},\theta_{max},\theta$ in the following manner. We start by fixing a wild-extension $\kappa_{\text{max}}$ of $\theta_{\text{max}}$. There is a unique wild-extension $\kappa$ of $\theta$ such that
\begin{equation}
\label{equakappaG}
\ind_{J}^{U(\mathfrak{B})U^1(\mathfrak{A})}\kappa\cong\ind_{U(\fB)J_{\text{max}}^1}^{U(\fB)U^1(\fA)}\kappa_{\text{max}}.
\end{equation}
Notice that $U(\fB_{\rL})J_{\text{max}}^1\subset J_{\text{max}}$. There is a unique wild-extension $\kappa_{\rL}$ of $\theta_{\rL}$ such that after extending to $J_{\rP}$ via acting trivially on $H_{\rP}^1\cap\bar{\rN}$ and $J_{\rP}^1\cap\rN$, we have
$$\ind_{J_{\rP}}^{J}\kappa_{\rL}\cong\kappa.$$
Denote by $\kappa_{\rP}$ the above extension of $\kappa_{\rL}$. We decompose $\lambda$ with respect to this choice of $\kappa_{\rL}$:
$$\lambda_{\rL}\cong\kappa_{\rL}\otimes\sigma_{\rL}.$$
Denote by $\lambda_{\rP}$ a representation of $J_{\rP}$ that extends $\lambda_{\rL}$ trivially to $H_{\rP}^1\cap\bar{\rN}$ and $J_{\rP}^1\cap\rN$. We have 
$$\lambda_{\rP}\cong\kappa_{\rP}\otimes\sigma_{\rL}.$$
The pair $(J_{\rP},\lambda_{\rP})$ is a \textbf{$\rL_{\text{max}}$-cover}  of $(J_{\rL},\lambda_{\rL})$:

Now we introduce $(J_{max,\alpha})$-cover. We have $J_{max,\alpha}\slash J_{max,\alpha}^1\cong U(\fB_{\rL})\slash U^1(\fB_{\rL})\cong J_{\rL}\slash J_{\rL}^1$. Denote by $\kappa_{max,\alpha}$ the restriction $\kappa_{\text{max}}\vert_{J_{max,\alpha}}$. Put $\lambda_{max,\alpha}=\kappa_{max,\alpha}\otimes\sigma_{\rL}$, and $\lambda_{\rP}=\kappa_{\rP}\otimes\sigma_{\rL}$. Then we deduce from equivalences
$$\ind_{J_{\rP}}^{U(\fB)U^1(\fA)}\lambda_{\rP}\cong (\ind_{J_{\rP}}^{U(\fB)U^1(\fA)}\kappa_{\rP})\otimes\sigma_{\rL},$$
and
$$\ind_{U(\fB)J_{\text{max}}^1}^{U(\fB)U^1(\fA)}\lambda_{max,\alpha}\cong(\ind_{U(\fB)J_{\text{max}}^1}^{U(\fB)U^1(\fA)}\kappa_{max,\alpha})\otimes\sigma_{\rL},$$
that
\begin{equation}
\label{equaGalpha}
\ind_{J_{\rP}}^{U(\fB)U^1(\fA)}\lambda_{\rP}\cong\ind_{U(\fB)J_{\text{max}}^1}^{U(\fB)U^1(\fA)}\lambda_{max,\alpha}.
\end{equation}
Meanwhile, we have
\begin{equation}
\label{alphaindequa}
\ind_{J_{max,\alpha}}^{J_{\text{max}}}\lambda_{max,\alpha}\cong\kappa_{\text{max}}\otimes\ind_{J_{max,\alpha}}^{J_{\text{max}}}\sigma_{\rL},
\end{equation}
and
$$\ind_{J_{max,\alpha}}^{J_{\text{max}}}\sigma_{\rL}\cong i_{\mathbb{L}}^{\mathbb{L}_{\text{max}}}\sigma_{\rL},$$
where $\mathbb{L}\cong J_{\rL}\slash J_{\rL}^1$ and $\mathbb{L}_{\text{max}}\cong J_{\text{max}}\slash J_{\text{max}}^1$. The pair $(J_{max,\alpha},\lambda_{max,\alpha})$ has been studied in \cite{SS} and \cite{BuKuI}, but not named. Thanks to Equation \ref{equaGalpha} and \ref{alphaindequa}, this pair will be mentioned repeatedly in this work, so we call it an \textbf{$(\rL_{\text{max}},\alpha)$-cover} of $(J_{\rL},\lambda_{\rL})$.

\begin{rem}
\label{remtypesupcuspsupp}
\begin{itemize}
\item We have
$$i_{\rL_{\text{max}}}^{\rG}\ind_{J_{\rP,\alpha}}^{\rL_{\text{max}}}\lambda_{max,\alpha}\cong i_{\rL_{\text{max}}}^{\rG}\ind_{J_{\rP}}^{\rL_{\text{max}}}\lambda_{\rP}\cong i_{\rL}^{\rG}\ind_{J_{\rL}}^{\rL}\lambda_{\rL};$$
\item Suppose that $\sigma_{\rL}$ is supercuspidal. Let $\rho$ be an irreducible representation of $\mathbb{L}_{\text{max}}$ of which the supercuspidal support contains $(\mathbb{L},\sigma_{\rL})$. Let $\kappa_{\text{max}},\kappa_{\rL}$ be as above. The first part and Equation \ref{alphaindequa} imply that any element in the supercuspidal support of an irreducible subquotient $\pi$ of $i_{\rL_{\text{max}}}^{\rG}\ind_{J_{\text{max}}}^{\rL_{\text{max}}}\kappa_{\text{max}}\otimes\rho$ must contain $(J_{\rL},\lambda_{\rL})$. 
\begin{itemize}
\item This compatibility between the supercuspidal support of $\pi$ and $\rho$ can be regarded as a kind of maintains of $(\mathbb{L},\sigma_{\rL})$ after induction to $\rG$, which strictly relies on the choice of $\kappa_{\text{max}},\kappa_{\rL}$.
\item Denote by $\gamma:=\kappa_{\text{max}}\otimes\rho$. Thanks to the above, we say the $\rL_{\text{max}}$-conjugacy class $[J_{\rL},\lambda_{\rL}]$ is \textbf{the supercuspidal support} of $(J_{\text{max}},\gamma)$.
\end{itemize}
\end{itemize}

\end{rem}

\section{$\rL_{\text{max}}'$-covers of cuspidal types}
\label{sectionG'covers}
From now on, we assume that \textbf{$p$ verifies the tameness condition}, that \textbf{$p$ does not divide the order of Weyl group} $\vert W_{\rG}\vert$ of $\rG$ ($W_{\rG'}=W_{\rG}$). Let $(J_{\rL},\lambda_{\rL})$ be a cuspidal $k$-type of $\rL$, with $\lambda_{\rL}\cong\kappa_{\rL}\otimes\sigma_{\rL}$. We introduce the cuspidal types of $\rL'$ coming from $(J_{\rL},\lambda_{\rL})$, and we establish $\rL_{\text{max}}'$-cover and $(\rL_{\text{max}}',\alpha)$-cover of them and generalise the properties of the above section.

\subsection{The condition: $p$ does not divide $\vert W_{\rG'} \vert$}
Under tameness condition, we know from \cite{C1} that $\tilde{J}_{\rL}=J_{\rL}$ and for any irreducible direct component $\lambda_{\rL}'$ of $\lambda_{\rL}\vert_{J_{\rL}'}$, the pair $(J_{\rL}',\lambda_{\rL}')$ is a cuspidal $k$-type of $\rL'$. It shows that this condition greatly simplifies the structure of cuspidal types of $\rL'$, making them more directly related to those of $\rL$, so that we can utilise some results of the latter. This is the reason we require it in this work.

For two pairs $(K_j,\rho_j),j=1,2$ that consist of compact open subgroups $K_j$ and their irreducible representations $\rho_j$, we say $(K_1,\rho_1)$ is \textbf{weakly intertwined} with $(K_2,\rho_2)$, if $\rho_1$ is a subquotient of $\res_{K_1}^{\rL'}\ind_{K_2}^{\rL'}\rho_2$. Since $\ell$=modular representations of a compact group is not always semisimple, the relation of weakly intertwining is different from intertwining in the usual sense. This relation has been firstly studied in \cite{V1}, and also in \cite{C1} for cuspidal types of $\rL'$.

Recall that for a subgroup $K$ of $\rG$, we always denote by $K'$ the intersection $K\cap\rG'$. We denote by $\kappa_{\rL}'$ the irreducible restriction $\kappa_{\rL}\vert_{J_{\rL}'}$. Hence $\lambda_{\rL}'\cong\kappa_{\rL}'\otimes\sigma_{\rL}'$ where $\sigma_{\rL}'$ is an irreducible direct component of $\sigma_{\rL}\vert_{J_{\rL}'}$.

\begin{prop}
\label{proptamecond}
Define $N_{\lambda_{\rL}'}:=\{u\in\rU(\mathfrak{B}_{\rL})\vert \det(u)\in\det(E_{\rL}^{\times})$. Then the set of direct components of $\lambda_{\rL}\vert_{J_{\rL}'}$ which are weakly intertwined with $\lambda_{\rL}'$ is equal to the $N_{\lambda_{\rL}'}$-conjugacy class of $\lambda_{\rL}'$. In particular, let $S_{\lambda_{\rL}'}$ be the subset of $N_{\lambda_{\rL}'}$ which stabilises $\lambda_{\rL}'$, then $S_{\lambda_{\rL}'}$ is independent of the choice of direct component of $\lambda_{\rL}\vert_{J_{\rL}'}$. 
\end{prop}

\begin{proof}


The direct components of $\lambda_{\rL}\vert_{J_{\rL}'}$ that are weakly intertwined with $\lambda_{\rL}'$ belong to the $(E_{\rL}^{\times}J_{\rL})'$-conjugacy class of $\lambda_{\rL}'$. Recall that $E_{\rL}\cong\prod_i E_i$ (see Section \ref{subsecmaxsimpG}). The index $[E_i:F]$ is coprime to $p$ under tameness condition for all $i\in I$. Let $N_{E_i\slash F}$ be the norm map. We have $N_{E_i\slash F}(1+\mathfrak{p}_{E_i})=1+\mathfrak{p}_{F}$. Hence $\det(J_{\rL})=\prod_{i\in I}N_{E_i\slash F}(\mathfrak{o}_{E_i})$, and $(E_{\rL}^{\times}J_{\rL})'\slash J_{\rL}'$ is a quotient group of $(E_{\rL}^{\times}\rU(\mathfrak{B}_{\rL}))'$. Let $\lambda_0'$ be a direct component of $\lambda_{\rL}\vert_{J_{\rL}'}$. Suppose that it is weakly intertwined with $\lambda_{\rL}'$ in $\rL'$. Since $J_{\rL}'$ normalises $\lambda_0'$, we can take $x\in(E_{\rL}^{\times}\rU(\mathfrak{B}_{\rL}))'$ such that $x(\lambda_0')\cong\lambda_{\rL}'$. Write 
$$\lambda_0'\cong\kappa_{\rL}'\otimes\sigma_0',$$
Since $x$ normalises $\kappa_{\rL}'$, we have
$$x(\sigma_0')\cong\sigma_{\rL}'.$$
We can write $x=\alpha u$ where $\alpha\in E_{\rL}^{\times}$ and $u\in\rU(\mathfrak{B}_{\rL})$. Since $\sigma_0'$ is inflated from $\rU(\mathfrak{B}_{\rL})'\slash\rU^1(\mathfrak{B}_{\rL})'$ and $E_{\rL}^{\times}$ commute with elements in $\rU(\mathfrak{B}_{\rL})'$ hence $u(\sigma_0')\cong\sigma_{\rL}'$ and $u(\lambda_0')\cong\lambda_{\rL}'$. We conclude that $\lambda_0'$ is weakly intertwined with $\lambda_{\rL}'$ if and only if it is conjugate to $\lambda_{\rL}'$ by an element in $\rU(\mathfrak{B}_{\rL})$ of which the determinant belongs to $\det(E_{\rL}^{\times})$, as desired.
\end{proof}

\begin{rem}
\label{defnofT}
We denote by $T=N_{\lambda_{\rL}'}\slash S_{\lambda_{\rL}'}$. For each coset we fix a representative $t$, and denote by $\lambda_t'$ the conjugation $t(\lambda_{\rL}')$. Denote  by $\sigma_t'=t(\sigma_{\rL}')$, we can write $\lambda_t'$ as $\kappa_{\rL}'\otimes\sigma_t'$.
\end{rem}

\subsubsection{The finite reductive quotient}

Recall that $\mathfrak{B}_{\rL}\cong\prod_{i\in I}\mathfrak{B}_i$ where $\mathfrak{B}_i$ is the $E_i$-hereditary order with the same lattice chain of $\fA_i$. The group $\mathbb{L}\cong\prod_{i\in I}\mathrm{GL}_{m_i}(k_{E_i})$, where $m_i[E_i:F]=n_i$ (hence $\sum_{i\in I}m_i[E_i:F]=n$), and $[E_i:F]=e_if_i$ where $e_i$ is the ramification index and $f_i=[k_{E_i}:k_{F}]$ is the index of residual field extension. Hence by embedding to $\mathrm{GL}_{m_if_i}(k_{F})$, we regard an element in $\mathrm{GL}_{m_i}(k_{E_i})$ as a matrix with coefficients in $\mathrm{GL}_{f_i}(k_F)$. Let $N_{E_i\slash F}$ be the norm map from $E_i^{\times}$ to $F^{\times}$, and $\det_{(\cdot)}$ the determinant function of matrices with coefficients in $(\cdot)$. Since $N_{E_i\slash F}(1+\mathfrak{p}_{E_i})=1+\mathfrak{p}_{F}$, we have 
$$\det_F(\rU^1(\mathfrak{B}_{\rL})=\prod_{i\in I}N_{E_i\slash F}(1+\mathfrak{p}_{E_i})=1+\mathfrak{p}_F,$$
which implies
\begin{itemize}
\item the quotient map of $\det_F$ on $J_{\rL}\slash J_{\rL}^1$ is equal to $\prod_{i\in I}(\det_{k_{F}})^{e_i}=(N_{k_{E_i}\slash k_F}\circ\det_{k_{E_i}})^{e_i}$;
\item $\mathbb{L}':=J_{\rL}'\slash J_{\rL}^{1'}$ is the subgroup of $\mathbb{L}$, and $\mathbb{L}'=\mathrm{ker}(\prod_{i\in I}\det_{k_F}(x_i)^{e_i})$.
\end{itemize}

By putting $k_i=k_{E_i},\F_q=k_F$, the group $\mathbb{L}'$ identifies with $\mathbb{G}'$ and $\mathbb{L}$ with $\mathbb{G}$ in Section \ref{RSG}. Hence $\mathbb{L}'$ is the $k_F$-rational points of a reductive group over $k_F$, which is disconnected in general.

\subsubsection{$\rL_{\text{max}}'$-cover of supercuspidal $k$-types}
\label{sectionG'cover}
Let $(J_{\rL},\lambda_{\rL})$ and $(J_{\rL}',\lambda_{\rL}')$ be as above. Let $\rP=\rL_{\text{max}}\rN$ be a standard parabolic subgroup, where $\rL_{max}$ is the homogeneous Levi subgroup, and $(J_{\rP},\lambda_{\rP})$ a $\rL_{\text{max}}$-cover as in Section \ref{Lmaxcovers}. In this section, we give an $\rL_{\text{max}}'$-cover of $(J_{\rL}',\lambda_{\rL}')$. Now taking $J_{\rP}'= J_{\rP}\cap\rL_{\text{max}}'$, by definition (see Section \ref{sectionGcover}) we have
$$J_{\rP}'=(H^1\cap\bar{\rN})(J\cap\rL')(J^1\cap\rN)=(J\cap\rP')H^{1'}.$$
Denote by $\kappa_{\rP}'$ (resp. $\kappa_{max}'$) the irreducible restriction $\kappa_{\rP}\vert_{J_{\rP}'}$ (resp. $\kappa_{max}\vert_{J_{max}'}$). Let $\lambda_{\rP}'$ be the direct component of $\lambda_{\rP}\vert_{J_{\rP}'}$, such that $\lambda_{\rP}'\cong\kappa_{\rP}'\otimes\sigma_{\rL}'$.


\begin{thm}
\label{G'cover}
The pair $(J_{\rP}',\lambda_{\rP}')$ is an $\rL_{\text{max}}'$-cover of $(J_{\rL}',\lambda_{\rL}')$.
\end{thm}

\begin{proof}
We only need to check Condition 3 and 4 in Criterion \ref{critcover}. For Condition 3: Let $\pi'$ be irreducible of $\rL'$, and $\pi$ an irreducible of $\rL$ such that $\pi'\hookrightarrow\pi\vert_{\rL'}$. Let $O$ be the kernel of $\lambda_{\rP}$. By Frobenius reciprocity, we have
\begin{equation}
\label{equaembedding}
\mathrm{Hom}_{J_{\rP}'}(\lambda_{\rP}',\pi')\hookrightarrow\mathrm{Hom}_{J_{\rP}}(\lambda_{\rP},\pi\otimes\ind_{J_{\rL}'O}^{J_{\rL}} \mathds{1}).
\end{equation}
Let $\det_{\ell}$ (resp. $\det_{\ell'}$) be the $\ell$-part (resp. $\ell$-prime part) of the finite abelian group $\det(J_{\rP})\slash\det(O)$, which consists of elements with an $\ell$-power order (resp. with an order prime to $\ell$). The latter is equivalent to $\oplus_{\chi\in(\det_{\ell'})^{\vee}}\Hom_{J_{\rP}}(\lambda_{\rP}\otimes\chi,\pi\otimes\ind_{\{1\}}^{\det_{\ell}}\mathds{1}).$ By a similar manner there is an injection from $\Hom_{J_{\rL}'}(\lambda_{\rL}',\pi'_{\rN})$, where $\pi_{\rN}'=r_{\rL'}^{\rL_{max}'}(\pi')$. We have:

\begin{tikzpicture}
\node at(-1,0) {$\Hom_{J_{\rP}'}(\lambda_{\rP}',\pi')$};
\node at(7,0) {$\Hom_{J_{\rL}'}(\lambda_{\rL}',\pi_{\rN}')$};
\node at(-1,-2) {$\oplus_{\chi\in(\det_{\ell'})^{\vee}}\Hom_{J_{\rP}}(\lambda_{\rP}\otimes\chi,\pi\otimes\ind_{\{1\}}^{\det_{\ell}}\mathds{1})$};
\node at(7,-2) {$\oplus_{\chi\in(\det_{\ell'})^{\vee}}\Hom_{J_{\rL}}(\lambda_{\rL}\otimes\chi,\pi_{\rN}\otimes\ind_{\{1\}}^{\det_{\ell}}\mathds{1})$};
\node at(3,0.5) {$r_{\rL'}^{\rL_{max}'}\circ\{\cdot\}$};
\node at(3,-1.5) {$r_{\rL}^{\rL_{max}}\circ\{\cdot\}$};
\draw[->] (1,0)--(5,0);
\draw[->] (2.3,-2)--(3.7,-2);
\draw[->] (-1,-0.5)--(-1,-1.5);
\draw[->] (7,-0.5)--(7,-1.5);
\end{tikzpicture}

As we explained above, the two vertical arrows are injective. To show the upper horizontal arrow is injective, it is sufficient to show that $r_{\rL}^{\rL_{max}}\circ\{\cdot\}$ is injective. Now we assume that $\chi$ is trivial. Let $V$ be the representation space of $\pi$. There is a filtration:
$$V:=V_0\subset V_1\subset\cdots\subset V_m := V\otimes\ind_{\{1\}}^{\det_{\ell}}\mathds{1},$$
such that $V_j\slash V_{j-1}\cong V$ for $j=1,\cdots,m$ and  $m\in\mathbb{N}$. Therefore, by applying $r_{\rL}^{\rL_{max}}$ it gives a filtration:
\begin{equation}
\label{equarNV}
r_{\rL}^{\rL_{max}}(V):=r_{\rL}^{\rL_{max}}(V_0)\subset r_{\rL}^{\rL_{max}}(V_1)\subset\cdots\subset r_{\rL}^{\rL_{max}}(V_m) := r_{\rL}^{\rL_{max}}(V)\otimes\ind_{\{1\}}^{\det_{\ell}}\mathds{1},
\end{equation}
where $r_{\rL}^{\rL_{max}}(V_j)\slash r_{\rL}^{\rL_{max}}(V_{j-1})\cong r_{\rL}^{\rL_{max}}(V)$ for each $j$.

For a non-trivial morphism $f\in\Hom_{J_{\rP}}(\lambda_{\rP},\pi\otimes\ind_{\{1\}}^{\det_{\ell}}\mathds{1})$. Since $\lambda_{\rP}$ is irreducible, there exists a unique $j$ such that 
$$V_j\cap f(\lambda_{\rP})\neq\{0\}.$$ 
Denote it by $j_f$, we have $V_{j_f}\cap f(\lambda_{\rP})=f(\lambda_{\rP})$.
Since $\lambda_{\rP}\otimes\chi$ is a $\rL_{\text{max}}$-cover of $\lambda_{\rL}\otimes\chi$, hence 
$$\Hom_{J_{\rP}}(\lambda_{\rP}\otimes\chi,\pi)\rightarrow\Hom_{J_{\rL}}(\lambda_{\rL}\otimes\chi,r_{\rL}^{\rL_{max}}(\pi)),$$
is injective, which implies that $r_{\rL}^{\rL_{max}}(\pi)\neq 0$. Hence the filtration \ref{equarNV} is non-trivial has length $m$. We deduce furthermore that $j_f=j_{r_{\rL}^{\rL_{max}}(f)}$. Now for two $f_1\neq f_2\in\Hom_{J_{\rP}}(\lambda_{\rP}\otimes\chi,\pi\otimes\ind_{\{1\}}^{\det_{\ell}}\mathds{1})$. Suppose first that $j_{f_1}=j_{f_2}$. Then 
$$f_1\neq f_2\in\Hom_{J_{\rP}}(\lambda_{\rP},V_{i+1}\slash V_i\cong V)\cong\Hom_{J_{\rP}}(\lambda_{\rP},\pi).$$
Hence $r_{\rL}^{\rL_{max}}(f_1)\neq r_{\rL}^{\rL_{max}}(f_2)$ as explained above. When $j_{f_1}\neq j_{f_2}$, then we deduce the result by $j_{r_{\rL}^{\rL_{max}}(f_1)}\neq j_{r_{\rL}^{\rL_{max}}(f_2)}$.

For Condition 4: Let $V'$ be the representation space of $\pi'$. Denote by $W_0'$ be the subspace of $V'$, on which $J_{\rP}'$ acts as a direct sum of $\lambda_{\rP}'$. Let $W'$ be the kernel of $r_{\rL'}^{\rL_{\max}'}$ on $V'$, to verify Condition 4 is equivalent to show that $W_0'\cap W'=\{0\}$. We know that $W_0'=\sum_{f'\in\Hom(\lambda_{\rP}',\pi')}\mathrm{Im}(f')$ where $\mathrm{Im}$ denotes the image. Equation \ref{equaembedding} sends $f'$ to an element in $\Hom(\lambda_{\rP},\pi\otimes\ind_{J_{\rP}'O}^{J_{\rP}}\mathds{1})$. Notice that $\ind_{J_{\rP}'O}^{J_{\rP}}\mathds{1}$ is a subrepresentation of $\res_{\det(J_{\rP})}^{\mathfrak{o}_{F}^{\times}}\ind_{\det(O)}^{\mathfrak{o}_F^{\times}}\mathds{1}$, and the latter can be regarded as a representation of $F^{\times}$ by extending to $\varpi_{F}$ trivially. Hence can be inflated to $\rL_{\max}$, and we denote it by $\gamma$. We have
$$\Hom(\lambda_{\rP}',\pi')\hookrightarrow\Hom(\lambda_{\rP},\pi\otimes\gamma).$$
Let $f\in\Hom(\lambda_{\rP},\pi\otimes\gamma)$ that corresponds to $f'$. We have $\mathrm{Im}(f')\subset\mathrm{Im}(f)$. Denote by $W$ the kernel of $r_{\rL}^{\rL_{\max}}$ on $\pi\otimes\gamma$, and $W_0:=\sum_{f\in\Hom(\lambda_{\rP},\pi\otimes\gamma)}\mathrm{Im}(f)$. By \cite[Theorem 7.9]{BuKu98}, we know that $(\pi\otimes\gamma)^{\lambda_{\rP}}\cong (r_{\rL}^{\rL_{\max}}\pi\otimes\gamma)^{\lambda}$, which implies that $W\cap W_0=\{0\}$. Notice that $\mathrm{Im}(f')\subset\mathrm{Im}(f)$. Hence $W_0'\subset W_0$. Meanwhile $\pi'$ is a subrepresentation of $\res_{\rL_{\max}'}^{\rL_{\max}}\pi\otimes\gamma$, hence $W'\subset W$. We conclude that $W'\cap W_0'$ is trivial.
\end{proof}

\begin{cor}
\label{iG'cover}
We have an equivalence
$$\ind_{J_{\rP}'}^{\rL_{\text{max}}'}\lambda_{\rP}'\cong i_{\rL'}^{\rL_{\text{max}}'}\ind_{J_{\rL}'}^{\rL'}\lambda_{\rL}'.$$
\end{cor}

\begin{proof}
It is deduced from \cite[Theorem 2]{Blon05}.
\end{proof}

\subsubsection{$(\rL_{\text{max}}',\alpha)$-cover of supercuspidal $k$-types}
\label{sectionG'alphacover}
Let $(J_{max,\alpha},\lambda_{max,\alpha})$ be an $(\rL_{\text{max}},\alpha)$-cover as in Section \ref{sectionGcover}. Denote by $\kappa_{max,\alpha}'$ the irreducible restriction $\kappa_{\text{max}}\vert_{J_{max,\alpha}'}$. Under the tameness condition, we have a decomposition $J_{max,\alpha}'=U(\mathfrak{B})'J_{\text{max}}^{1'}$. In particular, $J_{\text{max}}^{1'} \subset J_{max,\alpha}^{1'}=U^1(\mathfrak{B})'J_{\text{max}}^{1'}$. We have $J_{max,\alpha}'\slash J_{max,\alpha}^{1'}\cong\mathbb{L}'$, and $J_{max,\alpha}'\slash J_{\text{max}}^{1'}\cong\mathbb{P}'$, where the latter is a parabolic subgroup of $\mathbb{L}_{\text{max}}'$ with Levi subgroup $\mathbb{L}'$. We fix a Borel subgroup in $\mathbb{L}_{\text{max}}'$ that consists of upper triangular matrices.  Recall that $\rP$ is standard with respect to the Borel subgroup of upper triangular matrices in $\rG$. There is a natural decomposition that $\rU(\fB)=(\rU(\fB)\cap \bar{\rN})(\rU(\fB)\cap\rL) (\rU(\fB)\cap\rN)$, which identifies $\mathbb{P}'$ with a standard parabolic subgroup of $\mathbb{L}_{max}'$. We denote by $\lambda_{max,\alpha}'=\kappa_{\rP,\alpha}'\otimes\sigma_{\rL}'$. 

\begin{prop}
\label{prop2.6}
The pair $(J_{max,\alpha}',\lambda_{max,\alpha}')$ verifies the following properties:
\begin{itemize}
\item $\ind_{J_{max,\alpha}'}^{(U(\fB)U^1(\fA))'}\lambda_{max,\alpha}'\cong\ind_{J_{\rP}'}^{(U(\fB)U^1(\fA))'}\lambda_{\rP}'$.
\item $\ind_{J_{max,\alpha}'}^{J_{\text{max}}'}\lambda_{max,\alpha}'\cong\kappa_{\text{max}}'\otimes i_{\mathbb{L}'}^{\mathbb{L}_{\text{max}}'}\sigma_{\rL}'$, where $\mathbb{L}_{\text{max}}'\cong J_{\text{max}}'\slash J_{\text{max}}^{1'}$.
\end{itemize}
We call it \textbf{an $(\rL_{\text{max}}',\alpha)$-cover} of $(J_{\rL}',\lambda_{\rL}')$.
\end{prop}

\begin{proof}
We recall the two equations in Section \ref{sectionGcover}:
\begin{itemize}
\item $\ind_{J}^{U(\mathfrak{B})U^1(\mathfrak{A})}\kappa\cong\ind_{U(\fB)J_{\text{max}}^1}^{U(\fB)U^1(\fA)}\kappa_{\text{max}}.$
\item $\ind_{J_{\rP}}^{J}\kappa_{\rP}\cong\kappa.$
\end{itemize}
It is worth noticing that under the condition $p$ does not divide $\vert W_{\rG'}\vert$, we have an equation
$$N_{E\slash F}(\mathfrak{o}_{E}^{\times})=\det_{F}(U(\mathfrak{B})U^1(\mathfrak{A}))=\det_{F}(J)=\det_F(J_{\rP})=\det_F(U(\fB)J_{\text{max}}^1).$$
Hence by applying restriction to the intersection with $\rL_{\text{max}}'$ and Mackey's formula, we have
\begin{itemize}
\item $\ind_{J'}^{(U(\mathfrak{B})U^1(\mathfrak{A}))'}\kappa'\cong\ind_{(U(\fB)J_{\text{max}}^1)'}^{(U(\fB)U^1(\fA))'}\kappa_{\text{max}}'.$
\item $\ind_{J_{\rP}'}^{J'}\kappa_{\rP}'\cong\kappa'.$
\end{itemize}
On the other hand, since $\det_F(\rU^1(\fB))=\det_F(\rU^1(\fA))=1+\mathfrak{p}_F$, we have $(\rU(\fB)\rU^1(\fA))'=\rU(\fB)'\rU^1(\fA)'$. Hence $\sigma_{\rL}'$ extends to $(\rU(\fB)\rU^1(\fA))'$ by acting $\rU^1(\fA)'$ trivially. Then we have
$$\ind_{J'}^{(U(\fB)U^1(\fA))'}\lambda_{\rP}'\cong (\ind_{J'}^{(U(\fB)U^1(\fA))'}\kappa_{\rP}')\otimes\sigma_{\rL}';$$
$$\ind_{(U(\fB)J_{\text{max}}^1)'}^{(U(\fB)U^1(\fA))'}\lambda_{max,\alpha}'\cong(\ind_{(U(\fB)J_{\text{max}}^1)'}^{(U(\fB)U^1(\fA))'}\kappa_{max,\alpha}')\otimes\sigma_{\rL}'.$$
They imply the following equivalences that generalise those of $(\rL_{\text{max}},\alpha)$-covers in Section \ref{sectionGcover}.
\begin{equation}
\label{equaG'alpha}
\ind_{J'}^{(U(\fB)U^1(\fA))'}\lambda_{\rP}'\cong\ind_{(U(\fB)J_{\text{max}}^1)'}^{(U(\fB)U^1(\fA))'}\lambda_{max,\alpha}'.
\end{equation}
Meanwhile,
$$\ind_{J_{max,\alpha}'}^{J_{\text{max}}'}\lambda_{max,\alpha}'\cong\kappa_{\text{max}}'\otimes\ind_{J_{max,\alpha}'}^{J_{\text{max}}'}\sigma_{\rL}'.$$
In particular, 
$$\ind_{J_{max,\alpha}'}^{J_{\text{max}}'}\sigma_{\rL}'\cong i_{\mathbb{L}'}^{\mathbb{L}_{\text{max}}'}\sigma_{\rL}',$$

\end{proof}

\begin{cor}
\label{corJ'supcuspsupp}
Let $\rho'$ be an irreducible representation of $\mathbb{L}_{\text{max}}'$ of which the supercuspidal support is $[\mathbb{L}',\sigma_{\rL}']$. Denote by $\gamma':=\kappa_{\text{max}}'\otimes\rho'$. Then the supercuspidal support of irreducible subquotients of $\ind_{J_{\text{max}}'}^{\rL_{\text{max}}'}\gamma'$ belong to the supercuspidal class determined by $[J_{\rL}',\lambda_{\rL}']$.
\end{cor}

\begin{proof}
By \cite{Blon05}, we have $i_{\rL'}^{\rL_{\text{max}}'}\ind_{J_{\rL}'}^{\rL'}\lambda_{\rL}'\cong\ind_{J_{\rP}'}^{\rL_{\text{max}}'}\lambda_{\rP}'$. Proposition \ref{prop2.6} implies that 
$$i_{\rL'}^{\rL_{\text{max}}'}\ind_{J_{\rL}'}^{\rL'}\lambda_{\rL}'\cong\ind_{J_{\text{max}}'}^{\rL_{\text{max}}'}\kappa_{\text{max}}'\otimes i_{\mathbb{L}'}^{\mathbb{L}_{\text{max}}'}\sigma_{\rL}'.$$
Since $\rho'$ is an irreducible subquotient of $i_{\mathbb{L}'}^{\mathbb{L}_{\text{max}}'}\sigma_{\rL}'$. We conclude the result by the fact that an irreducible subquotient of $\ind_{J_{\rL}'}^{\rL'}\lambda_{\rL}'$ must contain $(J_{\rL}',\lambda_{\rL}')$.
\end{proof}

\begin{defn}
\label{defnJ'supcuspsupp}
We define the $\rL_{\text{max}}'$-conjugacy class $[J_{\rL}',\lambda_{\rL}']$ \textbf{the supercuspidal support} of the pair $(J_{\text{max}}',\kappa_{\text{max}}'\otimes\rho')$.
\end{defn}

\subsection{The functor $\mathbf{K}'$}
In this section, we show a compatibility between parabolic induction of the finite reductive quotient groups and that of $p$-adic groups (Proposition \ref{propK'}), which relies on our choice of $\kappa_{max},\kappa,\kappa_{\rL}$ in Section \ref{sectionGcover}. However, we need this compatibility not only with the supercuspidal support, but also with the cuspidal support. In other words, we want to show this compatibility respects the transitivity of parabolic induction. To do that, we need to determine wild pairs of a Levi subgroup $\rM$ in between of $\rL_{max}$ and $\rL$ and prove all these choices meet our requirements.

\subsubsection{Wild pairs on different Levi subgroups}
\label{sectionK'Levi}

Recall that  $J_{\text{max}},J_{max,\alpha},J_{\rP}$ are open compact subgroups, and $[\fA_{\text{max}},0,\beta],[\fA,0,\beta]$ are simple strata as in Section \ref{sectionGcover}, which are determined by the endo-equivalence class of $\theta_{\rL}$. Suppose that $\rM$ is a Levi subgroup, that contains $\rL$ and is contained in $\rL_{max}$, and that has a simple character $\theta_{\rM}$ endo-equivalent to $\theta_{\rL}$. We will determine a wild pair of $\rM$, which are endo-equivalent to $\kappa_{\rL}$ and $\kappa_{\text{max}}$, then we show it is compatible with both of them in the sense of Equation \ref{equakappaG}. This choice is unique after fixing $\kappa_{max}$, and the compatibility can be deduced from \cite[\S 5]{SS}, and regarded as a missing piece. We \textbf{warn} again that our notations are different from those in \cite[\S 5]{SS}.


In order to simplify the notation, \textbf{we assume that $\rL_{\text{max}}=\rG=\mathrm{GL}_n(F)$} (only in this section). That is to say we only consider the homogeneous cases (the general cases can be regarded as a product of homogeneous ones), in other words, we are in the case that $E=F[\beta_{\rL}]$ is a field extension, meaning that $\beta_{\rL}=(\beta,\cdots,\beta)$ with $\beta\in E$. In order to discuss wild pairs in three different groups, we need to refine our notation system. Denote by
$$[\fA_{\rL}^{\rG},0,\beta]:=[\fA,0,\beta],$$
$$[\fA_{\text{max}}^{\rG},0,\beta]:=[\fA_{\text{max}},0,\beta]$$
$$J_{\max}^{\rG}:=J_{\text{max}}$$
$$J_{\rP,\rL}^{\rG}:=J_{\rP}.$$
The idea of this notation system is to represent $\fA_{\rL}^{\rG}$ that it comes from a maximal hereditary order of $\rL$ and is defined on $\rG$ by the lower and upper index.

Denote by $\Lambda_{\text{max}}^{\rG},\Lambda_{\rL}^{\rG}$ the $\mathfrak{o}_F$-lattice chains (resp. $\Lambda_{\rL}$ a product of $\mathfrak{o}_{F}$-lattice chain), which are $\mathfrak{o}_{E}$-lattice chains (resp. a product of $\mathfrak{o}_E$-lattice chain) as well. In particular, by writing $\Lambda_{\rL}\cong\prod_{i\in I}\Lambda_{\rL,i}$ as a direct sum of lattice chains, each component $\Lambda_{\rL,i}$ is an $\mathfrak{o}_{E}$-lattice chain in $E^{m_i}$ where $m_i=n_i\slash [E:F]$. 
Define $\fA_{\text{max}}^{\rG},\fA_{\rL}^{\rG},\fA_{\rL}$ accordingly. Denote by $\mathfrak{B}_{\text{max}}^{\rG},\mathfrak{B}_{\rL}^{\rG},\mathfrak{B}_{\rL}$ the corresponding $\mathfrak{o}_{E}$-hereditary orders. Recall that by definition,
$$\rU(\fA_{\rL}^{\rG})\subset\rU(\fA_{\text{max}}^{\rG});$$
$$\rU(\mathfrak{B}_{\rL}^{\rG})\cap\rL=\rU(\mathfrak{B}_{\rL}).$$
In additional, let $\rP_{\rL}^{\rG}=\rL\rN_{\rL}^{\rG}$ be a standard parabolic subgroup. Up to an $\rL$-conjugation of $\lambda_{\rL}$ we can assume that the lattice chains verify the following equations:
\begin{equation}
\label{equa9}
\rU(\mathfrak{B}_{\rL}^{\rG})=(\rU^1(\fA_{\text{max}}^{\rG})\cap \rU(\mathfrak{B}_{\rL}^{\rG}))(\rU(\mathfrak{B}_{\rL}^{\rG})\cap\rP_{\rL}^{\rG}).
\end{equation}
\begin{equation}
\label{equa10}
\rU^1(\fA_{\rL}^{\rG})\cap\bar{\rN}_{\rL}^{\rG}=\rU^1(\fA_{\text{max}}^{\rG})\cap\bar{\rN}_{\rL}^{\rG}.
\end{equation}
They are in \cite[Equation 5.1]{SS}. In \cite[\S 5.1]{SS}, an example of such kind of $\Lambda_{\rL}^{\rG},\Lambda_{\text{max}}^{\rG}$ has been established in the proof of Lemma 5.1 of \cite{SS}.

Now we consider the compatibility with a Levi subgroup $\rM\cong\prod_{s\in S}\mathrm{GL}_{n_s}(F)$ containing $\rL$. We start by assuming that $\Lambda_{\text{max}}^{\rG}$ is in the same form of the example in Section 5.1 of \cite{SS}, which is constructed with respect to a fixed standard basis. To be more precise, let $V$ be a $F$-vector space, such that $\mathrm{GL}_n(F)\cong\mathrm{GL}(V)$. Fix a basis $\mathcal{E}=(\mathfrak{e}_1,\cdots,\mathfrak{e}_n)$ for $V$ such that \textbf{the standard Borel subgroup} consists of upper triangular matrices with respect to this basis. For $\rL$, there is a partition $\mathcal{E}=\sqcup_{i\in I}\mathcal{E}_i$ which gives a family of sub-spaces $V_{\rL}:=(V_i,i\in I)$ where $V_i$ is generated by $\mathcal{E}_i$ such that $\rL\cong\prod_{i\in I}\mathrm{GL}(V_i)$. We also write $\rL$ as $\mathrm{GL}(V_{\rL})$. 
Similarly taking the partition $\mathcal{E}=\sqcup_{s\in S}\mathcal{E}_s$, we get $V_{\rM}$ such that $\mathrm{GL}(V_{\rM})=\rM$. The group $\rM$ is diagonal by blocks, so is the projection to $\mathrm{GL}_{n_s}(F)$. When $\rL\subset\rM$, there is an embedding of $E$ to $\rM$ and we can write $\beta=(\beta_s)_{s\in S}$. Furthermore, by removing the repetitions of lattice chain $\Lambda_{\rL}^{\rG}\cap V_{\rM}$ and $\Lambda_{\text{max}}^{\rG}\cap V_{\rM}$, we obtain lattice chains $\Lambda_{\rL}^{\rM}$ and $\Lambda_{\text{max}}^{\rM}$ respectively, which are also $\mathfrak{o}_{E}$-lattice chains in the above sense under our assumption on $\rM$. Let $[\fA_{\rL}^{\rM},n,0,\beta]$ and $[\fA_{\text{max}}^{\rM},n,0,\beta]$ be the corresponding simple strata, and the latter is maximal. Denote by $\mathfrak{B}_{X}^{\rM}$ the $\mathfrak{o}_{E}$-orders in $\rM$ with respect to $\Lambda_{X}^{\rM}$ ($X=\max,\rL$). There exists a maximal lattice chain $\Lambda_{\text{max}}^{\rG}\subset\Lambda_{\rM}^{\rG}$, which gives the lattice chain $\Lambda_{\text{max}}^{\rM}$ by removing the repetition of $\Lambda_{\rM}^{\rG}\cap V_{\rM}$, and which defines  a simple stratum $[\fA_{\rM}^{\rG},n,0,\beta]$ on $\rG$.

\begin{rem}
\label{remlatticedecomp}
We regard $\Lambda_{\rL}^{\rG}$ as an $\mathfrak{o}_{E}$-lattice chain, which gives a decomposition of $V\cong\oplus_{i\in I}V_i$ as in \cite[5.5.2]{BuKuI}. This decomposition is subordinate to the $\mathfrak{o}_F$-lattice chain $\Lambda_{\rL}^{\rG}$ as defined in \cite[7.1.1]{BuKuI}. Our Levi subgroup $\rL$ is identified with $\prod_{i\in I}\mathrm{GL}(V_i)$, and the standard parabolic subgroup $\rP_{\rL}^{\rG}$ coincides with the parabolic subgroup defined in \cite[7.1.13]{BuKuI}. In particular, by \cite[Theorem 7.1.14]{BuKuI} the equations \ref{equa9} and \ref{equa10} are satisfied.
\end{rem}

Denote by $\rP_{X}^{Y}$ the standard parabolic subgroup of $Y$ with Levi part $X$, and $\rN_{X}^{Y},\bar{\rN}_{X}^{Y}$ the unipotent radical of of $\rP_{X}^{Y}$ as well as the opposite unipotent radical. Let $H_{X}^{Y}, J_{X}^{Y}$ be defined from the simple stratum with the same upper and lower index (see \cite[3.1.4]{BuKuI}). For an open compact subgroup $K$, we denote by $K^1$ its pro-$p$ radical. We say a subgroup $K\subset\rG$ is \textbf{decomposed with respect to a parabolic subgroup} $\rP=\rL\rN$, if
$$K=(K\cap\bar{\rN})(K\cap\rL)(K\cap\rN_{\rL}^{\rM}),$$
where $\bar{\rN}$ is the opposite unipotent radical. 

Inside $\rM$, we define with respect to $\rP_{\rL}^{\rM}$:
$$J_{\rP,\rL}^{\rM}=(H_{\rL}^{\rM,1}\cap\bar{\rN}_{\rL}^{\rM})(J_{\rL}^{\rM}\cap\rL)(J_{\rL}^{\rM,1}\cap\rN_{\rL}^{\rM}).$$

Inside $\rG$, we define $J_{\rP,\rM}^{\rG},J_{\rP,\rL}^{\rG},J_{\rP,\rL}^{\rG,\rM}$ with respect to $\rP_{\rM}^{\rG}$ and $\rP_{\rL}^{\rG}$:
 $$J_{\rP,\rM}^{\rG}=(H_{\rM}^{\rG,1}\cap\bar{\rN}_{\rM}^{\rG})(J_{\rM}^{\rG}\cap\rM)(J_{\rM}^{\rG,1}\cap\rN_{\rM}^{\rG});$$
 $$J_{\rP,\rL}^{\rG}=(H_{\rL}^{\rG,1}\cap\bar{\rN}_{\rL}^{\rG})(J_{\rL}^{\rG}\cap\rL)(J_{\rL}^{\rG,1}\cap\rN_{\rL}^{\rG}).$$
The groups $J_{\rP,X}^{Y}$ are defined in a similar manner as $J_{\rP}$ in Section \ref{Lmaxcovers} to construct a cover of $(J_{\rL},\lambda_{\rL})$. 

We emphasis that we have fixed a family of lattice-chains and $\beta$. Now we define $\kappa_{\text{max}}^{\rG},\kappa_{\rM}^{\rG},\kappa_{\rL}^{\rG},\kappa_{\text{max}}^{\rM},\kappa_{\rL}^{\rM},\kappa_{\rL}$ as following:
\begin{enumerate}
\item On $\rL$, the simple stratum $[\fA_{\rL},n,0,\beta]$ has been determined with respect to $\Lambda_{\rL}$. We start by fixing $\theta_{\rL}\in\mathcal{C}(\fA_{\rL},n,0,\beta)$.
\item On $\rG$, we define $\kappa_{\text{max}}^{\rG},\kappa_{\rM}^{\rG},\kappa_{\rL}^{\rG}$: The simple strata $[\fA_{\text{max}}^{\rG},n,0,\beta],[\fA_{\rM}^{\rG},n,0,\beta],[\fA_{\rL}^{\rG},n,0,\beta]$ has been fixed as above. As in Section \ref{sectionGcover}, let $\theta_{\text{max}}^{\rG}\in\mathcal{C}(\fA_{\text{max}}^{\rG},n, 0,\beta)$ be endo-equivalent to $\theta_{\rL}$, we fix a wild-extension $\kappa_{\text{max}}^{\rG}$ of $\theta_{\text{max}}^{\rG}$. 

Again, for $X=\rM$ or $\rL$, let $\theta_{X}^{\rG}$ be the unique simple character in $\mathcal{C}(\fA_{X}^{\rG},n,0,\beta)$ that is endo-equivalent to $\theta_{\rL}$. Let $\kappa_{X}^{\rG}$ be the unique wild-extension of $\theta_{X}^{\rG}$ which verifies the following equation (Equation \ref{equakappaG}):

\begin{equation}
\label{equationcompat10}
\ind_{J_{X}^{\rG}}^{U(\fB_{X}^{\rG})U^1(\fA_{X}^{\rG})}\kappa_{X}^{\rG}\cong\ind_{U(\fB_{X}^{\rG})J_{\text{max}}^{1,\rG}}^{U(\fB_{X}^{\rG})U^1(\fA_{X}^{\rG})}\kappa_{\text{max}}^{\rG}.
\end{equation}

\item For $\rM$, we define $\kappa_{\text{max}}^{\rM}$ and $\kappa_{\rL,0}^{\rM}$: The simple strata $[\fA_{\text{max}}^{\rM},n,0,\beta],[\fA_{\rL}^{\rM},n,0,\beta]$ and open compact subgroups $J_{\text{max}}^{\rM},J_{\rL}^{\rM}$ are determined as above. For $X=\max$ or $\rL$, let $\theta_{X}^{\rM}$ be the unique simple character in $\mathcal{C}(\fA_{X}^{\rM},n,0,\beta)$ which is endo-equivalent to $\theta_{\rL}$. There is a unique wild-extension $\kappa_{\text{max}}^{\rM}$ of $\theta_{\text{max}}^{\rM}$ on $J_{\text{max}}^{\rM}$, such that after extending to $J_{\rP,\rM}^{\rG}$ as in Section \ref{Lmaxcovers} we have
\begin{equation}
\label{equationcompt11}
\ind_{J_{\rP,\rM}^{\rG}}^{J_{\rM}^{\rG}}\kappa_{\text{max}}^{\rM}\cong\kappa_{\rM}^{\rG}.
\end{equation}
Meanwhile there is a unique wild-extension $\kappa_{\rL,0}^{\rM}$ of $\theta_{\rL}^{\rM}$, such that
\begin{equation}
\label{equationcompt12}
\ind_{J_{\rL}^{\rM}}^{U(\fB_{\rL}^{\rM})U^1(\fA_{\rL}^{\rM})}\kappa_{\rL,0}^{\rM}\cong\ind_{U(\fB_{\rL}^{\rM})J_{\text{max}}^{\rM,1}}^{U(\fB_{\rL}^{\rM})U^1(\fA_{\rL}^{\rM})}\kappa_{\text{max}}^{\rM}.
\end{equation}

\item For $\rL$, we define $\kappa_{\rL}$ the unique wild extension of $\theta_{\rL}$, such that
\begin{equation}
\label{equationcompt14}
\kappa_{\rL}^{\rG}\cong\ind_{J_{\rP,\rL}^{\rG}}^{J_{\rL}^{\rG}}\kappa_{\rL}.
\end{equation}

Meanwhile, we define $\kappa_{\rL,0}$ the unique wild extension of $\theta_{\rL}$, such that
\begin{equation}
\label{equationcompt15}
\kappa_{\rL,0}^{\rM}\cong\ind_{J_{\rP,\rL}^{\rM}}^{J_{\rL}^{\rM}}\kappa_{\rL,0}.
\end{equation}

\end{enumerate}

\begin{lem}
\label{lemequapart2}
The above construction verifies:
$$\ind_{J_{\rL}^{\rG}}^{U(\fB_{\rL}^{\rG})U^1(\fA_{\rL}^{\rG})}\kappa_{\rL}^{\rG}\cong\ind_{U(\fB_{\rL}^{\rG})J_{\rM}^{\rG,1}}^{U(\fB_{\rL}^{\rG})U^1(\fA_{\rL}^{\rG})}\kappa_{\rM}^{\rG}.$$
\end{lem}

\begin{proof}
Since $U(\fB_{\rM}^{\rG})U^1(\fA_{\rM}^{\rG})\subset U(\fB_{\rM}^{\rG})U^1(\fA_{\rL}^{\rG})$ we deduce from Equation \ref{equationcompat10} an isomorphism of $k$-vector space:
$$\ind_{J_{\rM}^{\rG}}^{U(\fB_{\rM}^{\rG})U^1(\fA_{\rL}^{\rG})}\kappa_{\rM}^{\rG}\cong\ind_{U(\fB_{\rM}^{\rG})J_{\text{max}}^{\rG,1}}^{U(\fB_{\rM}^{\rG})U^1(\fA_{\rL}^{\rG})}\kappa_{\text{max}}^{\rG}.$$
Notice that $U(\fB_{\rM}^{\rG})U^1(\fA_{\rL}^{\rG})$ is not a group. We regard it as a double coset $J_{\rM}^{\rG}\cdot U(\fB_{\rL}^{\rG})U^1(\fA_{\rL}^{\rG})$ of $J_{\rM}^{\rG}$ and $U(\fB_{\rL}^{\rG})U^1(\fA_{\rL}^{\rG})$. The group $U(\fB_{\rL}^{\rG})U^1(\fA_{\rL}^{\rG})$ acts as right translation on the space of functions defined on this double coset. The above isomorphism induces an equivalence of representations of $U(\fB_{\rL}^{\rG})U^1(\fA_{\rL}^{\rG})$.
Since 
$$[U(\fB_{\rM}^{\rG})U^1(\fA_{\rL}^{\rG}):J_{\rM}^{\rG}]=[U(\fB_{\rL}^{\rG})U^1(\fA_{\rL}^{\rG}):U(\fB_{\rL}^{\rG})J_{\rM}^{\rG,1}],$$
and
$$[U(\fB_{\rM}^{\rG})U^1(\fA_{\rL}^{\rG}):U(\fB_{\rM}^{\rG})J_{\text{max}}^{\rG,1}]=[U(\fB_{\rL}^{\rG})U^1(\fA_{\rL}^{\rG}):U(\fB_{\rL}^{\rG})J_{\text{max}}^{\rG,1}],$$
by Mackey's formula we have 
$$\ind_{U(\fB_{\rL}^{\rG})J_{\rM}^{\rG,1}}^{U(\fB_{\rL}^{\rG})U^1(\fA_{\rL}^{\rG})}\kappa_{\rM}^{\rG}\cong\ind_{U(\fB_{\rL}^{\rG})J_{\text{max}}^{\rG,1}}^{U(\fB_{\rL}^{\rG})U^1(\fA_{\rL}^{\rG})}\kappa_{\text{max}}^{\rG}.$$
The latter is equivalent to $\ind_{J_{\rL}^{\rG}}^{U(\fB_{\rL}^{\rG})U^1(\fA_{\rL}^{\rG})}\kappa_{\rL}^{\rG}$ after Equation \ref{equationcompat10}.
\end{proof}

By Mackey's theory the above isomorphism is equivalent to
\begin{equation}
\ind_{J_{\rP,\rL}^{\rG}}^{U(\fB_{\rL}^{\rG})U^1(\fA_{\rL}^{\rG})}\kappa_{\rL}\cong\ind_{U(\fB_{\rL}^{\rG})J_{\rM}^{\rG,1}}^{U(\fB_{\rL}^{\rG})U^1(\fA_{\rL}^{\rG})}\res_{J_{\rP,\rM}^{\rG}\cap U(\fB_{\rL}^{\rG})J_{\rM}^{\rG,1}}^{J_{\rP,\rM}^{\rG}}\kappa_{\max}^{\rM}.
\end{equation}

\begin{lem}
\label{lemchartriv}
For any simple character $\theta$ in $\mathcal{C}(\fA_{\rL}^{\rG},0,\beta),\mathcal{C}(\fA_{\rM}^{\rG},0,\beta),\mathcal{C}(\fA_{\max}^{\rG},0,\beta)$, the restrictions $\theta\vert_{\rN_{\rL}^{\rG}}$ and $\theta\vert_{\bar{\rN}_{\rL}^{\rG}}$ are trivial.
\end{lem}

\begin{proof}
We take $\theta\in \mathcal{C}(\fA,0,\beta_{\rL})$ where $\fA=\fA_{\rL}^{\rG}$ (resp. $\fA_{\rM}^{\rG}$ or $\fA_{\max}^{\rG}$). Suppose $\beta$ is minimal over $F$ (see definition of minimal element in \cite[1.4.14]{BuKuI}). Then $\theta\vert_{\rN_{\rL}^{\rG}\cap H^1(\fA,0,\beta)}=\theta\vert_{\bar{\rN}_{\rL}^{\rG}\cap H^1(\fA,0,\beta)}$ is trivial from the definition (see \cite[Definition 3.2.1]{BuKuI}). Now suppose $\beta$ is not minimal. Then we obtain a series of elements $\gamma_j\in\rL$ where $0\leq j \leq m$ for $m\in\mathbb{N}$ as in \cite[2.4.2]{BuKuI}, such that $\gamma_0=\beta_{\rL}$, and the projection of $\gamma_m$ to each block of $\rL$ is either minimal over $F$ or belonging to $F$. Moreover, $[\fA,0,\gamma_j]$ is a simple stratum for each $j$, from which we obtain $\fB_{\gamma_j}=\fB_{\gamma_j,\rL}^{\rG}$ (resp. $\fB_{\gamma_j,\rM}^{\rG}$ or $\fB_{\gamma_j,\max}^{\rG}$ accordingly). Denote by $\alpha_j=\gamma_j-\gamma_{j+1}$. Notice that we are in the case the $p$ verifies the tameness condition. We summarise two properties from \cite[Proposition 7.3, Proposition 7.5]{AM}: For each $j\leq m-1$, define $d_j=[\frac{-\nu_{\fA}(\alpha_j)}{2}]+1$.
\begin{itemize}
\item $H^1(\fA,0,\gamma_j)=U^1(\fB_{\gamma_j})H^{d_j}(\fA,0,\gamma_{j+1})$, where $H^{d_j}(\fA,0,\gamma_{j+1})$ is a subgroup of $H^{1}(\fA,0,\gamma_{j+1})$ (see \cite[3.1.14]{BuKuI}).
\item For any $\theta_i\in\mathcal{C}(\fA,0,\gamma_i)$, there is a character $\theta_{j+1}\in\mathcal{C}(\fA,0,\gamma_{j+1})$ such that $\theta_i$ factor through determinant on $U^1(\fB_{\gamma_j})$ and is equal to $\theta_{j+1}\circ\psi_j$ on $H^{d_j}(\fA,0,\gamma_{j+1})$, where $\psi_j$ is the character of Pontrjagin dual of $\alpha_j$ (see for example \cite[1.1.6]{BuKuI}).
\end{itemize}
Now we do induction along $j$. We have already shown that any $\theta_m\in\mathcal{C}(\fA,0,\gamma_{m})$ is trivial on $\rN_{\rL}^{\rG}$ and $\bar{\rN}_{\rL}^{\rG}$. Suppose for each $\theta_{j+1}\in\mathcal{C}(\fA,0,\gamma_{j+1})$, it is trivial on $\rN_{\rL}^{\rG}$ and $\bar{\rN}_{\rL}^{\rG}$, then we conclude the same property for $\theta_j\in\mathcal{C}(\fA,0,\gamma_{j})$ by the above two properties and the fact that $\psi_j$ is also trivial on $\rN_{\rL}^{\rG}$ and $\bar{\rN}_{\rL}^{\rG}$.
\end{proof}

\begin{cor}
\label{cor4.10}
We have
\begin{enumerate}
\item $\Hom_{J_{\rP,\rL}^{\rM}\cap U(\fB_{\rL}^{\rM})J_{\max}^{\rM,1}}(\kappa_{\max}^{\rM},\kappa_{\rL,0})\neq\{0\}$, and
\item $\Hom_{J_{\rP,\rL}^{\rG}\cap J_{\rP,\rM}^{\rG}\cap U(\fB_{\rL}^{\rG})J_{\rM}^{\rG,1}}(\kappa_{\max}^{\rM},\kappa_{\rL})\neq \{0\}.$
\end{enumerate}
\end{cor}

\begin{proof}
We start from the first part. By \ref{equationcompt12} we know $\kappa_{\max}^{\rM}$ is the unique irreducible subquotient of $\res_{U(\fB_{\rL}^{\rM})J_{\max}^{\rM,1}}^{U(\fB_{\rL}^{\rM})U^1(\fA_{\rL}^{\rM})}\ind_{J_{\rP,\rL}^{\rM}}^{U(\fB_{\rL}^{\rM})U^1(\fA_{\rL}^{\rM})}\kappa_{\rL,0}$ which contains $\theta_{\max}^{\rM}$. Denote by $J_{\rP,\rL}^{\rM,1}:=H_{\rL}^{\rM,1}(J_{\rL}^{\rM,1}\cap \rP_{\rL}^{\rM})$. Since $J_{\rP,\rL}^{\rM}\cap J_{\max}^{\rM,1}=J_{\rP,\rL}^{\rM,1}\cap J_{\max}^{\rM,1}$, by Mackey's theory, it is sufficient to show that 
$$\theta_{\max}^{\rM}\hookrightarrow\res_{J_{\rP,\rL}^{\rM,1}\cap H_{\max}^{\rM,1}}^{J_{\rP,\rL}^{\rM,1}}\eta_{\rL},$$
where $\eta_{\rL}$ is inflated to $J_{\rP,\rL}^{\rM,1}$. By \cite[7.2.4]{BuKuI}, we know $\ind_{H_{\rL}^{\rM,1}}^{J_{\rP,\rL}^{\rM,1}}\theta_{\rL}^{\rM}$ is semisimple and each direct component is isomorphism to $\eta_{\rL,\phi}$, which is $\eta_{\rL}$ on $J_{\rL}^1=J_{\rP,\rL}^{\rM,1}\cap\rL$, and is a multiple of a character $\phi$ on $J_{\rP,\rL}^{\rM,1}\cap\rN_{\rL}^{\rM}$. Notice that $\theta_{\max}^{\rM}$ is the unique simple character in $\mathcal{C}(\fA_{\max}^{\rM},0,\beta)$ such that 
$$\theta_{\max}^{\rM}\vert_{H_{\max}^{\rM,1}\cap H_{\rL}^{\rM,1}}\cong\theta_{\rL}^{\rM}\vert_{H_{\max}^{\rM,1}\cap H_{\rL}^{\rM,1}}.$$
Hence $\theta_{\max}^{\rM}\hookrightarrow\ind_{H_{\rL}^{\rM,1}\cap H_{\max}^{\rM,1}}^{H_{\max}^{\rM,1}}\theta_{\rL}^{\rM}$. We deduce that $\theta_{\max}^{\rM}\hookrightarrow\ind_{H_{\max}^{\rM,1}\cap J_{\rP,\rL}^{\rM,1}}^{H_{\max}^{\rM,1}}\eta_{\rL,\phi}$ for some $\phi$ as above. We deduce by Lemma \ref{lemchartriv} that $\theta_{\max}^{\rM}\hookrightarrow\ind_{H_{\max}^{\rM,1}\cap J_{\rP,\rL}^{\rM,1}}^{H_{\max}^{\rM,1}}\eta_{\rL}$.

Now we look at the second assertion. We deduce from Lemma \ref{lemequapart2} that $\kappa_{\max}^{\rM}$ is the unique irreducible subquotient of $\res_{J_{\rP,\rM}^{\rG}\cap U(\fB_{\rL}^{\rG})J_{\rM}^{\rG,1}}^{U(\fB_{\rL}^{\rG})U^1(\fA_{\rL}^{\rG})}\ind_{J_{\rP,\rL}^{\rG}}^{U(\fB_{\rL}^{\rG})U^1(\fA_{\rL}^{\rG})}\kappa_{\rL}$ which contains $\theta_{\rM}^{\rG}$ and is $J_{\rP,\rM}^{\rG,1}\cap\rN_{\rM}^{\rG}$-invariant. Hence by inflating $\eta_{\max}^{\rM}$ to $J_{\rP,\rM}^{\rG,1}$, the desired result is equivalent to
$$\Hom_{J_{\rP,\rL}^{\rG,1}\cap J_{\rP,\rM}^{\rG,1}}(\eta_{\max}^{\rM},\eta_{\rL})\neq \{0\}$$
By applying a similar argument as above, we know that
\begin{equation}
\label{equathetaMG}
\theta_{\rM}^{\rG}\hookrightarrow\ind_{J_{\rP,\rL}^{\rG,1}\cap H_{\rM}^{\rG,1}}^{J_{\rP,\rL}^{\rG,1}}\eta_{\rL}.
\end{equation}
Again by \cite[7.2.4]{BuKuI}, we know $\ind_{H_{\rM}^{\rG,1}}^{J_{\rP,\rM}^{\rG,1}}\theta_{\rM}^{\rG}$ is semisimple and each direct component is isomorphism to $\eta_{\rM,\varphi}$, which is $\eta_{\max}^{\rM}$ on $J_{\max}^{\rM,1}=J_{\rP,\rM}^{\rG,1}\cap\rM$, and is a multiple of a character $\varphi$ on $J_{\rP,\rM}^{\rG,1}\cap\rN_{\rM}^{\rG}$. We deduce by \ref{equathetaMG} that $\Hom_{J_{\rP,\rL}^{\rG,1}\cap J_{\rP,\rM}^{\rG,1}}(\eta_{\rM,\varphi},\eta_{\rL})\neq\{0\}$ for some $\varphi$. Since $\eta_{\rL}$ is trivial on $J_{\rP,\rL}^{\rG,1}\cap\rN_{\rL}^{\rG}$, we conclude that
$$\Hom_{J_{\rP,\rL}^{\rG,1}\cap J_{\rP,\rM}^{\rG,1}}(\eta_{\max}^{\rM},\eta_{\rL})\neq \{0\}.$$
\end{proof}

Denote by $J_{\rG,\rL}^{\rM}:=J_{\rP,\rL}^{\rG}\cap\rM$, which is decomposable with respect to $\rP_{\rL}^{\rM}$, and $J_{\rP,\rL}^{\rG}\cap\rL=J_{\rP,\rL}^{\rM}\cap\rL=J_{\rL}$. We denote by $\kappa_{\rG,\rL}^{\rM}$ the restriction $\kappa_{\rP,\rL}^{\rG}\vert_{J_{\rG,\rL}^{\rM}}$, which is an inflation of $\kappa_{\rL}$ to $J_{\rG,\rL}^{\rM}$. To show the compatibility property, we need to understand the relation between $J_{\rP,\rL}^{\rG}\cap\rM$ and $J_{\rP,\rL}^{\rM}$. They are different in general, but they verify the following equation.  

\begin{lem}
\label{lemcompatinclu}
For $J=J_{\rG,\rL}^{\rM}$ or $J_{\rP,\rL}^{\rM}$, we have:
\begin{itemize}
\item  $J\cap \bar{\rN}_{\rL}^{\rM}\subset \mathrm{ker}(\theta_{\max}^{\rM})$;
\item $J\cap U(\fB_{\rL}^{\rM}) J_{\max}^{\rM,1}=(J\cap\bar{\rN}_{\rL}^{\rM})(J\cap\rL)(U(\fB_{\rL}^{\rM})J_{\max}^{\rM,1}\cap\rN_{\rL}^{\rM})$; and
\item $(J_{\rP,\rL}^{\rM}\cap U(\fB_{\rL}^{\rM})J_{\max}^{\rM,1})\subset (J_{\rG,\rL}^{\rM}\cap U(\fB_{\rL}^{\rM})J_{\max}^{\rM,1}).$
\end{itemize}
\end{lem}

\begin{proof}
We deduce from the proof of Corollary \ref{cor4.10}, that
$$J_{\max}^{\rM}\cap\bar{\rN}_{\rL}^{\rM}\subset \mathrm{ker}(\theta_{\max}^{\rM}).$$


We occupy the notations in the proof of Lemma \ref{lemchartriv}, and define 
\begin{itemize}
\item $d_{j,\rG}=[\frac{-\nu_{\fA_{\rL}^{\rG}}(\alpha_j)+1}{2}],d_{j,\rG}^{+}=[\frac{-\nu_{\fA_{\rL}^{\rG}}(\alpha_j)}{2}]+1$;
\item $d_{j,\rM}=[\frac{-\nu_{\fA_{\max}^{\rM}}(\alpha_j)+1}{2}],d_{j,\rM}^{+}=[\frac{-\nu_{\fA_{\max}^{\rM}}(\alpha_j)}{2}]+1$;
\item $d_{j,\rL}=[\frac{-\nu_{\fA_{\rL}^{\rM}}(\alpha_j)+1}{2}],d_{j,\rL}^{+}=[\frac{-\nu_{\fA_{\rL}^{\rM}}(\alpha_j)}{2}]+1$.
\end{itemize}
Define $e_j=e(F[\beta]\vert F[\alpha_j])\nu_{F[\alpha_j](\alpha_j)}$, where the former is the ramification index of a field extension. This integer depends only on $\beta$ and $\alpha_j$. Denote by $e(\fA)$ the period of a hereditary order as defined in \cite[\S 1.1]{BuKuI}. We deduce from \cite[Lemma 2.1]{AM} that
\begin{itemize}
\item $\nu_{\fA_{\rL}^{\rG}}(\alpha_j)=e(\fB_{\rL}^{\rG})e_j$;
\item $\nu_{\fA_{\max}^{\rM}}(\alpha_j)=e(\fB_{\max}^{\rM})e_j=e_j$;
\item $\nu_{\fA_{\rL}^{\rM}}(\alpha_j)=e(\fB_{\rL}^{\rM})e_j$.
\end{itemize}
We compute directly, and obtain the following inclusions:
\begin{itemize}
\item $U^{d_{j,\rM}}(\fB_{\gamma_{j+1},\max}^{\rM})\cap\rN_{\rL}^{\rM}\subset U^{d_{j,\rG}}(\fB_{\gamma_{j+1},\rL}^{\rG})\cap\rN_{\rL}^{\rM}\subset U^{d_{j,\rL}}(\fB_{\gamma_{j+1},\rL}^{\rM})\cap\rN_{\rL}^{\rM}.$
\item $U^{d_{j,\rL}^+}(\fB_{\gamma_{j+1},\rL}^{\rM})\cap\bar{\rN}_{\rL}^{\rM}\subset U^{d_{j,\rG}^+}(\fB_{\gamma_{j+1},\rL}^{\rG})\cap\bar{\rN}_{\rL}^{\rM}\subset U^{d_{j,\rM}^+}(\fB_{\gamma_{j+1},\max}^{\rM})\cap\bar{\rN}_{\rL}^{\rM}.$
\item $U^{d_{j,\rM}}(\fB_{\gamma_{j+1},\max}^{\rM})\cap\rL= U^{d_{j,\rG}}(\fB_{\gamma_{j+1},\rL}^{\rG})\cap\rL= U^{d_{j,\rL}}(\fB_{\gamma_{j+1},\rL}^{\rM})\cap\rL.$
\end{itemize}
From \cite[Proposition 7.3, Proposition 7.5]{AM}, we deduce that
\begin{equation}
\label{equation16}
U(\fB_{\rL}^{\rM})J_{\max}^{\rM,1}\cap\rN_{\rL}^{\rM}\subset J_{\rG,\rL}^{\rM}\cap\rN_{\rL}^{\rM}\subset J_{\rP,\rL}^{\rM}\cap\rN_{\rL}^{\rM};
\end{equation}
\begin{equation}
\label{equation17}
J_{\rP,\rL}^{\rM}\cap\bar{\rN}_{\rL}^{\rM}\subset J_{\rG,\rL}^{\rM}\cap\bar{\rN}_{\rL}^{\rM}\subset H_{\max}^{\rM,1}\cap\bar{\rN}_{\rL}^{\rM};
\end{equation}
and
\begin{equation}
\label{equation18}
U(\fB_{\rL}^{\rM})J_{\max}^{\rM,1}\cap\rL= J_{\rG,\rL}^{\rM}\cap\rL=J_{\rP,\rL}^{\rM}\cap\rL.
\end{equation}
We deduce Part $(1)$ from the proof of Corollary \ref{cor4.10} and \ref{equation17}. Moreover, it is implied that the intersections $J_{\rG,\rL}^{\rM}\cap U(\fB_{\rL}^{\rM})J_{\max}^{\rM}$ and $J_{\rP,\rL}^{\rM}\cap U(\fB_{\rL}^{\rM})J_{\max}^{\rM}$ are decomposable with respect to $\rP_{\rL}^{\rM}$. Then Part $(2),(3)$ follows.

\end{proof}

Now we show the compatibility property:

\begin{prop}[Compatibility]
\label{propcomparab}
We have:
$$\kappa_{\rL}\cong\kappa_{\rL,0}.$$
\end{prop}

\begin{proof}


By \cite[5.1.16]{BuKuI} we know $\kappa_{\rL,0}^{\rM}$ is the unique subquotient of $\res_{J_{\rL}^{\rM}}^{U(\fB_{\rL}^{\rM})U^1(\fA_{\rL}^{\rM})}\ind_{U(\fB_{\rL}^{\rM})J_{\max}^{\rM,1}}^{U(\fB_{\rL}^{\rM})U^1(\fA_{\rL}^{\rM})}\kappa_{\max}^{\rM}$ which contains $\eta_{\rL}^{\rM}$ (the Heisenberg representation of $\theta_{\rL}^{\rM}$) after restricting to $J_{\rL}^{\rM,1}$. Hence $\kappa_{\rL,0}^{\rM}$ is the unique irreducible subquotient of which the restriction on $H_{\rL}^{\rM,1}$ contains $\theta_{\rL}^{\rM}$. By \cite[7.2.4,7.2.15]{BuKuI}, $\kappa_{\rL,0}$ is the unique subquotient of $\res_{J_{\rP,\rL}^{\rM}}^{U(\fB_{\rL}^{\rM})U^1(\fA_{\rL}^{\rM})}\ind_{U(\fB_{\rL}^{\rM})J_{\max}^{\rM,1}}^{U(\fB_{\rL}^{\rM})U^1(\fA_{\rL}^{\rM})}\kappa_{\max}^{\rM}$ which contains $\theta_{\rL}^{\rM}$ and is $J_{\rL}^{\rM,1}\cap\rN_{\rL}^{\rM}$-invariant. By Part $(1),(2)$ of Lemma \ref{lemcompatinclu}, we have 
$$\Hom_{J_{\rP,\rL}^{\rM}\cap U(\fB_{\rL}^{\rM})J_{\max}^{\rM,1}}(\kappa_{\max}^{\rM},\kappa_{\rL,0})=\Hom_{J_{\rP,\rL}^{\rM}\cap U(\fB_{\rL}^{\rM})J_{\max}^{\rM,1}\cap\rP_{\rL}^{\rM}}(\kappa_{\max}^{\rM},\kappa_{\rL,0})\neq 0.$$
By the second part of Corollary \ref{cor4.10} and \ref{equation18}, we have
$$\Hom_{J_{\rG,\rL}^{\rM}\cap U(\fB_{\rL}^{\rM})J_{\max}^{\rM,1}\cap\rP_{\rL}^{\rM}}(\kappa_{\max}^{\rM},\kappa_{\rL})=\Hom_{J_{\rP,\rL}^{\rM}\cap U(\fB_{\rL}^{\rM})J_{\max}^{\rM,1}\cap\rP_{\rL}^{\rM}}(\kappa_{\max}^{\rM},\kappa_{\rL})\neq\{0\}.$$
Again, by Lemma \ref{lemcompatinclu} and \ref{equation18} we have
$$\Hom_{J_{\rP,\rL}^{\rM}\cap U(\fB_{\rL}^{\rM})J_{\max}^{\rM,1}\cap\rP_{\rL}^{\rM}}(\kappa_{\max}^{\rM},\kappa_{\rL})=\Hom_{J_{\rP,\rL}^{\rM}\cap U(\fB_{\rL}^{\rM})J_{\max}^{\rM,1}}(\kappa_{\max}^{\rM},\kappa_{\rL})\neq\{0\}.$$
By Frobenius reciprocity, it implies that $\kappa_{\rL}$ is an irreducible quotient of $\ind_{J_{\rP,\rL}^{\rM}\cap U(\fB_{\rL}^{\rM})J_{\max}^{\rM,1}}^{J_{\rP,\rL}^{\rM}}\res_{J_{\rP,\rL}^{\rM}\cap U(\fB_{\rL}^{\rM})J_{\max}^{\rM,1}}^{U(\fB_{\rL}^{\rM})J_{\max}^{\rM,1}}\kappa_{\max}^{\rM}$. By Mackey's theory, the latter is a direct component of $\res_{J_{\rP,\rL}^{\rM}}^{U(\fB_{\rL}^{\rM})U^1(\fA_{\rL}^{\rM})}\ind_{U(\fB_{\rL}^{\rM})J_{\max}^{\rM,1}}^{U(\fB_{\rL}^{\rM})U^1(\fA_{\rL}^{\rM})}\kappa_{\max}^{\rM}$. By the unicity property explained above, we have 
$$\kappa_{\rL}\cong\kappa_{\rL,0}.$$

\end{proof}

The triple $((J_{\max}^{\rG},\kappa_{\max}^{\rG}),(J_{\max}^{\rM},\kappa_{\max}^{\rM}),(J_{\rL},\kappa_{\rL}))$ verifies Equation \ref{equationcompat10},\ref{equationcompt11}, \ref{equationcompt12}, \ref{equationcompt14} \ref{equationcompt15}. We call any triple of wild pairs \textbf{a compatible system}. Let $g\in\rL$, we consider a wild pair $(g(J_{\rL}^{\rG}),g(\kappa_{\rL}^{\rG})$, which is defined from a simple stratum $(g(\fA_{\rL}^{\rG}),0,g(\beta))$ with lattice chain $g(\Lambda_{\rL}^{\rG})$. Here $g(\beta)$ means we apply $g$-conjugation to the image of $\beta$ of the simple stratum $(\fA_{\rL}^{\rG},0,\beta)$. The decomposition of $V$ subordinate to $g(\Lambda_{\rL}^{\rG})$ in Remark \ref{remlatticedecomp} is also $V\cong\oplus_{i\in I}V_i$. Hence $g(J_{\rL}^{\rG})$ is decomposable with respect to $\rP_{\rL}^{\rG}$ as well. Moreover, we replace $((J_{\max}^{\rG},\kappa_{\max}^{\rG}),(J_{\max}^{\rM},\kappa_{\max}^{\rM}),(J_{\rL},\kappa_{\rL}))$ by their $g$-conjugation, then we obtain the same results in this section by repeating the same procedure. In particular we have:
\begin{itemize}
\item $(g(J_{X}^{\rG}),g(\kappa_{X}^{\rG}))$ and $(g(J_{\max}^{\rG}),g(\kappa_{\max}^{\rG}))$ satisfy \ref{equationcompat10} for $X=\rL,\rM$.
\item $(g(J_{\max}^{\rM}),g(\kappa_{\max}^{\rM}))$ satisfies \ref{equationcompt11} (resp. \ref{equationcompt12}) with $(g(J_{\rM}^{\rG}),g(\kappa_{\rM}^{\rG}))$ (resp. $(g(J_{\rL}^{\rM}),g(\kappa_{\rL}^{\rM}))$).
\item $(g(J_{\rL}),g(\kappa_{\rL}))$ satisfies \ref{equationcompt14} with $(g(J_{\rL}^{\rG}),g(\kappa_{\rL}^{\rG}))$, and satisfies \ref{equationcompt15} with $(g(J_{\rL}^{\rM}),g(\kappa_{\rL}^{\rM}))$.
\end{itemize}
We conclude that:

\begin{rem}
\label{remkappaconj}
Let $g\in\rL$, the triple $(g(J_{\max}^{\rG},\kappa_{\max}^{\rG}),g(J_{\max}^{\rM},\kappa_{\max}^{\rM}),g(J_{\rL},\kappa_{\rL}))$ is also a compatible system
\end{rem}

\subsubsection{Parabolic induction}
When $(J_{\rL},\lambda_{\rL})$ is of depth zero, there is a well-known property of compatibility of parabolic and parahoric induction. For general cases, to study this property there is a family of functors $\mathbf{K}$ to map a representation of a Levi subgroup of $\rG$ to a representation of a finite reductive group, via taking invariant with respect to wild pairs. In this section, we introduce $\mathbf{K}$ functors of representations of Levi subgroups of $\rG'$.

\begin{defn}
\label{defK'functor}
 Define
 \begin{itemize}
\item We define $\mathbf{K}_{\rL}'$ with respect to $\kappa_{\rL}'$.
$$\mathbf{K}_{\rL}':\mathrm{Rep}_k(J_{\rL}')\rightarrow\mathrm{Rep}_k(\mathbb{L}'),$$
which maps $\pi\in\mathrm{Rep}_k(J_{\rL}')$ to $\Hom_{J_{\rL}^{1'}}(\kappa_{\rL}',\pi)$, such that for $x\in J_{\rL}', f\in\Hom_{J_{\rL}^{1'}}(\kappa_{\rL}',\pi)$, $x\cdot f=\pi(x)\circ f\circ \kappa_{\rL}'(x)^{-1}$. It is a representation of $J_{\rL}'$ and is $J_{\rL}^{1'}$-invariant, hence can be regarded as an $\mathbb{L}'$ representation.
\item Let $(\fA_{\text{max}},0,\beta)$ and $\kappa_{\text{max}},\kappa_{\max}'$ be as in Section \ref{sectionG'cover}. We define $\mathbf{K}_{\text{max}}'$ with respect to $\kappa_{\text{max}}'$
$$\mathbf{K}_{\text{max}}' : \mathrm{Rep}_k(J_{\text{max}}')\rightarrow\mathrm{Rep}_k(\mathbb{L}_{\text{max}}'),$$ 
such that for $x\in J_{\text{max}}', f\in\Hom_{J_{\text{max}}^{1'}}(\kappa_{\text{max}}',\pi)$, $x\cdot f=\pi(x)\circ f\circ \kappa_{\text{max}}'(x)^{-1}$.
\item Define $\mathbf{K}_{max,\alpha}'$ with respect to $\kappa_{max,\alpha}'=\kappa_{\max}'\vert_{J_{\max,\alpha}'}$.
$$\mathbf{K}_{max,\alpha}': \mathrm{Rep}_k(J_{\text{max},\alpha}')\rightarrow\mathrm{Rep}_k(\mathbb{L}'),$$
such that for $x\in J_{\text{max},\alpha}', f\in\Hom_{J_{\text{max},\alpha}^{1'}}(\kappa_{\text{max},\alpha}',\pi)$, $x\cdot f=\pi(x)\circ f\circ \kappa_{\text{max},\alpha}'(x)^{-1}$.
\end{itemize}
\end{defn}

\begin{rem}
\label{exactK}
The functors $\mathbf{K}_{\rL}',\mathbf{K}_{\text{max}}'$ and $\mathbf{K}_{max,\alpha}'$ are exact, since $J_{\rL}^{1'},J_{\text{max}}^{1'}$ and $J_{\max,\alpha}^{1'}$ are pro-$p$ and open. For a representation of $\rL'$, we apply $\mathbf{K}_{\rL}'$ to it by restricting to $J_{\rL}'$ first.
\end{rem}


\begin{lem}
\label{lem2.10}
Recall that $\eta_{\rL}$ is the Heisenberg representation of $\theta_{\rL}$. When $p$ does not divide $\vert W_{\rL}\vert$, for any non-trivial $k$-quasicharacter $\chi$ of $\det(J_{\rL}^1)$, the tensor product $\eta_{\rL}\otimes\chi\circ\det$ is never intertwined with $\eta_{\rL}$.
\end{lem}

\begin{proof}
Suppose $\eta_{\rL}\otimes\chi\circ\det$ is intertwined with $\eta_{\rL}$, then $\theta_{\rL}\otimes\chi\circ\det$ is intertwined with $\theta_{\rL}$. By 3.5.11 of \cite{BuKuI}, there exists $x\in\rU(\fA_{\rL})$ such that $x(\theta_{\rL})\cong\theta_{\rL}\otimes\chi\circ\det$, which implies that the set of such $\chi$ forms a finite subgroup of $(\det(J_{\rL}^1))^\vee$, hence the order of $\chi$ is a power of $p$. Meanwhile, for such a $\chi$ we also have $\eta_{\rL}\otimes\chi\circ\det\cong x(\eta_{\rL})$, which means $\chi\circ\det\vert_{Z(\rL)\cap J_{\rL}^1}$ is trivial, where $Z(\rL)$ is the centre of $\rL$. Write $\rL\cong\prod_{i\in I}\mathrm{GL}_{n_i}(F)$ and $J_{\rL}\cong\prod_{i\in I}J_i$. Let $Z_i$ be the centre of $\mathrm{GL}_{n_i}(F)$. Then by the fact that $\rU(\mathfrak{o}_F)^{n_i}\subset\det(Z_i\cap J_i^1)$, we have $\chi^{n_i}=\mathds{1}$. Since $p$ does not divide $n_i$ for each $i$ by the tameness condition, we have $\chi=\mathds{1}$.
\end{proof}

\begin{prop}
\label{prop2.11}
An irreducible subquotient of $\mathbf{K}_{\rL}'(\ind_{J_{\rL}'}^{\rL'}\lambda_{\rL}')$ is isomorphic to $\sigma_t'$ for a $t\in T$ (see Remark \ref{defnofT}).
\end{prop}

\begin{proof}
Let $\rho$ be an irreducible subquotient of $\mathbf{K}_{\rL}'(\ind_{J_{\rL}'}^{\rL'}\lambda_{\rL}')$. By the definition, $\kappa_{\rL}'\otimes\rho$ is an irreducible subquotient of $\res_{J_{\rL}'}^{\rL'}\ind_{J_{\rL}'}^{\rL'}\lambda_{\rL}'$. In particular, $\kappa_{\rL}'\otimes\rho$ is a subquotient of $\res_{J_{\rL}'}^{\rL}\ind_{J_{\rL}}^{\rL}\lambda_{\rL}$. By \cite[Theorem 2.11]{C1}, we have an equivalence
$$\res_{J_{\rL}'}^{\rL}\ind_{J_{\rL}}^{\rL}\lambda_{\rL}\cong(\res_{J_{\rL}'}^{J_{\rL}}\Lambda(\lambda_{\rL}))\oplus (\res_{J_{\rL}'}^{J_{\rL}}W),$$
where $\Lambda(\lambda_{\rL})$ is a multiple of $\lambda_{\rL}$, and the $\eta_{\rL}$-coinvariant $(\res_{J_{\rL}'}^{J_{\rL}}W)^{\eta_{\rL}}$ is null, where $\eta_{\rL}=\kappa_{\rL}\vert_{J_{\rL}^{1}}$. Now we suppose there exists a $k$-quasicharacter $\chi$, such that $(\res_{J_{\rL}'}^{J_{\rL}}W)^{\eta_{\rL}\otimes\chi\circ\det}$ is non-trivial, which implies that $\eta_{\rL}$ is intertwined with $\eta_{\rL}\otimes\chi\circ\det$ by the fact that $J_{\rL}^1$ is normal in $J_{\rL}$. Hence by Lemma \ref{lem2.10} we deduce that $(\res_{J_{\rL}'}^{J_{\rL}}W)^{\eta_{\rL}\otimes\chi\circ\det}=0$. Suppose the image of $\kappa_{\rL}'\otimes\rho$ is a subquotient of $\res_{J_{\rL}'}^{J_{\rL}}W$. Then by \cite[Proposition 2.6]{C1}, there exists $\chi$ such that $\eta_{\rL}\otimes\chi\circ\det\hookrightarrow W$, which contradicts with the analysis above. Hence $\kappa_{\rL}'\otimes\rho$ is a subquotient of $\res_{J_{\rL}'}^{J_{\rL}}\Lambda(\lambda_{\rL})$ and is intertwined with $\lambda_{\rL}'$, which implies that it belongs to $\{\kappa_{\rL}'\otimes\sigma_t',t\in T\}$.

\end{proof}

The key result in this section is the equivalence in Proposition \ref{propK'}. We apply the same strategy in \cite{SS}, which requires Proposition 3.2, Corollary 3.3 and Corollary 5.5 in \cite{SS}. In particular, Proposition \ref{propK'} is a generalised version to $\rL_{\max}'$ of Proposition 5.6 of \cite{SS}.

\begin{lem}
$\eta_{\rL}'$ is the unique irreducible $k$-representation of $J_{\rL}^{1'}$ such that the restriction $\eta_{\rL}'\vert_{H_{\rL}^{1'}}$ contains $\theta_{\rL}'$.
\end{lem}

\begin{proof}
Suppose $\eta_1'$ is an irreducible $k$-representation of $J_{\rL}^{1'}$ such that $\theta_{\rL}'\hookrightarrow \eta_1'$. Let $\eta_1$ be irreducible of $J_{\rL}^{1}$ such that $\eta_1'\hookrightarrow \eta_1\vert_{J_{\rL}^{1'}}$. Up to twist a $k$-character which factors through determinant, we can assume that $\eta_1\vert_{H_{\rL}^1}$ contains $\theta_{\rL}$. Hence $\eta_{\rL}\cong\eta_1$, which implies that $\eta_1'\cong\eta_{\rL}'$.
\end{proof}

\begin{prop}
Recall that $\rP'=\rL'\rN$. Let $g\in\rL_{\max}'$. The following are equivelant:
\begin{itemize}
\item The functor $\mathbf{K}_{\text{max}}'\circ\ind_{\rP'}^{\rP'gJ_{\text{max}}'}$ is non zero on $\mathrm{Rep}_k(\rL')$;
\item The functor $\mathbf{K}_{\text{max}}'\circ\ind_{\rP'}^{\rP'gJ_{\text{max}}'}$ is non zero on $\mathrm{Irr}_k(\rL')$;
\item $\Hom_{J_{\text{max}}'\cap g(\rN)}(\kappa_{\text{max}}',\mathds{1})\neq 0$.
\item Let $\theta_{\text{max}}'$ be the simple character contained in $\kappa_{\text{max}}'$, then $\theta_{\text{max}}'$ is trivial on $H_{\text{max}}^{1'}\cap g(\rN)=H_{\text{max}}^1\cap g(\rN)$.
\end{itemize}
\end{prop}

\begin{proof}
The proof of Proposition 3.2 and Corollary 3.3 in \cite{SS} can be applied here.
\end{proof}

\begin{prop}[Corollary 5.5 in \cite{SS}]
\label{cor5.5SS}
Let $\pi\in\mathrm{Rep}_k(\rL')$, we have:
$$\mathbf{K}_{\text{max}}'(\ind_{\rP'}^{\rP'J_{\text{max}}'}\pi)\cong i_{\mathbb{L}'}^{\mathbb{G}'}\mathbf{K}_{\rL'}(\pi).$$
\end{prop}

\begin{proof}
We follow the strategy in \cite{SS}, by first showing that
\begin{equation}
\label{equa0011}
\mathbf{K}_{\rL}'(\pi)\cong \mathbf{K}_{max,\alpha}'(\ind_{\rP'}^{\rP'J_{max,\alpha}'}\pi).
\end{equation}
Then
\begin{equation}
\label{equa0022}
\mathbf{K}_{max,\alpha}'(\ind_{\rP'}^{\rP'J_{max,\alpha}'}\pi)\cong\mathbf{K}_{\text{max}}'(\ind_{\rP'}^{\rP'J_{max,\alpha}'}\pi).
\end{equation}
Finally we show that for a representation $\tau$ of $J_{max,\alpha}'$, we have
\begin{equation}
\label{equa0033}
\mathbf{K}_{\text{max}}'(\ind_{J_{max,\alpha}'}^{J_{\text{max}}'}\tau)\cong i_{\mathbb{L}'}^{\mathbb{L}_{\text{max}}'}\mathbf{K}_{\text{max},\alpha}'(\tau).
\end{equation}

The above three equations are generalisation of \cite[Proposition 5.2, Lemma 5.3 and Lemma 5.4]{SS}, of which the proof can be generalised to our case. In particular, the proof of Lemma 5.4 in \cite{SS} can be directly applied to show Equation \ref{equa0033}. Now we sketch the proof for Equation \ref{equa0011} and \ref{equa0022}. We have 
$$\mathbf{K}_{\rL}'(\pi)\cong\mathrm{Hom}_{J_{\rL}^{1'}\cap\rP'}(\kappa_{\rP}',\pi).$$
By the condition $p$ does not divide $\vert W_{\rG'}\vert$, we have $\det(J_{\rP}^1)=1+\mathfrak{p}_{F}$, and $\det(J_{\rL}^1)=1+\mathfrak{p}_{F}$. Hence $J_{\rP}'=J_{\rP}^{1'}(J_{\rP}'\cap\rP')$. Hence by Frobenius reciprocity and the Mackey's formula we have
$$\mathrm{Hom}_{J_{\rP}^{1'}\cap\rP'}(\kappa_{\rP}',\pi)\cong\mathrm{Hom}_{J_{\rP}^{1'}}(\kappa_{\rP}',\res_{J_{\rP}^{1'}}^{J_{\rP}'}\ind_{J_{\rP}'\cap\rP'}^{J_{\rP}'}\pi).$$ 
Then the proof of Proposition 5.2 in \cite{SS} can be applied here, and we obtain Equation \ref{equa0011}. For Equation \ref{equa0022}, we applied the proof of Lemma 5.3 in \cite{SS} by noticing that under the condition $p$ does not divide $\vert W_{\rG'}\vert$, we have $J_{max,\alpha}'=U(\fB)'J_{\text{max}}^{1'}$, hence $\rP'J_{max,\alpha}'=\rP' J_{\text{max}}^{1'}.$
\end{proof}

\begin{prop}[Compatibility of parabolic and parahoric induction]
\label{propK'}
Let $\pi\in\mathrm{Rep}_k(\rL')$, there is an equivalence:
$$\mathbf{K}_{\text{max}}'(i_{\rL'}^{\rL_{\text{max}}'}\pi)\cong i_{\mathbb{L}'}^{\mathbb{L}_{\text{max}}'}\mathbf{K}_{\rL}'(\pi).$$
\end{prop}

\begin{proof}
By the proof of Proposition 5.6 of \cite{SS}, $\theta_{\rL}$ is non-trivial on $H_{\rL}^1\cap g(\rN)$ for all $g\notin \rP J_{\rL}$. Since $H_{\rL}^1\cap g(\rN)= H_{\rL}^{1'}\cap g(\rN)$, $\theta_{\rL}'$ is non-trivial on $H_{\rL}^{1'}\cap g(\rN)$ for all $g\in \rG'\backslash(\rP J_{\rL}\cap\rG')$. Hence 
$$\mathbf{K}_{\text{max}}'(i_{\rL'}^{\rG'}\pi)\cong \mathbf{K}_{\text{max}}'(\ind_{\rP'}^{(\rP J_{\text{max}})\cap\rG'}\pi).$$
By Proposition \ref{cor5.5SS}, we only need to show that $\rP J_{\text{max}}\cap\rG'=\rP' J_{\rL}'$. Write $\rL_{\text{max}}\cong\prod_{s\in S}\mathrm{GL}_{n_s}(F)$, and $J_{\text{max}}\cong\prod_{s\in S}J_s$, where $J_s$ is defined from a maximal simple stratum $[\fA_s,0,\beta]$ and $E=F[\beta]$. We have 
$$\det(J_{\text{max}})=\prod_{s\in S}N_{E\slash F}(\mathcal{O}_{E}^{\times})=\det(\rU(\fB_{\text{max}})),$$
 where $N_{E\slash F}$ denotes the norm mapping. Meanwhile we deduce from the equation $\fB_{\text{max}}\cap\rL=\fB_{\rL}$ that $\det(\rU(\fB_{\rL}))=\det(\rU(\fB_{\text{max}}))$. Since $\rU(\fB_{\rL})\subset J_{\text{max}}\cap \rL$, for $px\in\rP J_{\text{max}}\cap\rG'$ such that $p\in\rP,x\in J_{\text{max}}$ there exists $y\in\rU(\fB_{\rL})$ such that $\det(y)=\det(x)$. Hence $px=pyy^{-1}x\in\rP' J_{\text{max}}'$.
\end{proof}

\begin{cor}
\label{corsupcuspK'}
The supercuspidal support of an irreducible subquotient $\mathbf{K}_{\text{max}}'(i_{\rL'}^{\rL_{\max}'}\ind_{J_{\rL}'}^{\rL'}\lambda_{\rL}')$ belongs to $\{(\mathbb{L}',\sigma_t'),t\in T\}$. In particular, by Remark \ref{defnofT} that $\sigma_t'$ are $\rL'$-conjugate to $\sigma_{\rL}'$.
\end{cor}
\begin{proof}
It is directly deduced from Proposition \ref{prop2.11} and Proposition \ref{propK'}.
\end{proof}

\begin{cor}
\label{corKmax'KM'}
Assume that $\rG'=\rL_{\max}'$. Let $(\kappa_{\text{max}}^{\rG},\kappa_{\text{max}}^{\rM},\kappa_{\rL})$ be the compatible system defined in Section \ref{sectionK'Levi} or its $g$-conjugation where $g\in\rL$ (see Remark \ref{remkappaconj}).  Their restriction $(\kappa_{\text{max}}^{\rG,'},\kappa_{\text{max}}^{\rM,'},\kappa_{\rL}')$ to $\rG',\rM'$ and $\rL'$ accordingly define functors $\mathbf{K}_{\text{max}}^{\rG,'},\mathbf{K}_{\text{max}}^{\rM,'},\mathbf{K}_{\rL}'$ in a same manner as in Definition \ref{defK'functor}. Let $\pi_1,\pi_2$ be a representation of $\rM'$ and $\rL'$ respectively. Denote by $\mathbb{G}'$ the quotient $J_{\max}^{\rG'}\slash J_{\max}^{\rG,1'}$. We have
\begin{itemize}
\item $\mathbf{K}_{\text{max}}^{\rG,'}(i_{\rM'}^{\rG'}\pi_1)\cong i_{\mathbb{M}'}^{\mathbb{G}'}\mathbf{K}_{\text{max}}^{\rM,'}(\pi_1);$
\item $\mathbf{K}_{\text{max}}^{\rM,'}(i_{\rL'}^{\rM'}\pi_2)\cong i_{\mathbb{L}'}^{\mathbb{M}'}\mathbf{K}_{\rL}'(\pi_2).$
\end{itemize}
\end{cor}
\begin{proof}
By replacing $\rL'$ by $\rM'$ and replacing $\rG'$ by $\rM'$, we obtain the above two arguments by applying Proposition \ref{propcomparab} and by repeating the same proof of Proposition \ref{propK'}.
\end{proof}

\subsubsection{Compatibility with parabolic induction}
\label{sectionKb'}
Without loss of generality and to simplify the notations, \textbf{we assume again that $\rL_{\text{max}}=\rG$} in this section. We use the notations in Section \ref{sectionK'Levi}. In general, there exist two different $(\rG',\alpha)$-covers for $(J_{\rL}',\lambda_{\rL}')$, of which their maximal simple strata in $\rG$ are $\rG$-conjugate but not $\rG'$-conjugate, and the same for wild pairs appearing in these $(\rG',\alpha)$-covers. 
We have two goals in this section. First, we determine the $\rG'$-conjugacy classes (finitely many) that need to be considered. Then for each $\rG'$-conjugacy class of maximal simple strata of $\rG$ and each Levi subgroups $\rM$ in between, we fix a triple of wild pairs of $\rG',\rM'$ and $\rL'$ that forms a compatible system and verifies Corollary \ref{corKmax'KM'}. 

Recall that in Section \ref{sectionK'Levi}, we fix the supercuspidal type $(J_{\rL},\lambda_{\rL})$, hence we fix the simple character $\theta_{\rL}$. Also we fix a wild pair $(J_{\max}^{\rG},\kappa_{\max}^{\rG})$. For each Levi subgroup $\rM$ of $\rG$ such that $\rL\subset\rM\subset\rG$, when $\rM\neq\rL$ (resp. $\rM=\rL$) let $(J_{\text{max}}^{\rM},\kappa_{\text{max}}^{\rM})$ (resp. $(J_{\rL},\kappa_{\rL})$) be as in Equation \ref{equationcompt11} (resp. Proposition \ref{propcomparab}). Denote by $\mathrm{det}(\cdot)$ the determinant of a subgroup of $\rG$, and by $\mathrm{Stab}_{\rL}(\kappa_{\text{max}}^{\rG,'})$ the subgroup of $\rL$ contains elements $x$ such that $x(J_{\text{max}}^{\rG,'})=J_{\text{max}}^{\rG,'}$ and $x(\kappa_{\text{max}}^{\rG,'})\cong\kappa_{\text{max}}^{\rG,'}$. We have

\begin{lem}
\label{lemStabkappa'}
There is an inclusion $\det(\mathrm{Stab}_{\rL}(\kappa_{\text{max}}^{\rG,'}))\subset \det(\mathrm{Stab}_{\rL}(\kappa_{\rL}'))$, where $\mathrm{Stab}_{\rL}(\ast)$ indicates the normaliser subgroup in $\rL$. Furthermore, $\det(\mathrm{Stab}_{\rL}(\kappa_{\text{max}}^{\rG,'}))$ has finite index in $\det(\mathrm{Stab}_{\rL}(\kappa_{\rL}'))$.
\end{lem}

\begin{proof}
Let $\pi'$ be an irreducible supercuspidal representation of $\rL'$ which contains a supercuspidal type $(J_{\rL}',\kappa_{\rL}'\otimes\sigma_{\rL}')$. For $x\in\mathrm{Stab}_{\rL}(\kappa_{\text{max}}^{\rG,'})$, by Proposition \ref{propK'} we have $x(\mathbf{K}_{\rL}')(\pi')\neq 0$ (here $x(\mathbf{K}_{\rL}')$ is defined with respect to $x(J_{\rL}',\kappa_{\rL}')$), which implies that there is an irreducible component $\sigma_0'$ of $\sigma_{\rL}\vert_{J_{\rL}'}$, such that $\pi'$ contains $x(\kappa_{\rL}'\otimes\sigma_0')$, and there exists $j\in J_{\rL}$ such that 
$$xj(\kappa_{\rL}'\otimes\sigma_{\rL}')\cong x(\kappa_{\rL}'\otimes\sigma_0').$$
Hence there exists $y\in\rL'$ such that $y^{-1}xj(\kappa_{\rL}'\otimes\sigma_{\rL}')\cong\kappa_{\rL}'\otimes\sigma_{\rL}'$, which implies that $y^{-1}xj$ intertwines $\lambda_{\rL}$ with $\lambda_{\rL}\otimes\chi\circ\det$ for a $k$-character $\chi$ of $F^{\times}$. Since $p$ does not divide $\vert W_{\rG}\vert$, the element $y^{-1}xj$ normalises $\lambda_{\rL}$ hence it belongs to $E_{\rL}^{\times}J_{\rL}$, hence $\det(x)\in\det(E_{\rL}^{\times}J_{\rL})=\det(\mathrm{Stab}_{\rL}(\kappa_{\rL}'))$. The second part of the lemma is obtained by the fact that $\det(\mathrm{Stab}_{\rL}(\kappa_{\text{max}}^{\rG,'}))$ is open and contains the determinant of $Z_{\rL}$ which is the centre of $\rL$, and $\det(\mathrm{Stab}_{\rL}(\kappa_{\rL}'))$ is compact modulo the centre.
\end{proof}

Furthermore, for each Levi subgroup $\rM$ containing $\rL$, we have

\begin{prop}
\label{propdetinclus}
$$\det(\mathrm{Stab}_{\rL}(\kappa_{\text{max}}^{\rG,'}))\subset\det(\mathrm{Stab}_{\rL}(\kappa_{\text{max}}^{\rM,'}))\subset \det(\mathrm{Stab}_{\rL}(\kappa_{\rL}')).$$
\end{prop}

\begin{proof}
By Corollary \ref{corKmax'KM'}, the proof of Lemma \ref{lemStabkappa'} can be applied to obtain the first inclusion by replacing $\kappa_{\rL}'$ by $\kappa_{\text{max}}^{\rM,'}$, and the second inclusion by replacing $\kappa_{\text{max}}^{\rG,'}$ by $\kappa_{\text{max}}^{\rM,'}$.
\end{proof}

\begin{defn}
\label{defnB}
Let $\{b,b\in B\}$ be a finite subset of $\mathrm{Stab}_{\rL}(\kappa_{\rL}')$ (we can choose them in the fixed maximal split torus) whose determinants are representatives of the quotient group 
$$\det(\mathrm{Stab}_{\rL}(\kappa_{\rL}'))\slash\det(\mathrm{Stab}_{\rL}(\kappa_{\text{max}}^{\rG,'})).$$
\end{defn} 

Denote $J_{\text{max}}^{\rG,'}\slash J_{\text{max}}^{\rG,1'}$ by $\mathbb{G}'$. Recall that in general when $\rL_{\text{max}}\neq\rG$, this quotient is denoted by $\mathbb{L}_{\text{max}}'$ in the previous sections.

\begin{rem}[Compatibility after conjugation by $b$]
\label{propconjbK'}
\label{propcompKpara}
\begin{itemize}
\item Denote by $\mathbf{K}_{b}'$ the functor $b(\mathbf{K}_{\text{max}}^{\rG,'})$. For each $b$ and $\pi\in\mathrm{Rep}_k(\rL')$, it normalises $\kappa_{\rL}'$, then by Corollary \ref{corKmax'KM'} we have
$$\mathbf{K}_{b}'(i_{\rL'}^{\rG'}\pi)\cong i_{\mathbb{L}'}^{\mathbb{G}'}\mathbf{K}_{\rL}'(\pi).$$

\item Define $\mathbf{K}_{\rM,b}'$ to be the functor $b(\mathbf{K}_{\text{max}}^{\rM,'})$. For each $b$ and a representation $\rho$ of $\rM'$, by Corollary \ref{corKmax'KM'} we obtain the following equivalences:
\begin{enumerate}
\item $\mathbf{K}_{\rM,b}'(i_{\rL'}^{\rM'}\pi)\cong i_{\mathbb{L}'}^{\mathbb{M}'}\circ\mathbf{K}_{\rL}'(\pi).$
\item $\mathbf{K}_{b}'(i_{\rM'}^{\rG'}\rho)\cong i_{\mathbb{M}'}^{\mathbb{G}'}\circ\mathbf{K}_{\rM,b}'(\rho).$
\item $\mathbf{K}_{b}'(i_{\rL'}^{\rG'}\pi)\cong i_{\mathbb{M}'}^{\mathbb{G}'}\circ\mathbf{K}_{\rM,b}'(i_{\rL'}^{\rM'}\pi).$
\end{enumerate}
\end{itemize}
\end{rem}

\begin{nota}
In Section \ref{sectiondecomp}, we use the simplified notations as below: 
\begin{itemize}

\item We denote $(b(J_{\text{max}}^{\rG,'}),b(\kappa_{\text{max}}^{\rG,'}))$ by $(J_{max,b}',\kappa_{max,b}')$;
\item and $(b(J_{\text{max}}^{\rM,'}),b(\kappa_{\text{max}}^{\rM,'}))$ by $(J_{\rM,b}',\kappa_{\rM,b}')$.
\end{itemize}
\end{nota}

\section{Decomposition}
\label{sectiondecomp}

In this section, we establish blocks decomposition of $\mathrm{Rep}_k(\rG')$. To be more precise, for our fixed supercuspidal pair $(\rL',\tau')$, we establish the block containing the full-subcategory $\mathrm{Rep}_k(\rG')_{[\rL',\tau']}$. 

\subsection{Decomposition of $\mathrm{Rep}_k(\rG')$}
We define full-subcategories of $\mathrm{Rep}_k(\rG')$ which are generated by finitely many supercuspidal classes, then we introduce a non-split property and a decomposition theoreom.  Recall that $(J_{\rL}',\lambda_{\rL}')$ is a supercuspidal type contained in $(\rL',\tau')$, where $\lambda_{\rL}'\cong\kappa_{\rL}'\otimes\sigma_{\rL}'$ and $\kappa_{\rL}'$ is a restriction of a wild pair $(J_{\rL},\kappa_{\rL})$ determined in Proposition \ref{propcomparab}.

\begin{defn}
\label{defnsubgentype}
\begin{itemize}
\item Consider the supercuspidal class $[\rL',\tau']$. Denote by $\mathrm{Rep}_k(\rG')_{[\rL',\tau']}$ the full sub-category generated by $[\rL',\tau']$, which contains $\Pi$ such that the supercuspidal support of any irreducible subquotient of $\Pi$ is contained in $[\rL',\tau']$.
\item Let $\mathcal{I}$ be a union of a family of supercuspidal classes. Define the full sub-category $\mathrm{Rep}_k(\rG')_{\mathcal{I}}$ generated by $\mathcal{I}$, that contains objects $\Pi$ such that the supercuspidal support of any irreducible subquotient of $\Pi$ is contained in $\mathcal{I}$.
\end{itemize}
\end{defn}

\begin{rem}
\label{remcattau'}
We also call $\mathrm{Rep}_k(\rG')_{[\rL',\tau']}$ the full sub-category generated by $[J_{\rL}',\lambda_{\rL}']$, and denote it by $\mathrm{Rep}_k(\rG')_{[J_{\rL}',\lambda_{\rL}']}$ sometimes for the convenience.
\end{rem}

\begin{defn}

\begin{itemize}
\item  Let $\Pi$ be a $k$-representation of $\rG'$. We say $\Pi$ is split if it does not belong to a block  of $\mathrm{Rep}_{k}(\rG')$, otherwise we say it is non-split.
\item We say a full-subcategory is non-split, when it is contained in a block.
\end{itemize}
\end{defn}

Now we introduce two practical results.

\begin{prop}
The full-subcategory $\mathrm{Rep}_k(\rG')_{[\rL',\tau']}$ is non-split.
\end{prop}
\begin{proof}
This is proved in \cite[\S III.5]{V2} under the condition of generic irreducibility. The latter is proved in \cite[Theorem 5.1]{Da05} for groups containing a discrete cocompact subgroup, which is proved to be existed for $\rG'$ in \cite[Theorem 3.3]{BoHa78}.
\end{proof}

Let $[\rL',\tau']_{\rG}$ be the union of supercuspidal classes which are $\rG$-conjugate to $[\rL',\tau']$. It is always a finite union. Let $\mathrm{Rep}_k(\rG')_{[\rL',\tau']_{\rG}}$ be the full-subcategory generated by this union. Denote by $\mathcal{SC}_{\rG'}^{\rG}$ the set of the unions of supercuspidal classes which are $\rG$-conjugate.  The equivalence below has been proved in \cite{C3}. 
$$\mathrm{Rep}_k(\rG')\cong\prod_{\mathcal{SC}_{\rG'}^{\rG}}\mathrm{Rep}_k(\rG')_{[\rL',\tau']_{\rG}}.$$

After the above two statements, to establish blocks of $\mathrm{Rep}_k(\rG')$ is equivalent to establish those of $\mathrm{Rep}_k(\rG')_{[\rL',\tau']_{\rG}}$, which gives a partition on the set of supercuspidal classes in $[\rL',\tau']_{\rG}$. We will construct a finite family of non-split projective objects in $\mathrm{Rep}_k(\rG')_{[\rL',\tau']_{\rG}}$ from projective cover of irreducible representations of finite reductive groups. They can be viewed as building stones of projective generator of blocks of $\mathrm{Rep}_k(\rG')_{[\rL',\tau']_{\rG}}$. We determine those which are in the same block containing $\mathrm{Rep}_k(\rG')_{[\rL',\tau']}$. Finally we find their direct sum is a projective generator of this block.

\subsection{Non-split projective objects}

Let $\sigma'$ be an irreducible representation of $\mathbb{L}':=J_{\rL}'\slash J_{\rL}^{1'}$, and denote by $\mathcal{P}_{\sigma'}$ its projective cover.

\begin{prop}
\label{formprop0016}
An induced representation of the form $P:=i_{\rL'}^{\rG'}\ind_{J_{\rL}'}^{\rL'}\Pi$ where $\Pi\cong\kappa_{\rL}'\otimes\cP_{\sigma'}$ is contained in a block of $\mathrm{Rep}_{k}(\rG')$.
\end{prop}

\begin{proof}
Let $\{\pi_i',i=1,\cdots,s\}$ be the set of irreducible subquotients of $\cP_{\sigma'}$. By Corollary \ref{cor2.3}, there is a standard Levi subgroup $\mathbb{M}_0'
$, that for each $i$ we can find $(\mathbb{M}_0',\sigma_i')$ belonging to the supercuspidal support of $\pi_i'$. Then as in Section \ref{sectionK'Levi} there exists a standard Levi subgroup $\rM_0'\subset \rL$, and a maximal simple stratum $(\fA_{\rM_0},0,\beta)$ of $\rM_0$, as well as a wild-pair $(J_{\rM_0},\kappa_{\rM_0})$ which verifies Equation \ref{equationcompat10}, \ref{equationcompt11} (by replacing $\kappa_{\text{max}}^{\rG}$ by $\kappa_{\rL}$ and $\kappa_{\text{max}}^{\rM}$ by $\kappa_{\rM_0}$), such that $(J_{\rM_0}',\lambda_i'\cong\kappa_{\rM_0}'\otimes\sigma_i')$ belongs to the supercuspidal support of $(J_{\rL}',\kappa_{\rL}'\otimes\pi_i')$ (see Definition \ref{defnJ'supcuspsupp}). By Corollary \ref{corJ'supcuspsupp}, the supercuspidal support of any irreducible subquotients of $i_{\rL'}^{\rG'}\ind_{J_{\rL}'}^{\rL'}\kappa_{\rL}'\otimes\cP_{\sigma'}$ contains $(J_{\rM_0}',\lambda_i')$ for an $i$.

Let $I:=\{[J_{\rM_0}',\lambda_i'],i=1,\cdots,s\}$ be the set of $\rM_0'$-conjugacy classes of $(J_{\rM_0}',\lambda_i')$. Now suppose $P$ is split, which means there exists a non-trivial partition of $I=I_1\cup I_2$, such that the full subcategory $\mathrm{Rep}_k(\rG')_I\cong\mathrm{Rep}_k(\rG')_{I_1}\times\mathrm{Rep}_k(\rG')_{I_2}$, with a decomposition $P\cong P_1\oplus P_2$, where $P_1\in\mathrm{Rep}_k(\rG')_{I_1}$ and $P_2\in\mathrm{Rep}_k(\rG')_{I_2}$. We may assume that $[J_{\rM_0}',\lambda_1']\in I_1$ be the supercuspidal support of $[J_{\rL}',\sigma']$, and suppose $[J_{\rM_0}',\lambda_2']\in I_2$.  We have a filtration 
$$P^1\subset\cdots\subset P^s\cong P,$$
such that $P^i\slash P^{i-1}\cong i_{\rL'}^{\rG'}\ind_{J_{\rL}'}^{\rL'}\kappa_{\rL}'\otimes\pi_i'$.  Let $P^{i,1}$ and $P^{i,2}$ be the image of $P^i$ in the quotient space $P_1$ and $P_2$ accordingly. Hence by Corollary \ref{corsupcuspK'}, $P^{i,1}\slash P^{i-1,1}$ is non-trivial if and only if $[J_{\rM_0}',\lambda_i']\in I_1$. By Frobenius reciprocity 
$$\Pi\hookrightarrow \bar{r}_{\rL'}^{\rG'}P_1\vert_{J_{\rL}'}\oplus \bar{r}_{\rL'}^{\rG'}P_2\vert_{J_{\rL}'},$$
where $\bar{r}$ denotes the opposite parabolic restriction. Let $\Pi_j$ be the image of $\Pi$ in $P_j\vert_{J_{\rL}'}$ for $j=1,2$. 

We consider the functor $\mathbf{K}_{\rL}'$ defined with respect to $\kappa_{\rL}'$, which is exact by Remark \ref{exactK}. Let $\rho'$ be an element in the supercuspidal support of irreducible subquotients of $\mathbf{K}_{\rL}'(\ind_{J_{\rL}'}^{\rL'}\kappa_{\rL}'\otimes\pi_i')$. Since $\ind_{J_{\rL}'}^{\rL'}\kappa_{\rL}'\otimes\pi_i'$ is a subquotient of $i_{\rM_0'}^{\rL'}\ind_{J_{\rM_0}'}^{\rM_0'}\lambda_i'$, by Proposition \ref{prop2.11} and Proposition \ref{propK'}, we have $(J_{\rM_0}',\kappa_{\rM_0}'\otimes\rho')\in[J_{\rM_0}',\lambda_i']$. Hence $\Pi_1\neq \Pi$ and it does not contain $\kappa_{\rL}'\otimes\pi_2'$ as a subquotient by the uniqueness of supercuspidal support Corollary \ref{coruniqsupcuspL'}. By the same reason, $\Pi_2$ is non-trivial, and non of its irreducible subquotients is equivalent to $\kappa_{\rL}'\otimes\sigma'$. Hence there is a surjective morphism from $\Pi$ to an irreducible quotient of $\Pi_2$, different from $\kappa_{\rM}'\otimes\sigma$, which is contradicted to the fact that $\Pi$ is indecomposable.
\end{proof}

Now we come back to the supercuspidal $k$-type $(J_{\rL}',\lambda_{\rL}')$. Recall that $\rL_{\text{max}}$ is the homogeneous Levi subgroup of $\kappa_{\rL}$ in $\rG$ and $\rL_{\text{max}}'=\rL_{\text{max}}\cap\rG'$. Recall that $B$  (see Definition \ref{defnB})  and $T$ (see Remark \ref{defnofT}) are two finite set in $\rL$. Let $\{(J_{\rL}',\lambda_t'),t\in T\}$ be the set of supercuspidal $k$-types in Remark \ref{defnofT}, where $\lambda_t'\cong\kappa_{\rL}'\otimes\sigma_t'$. They are $(E_{\rL}^{\times}J_{\rL})'$-conjugate, and also $\mathbb{L}$-conjugate. We occupy the notations in Section \ref{sectionKb'}. 

\begin{defn}
\label{defcPsim}
Let
\begin{itemize}
\item $\mathcal{B}_{\sigma_{\rL}'}$ be the $\ell$-parablock of $\mathbb{L}_{\text{max}}':=J_{max,b}'\slash J_{max,b}^{1'}$ which contains $i_{\mathbb{L}'}^{\mathbb{L}_{\text{max}}'}\sigma_{\rL}'$, and $\mathcal{B}_{t}$ the $\ell$-parablock of $\mathbb{L}_{\text{max}}':=J_{max,b}'\slash J_{max,b}^{1'}$ which contains $i_{\mathbb{L}'}^{\mathbb{L}_{\text{max}}'}\sigma_{t}'$.
\item Let $\cP(\mathcal{B}_{\sigma_{\rL}'})$ be a projective generator of $\mathcal{B}_{\sigma_{\rL}'}$. Define
$$\cP_{[\lambda_{\rL}',\sim],b,0}:=i_{\rL_{\text{max}}'}^{\rG'}\ind_{J_{max,b}'}^{\rL_{\text{max}}'}\kappa_{max,b}'\otimes\cP(\mathcal{B}_{\sigma_{\rL}'});$$
\item Let $\cP(\mathcal{B}_{t})$ be a projective generator of $\mathcal{B}_{t}$. Define
$$\cP_{[\lambda_{\rL}',\sim],b,t}:=i_{\rL_{\text{max}}'}^{\rG'}\ind_{J_{max,b}'}^{\rL_{\text{max}}'}\kappa_{max,b}'\otimes\cP(\mathcal{B}_{t}).$$
Define
$$\cP_{[\lambda_{\rL}',\sim],b}:=\oplus_{t\in T}\cP_{[\lambda_{\rL}',\sim],b,t},$$
and
$$\cP_{[\lambda_{\rL}',\sim]}:=\oplus_{b\in B}\cP_{[\lambda_{\rL}',\sim],b}.$$
\end{itemize}
\end{defn}

Second adjunction has been proved in \cite{DHKM} for $\ell$-modular setting, which implies that parabolic induction preserve projectivity. Hence $\cP_{[\lambda_{\rL}',\sim],b,t}$ is projective.

\begin{lem}
\label{lemconjBt}
For any $\sigma_1'\in\mathbf{SC}(\cP(\mathcal{B}_t))$, there exists $\sigma_0'\in\mathbf{SC}(\cP(\mathcal{B}_{\sigma_{\rL}'}))$, such that the supercuspidal type $\lambda_1'=\kappa_{\rL}'\otimes\sigma_1'$ is $\rL'$-conjugate to $\lambda_0'=\kappa_{\rL}'\otimes\sigma_0'$.
 \end{lem}
 \begin{proof}
Up to a conjugation in $\mathbb{L}_{\text{max}}'$, we can assume that $\lambda_1'$ is a direct component of $\lambda_{\rL}\vert_{J_{\rL}'}$. Since $t(\lambda_{\rL}')\cong\lambda_t'$, we know from Proposition \ref{propconjparabl} that $t(\cP(\mathcal{B}_{\sigma_{\rL}'}))$ is a projective generator of $\mathcal{B}_t$. Since $t$ normalises $\kappa_{\rL}'$, the conjugation $t^{-1}(\lambda_1')$ is the desired $\lambda_0'.$
By Part 2 of Proposition \ref{proptamecond}, $\lambda_1'$ is weakly intertwined with $\lambda_0'$, hence they are conjugate in $\rL'$.
 \end{proof}

Denote by $\mathrm{Irr}[J_{\rL}',\lambda_{\rL}']$ the set of isomorphism classes of irreducible representations of $\rG'$ of which the supercuspidal supports contain a supercuspidal type in $[J_{\rL}',\lambda_{\rL}']$. For $\Pi\in\mathrm{Rep}_k(\rG)$, denote by $\mathrm{Irr}(\Pi)$ the set of equivalence classes of irreducible subquotients of $\Pi$.

\begin{prop}
\label{IrrPb}
The set $\mathrm{Irr}(\cP_{[\lambda_{\rL}',\sim],b,t})$ does not depend on the choice of $b$ nor $t$, and we denote it by $\mathrm{Irr}[\lambda_{\rL}',\sim]$. In particular, $\mathrm{Irr}(\cP_{[\lambda_{\rL}',\sim]})=\mathrm{Irr}(\cP_{[\lambda_{\rL}',\sim],b,t})$.
\end{prop}

\begin{proof}
First we show that for a fixed $b\in B$, for any $t\in T$ 
$$\mathrm{Irr}(\cP_{[\lambda_{\rL}',\sim],b,t})=\mathrm{Irr}(\cP_{[\lambda_{\rL}',\sim],b,0}).$$
In fact, we have
$$\mathrm{Irr}(\cP_{[\lambda_{\rL}',\sim],b,0})=\bigcup_{\sigma'\in\mathbf{SC}(\cP(\mathcal{B}_{\sigma_{\rL}'}))}\mathrm{Irr}(i_{\rL'}^{\rL_{\text{max}}'}\ind_{J_{\rL}'}^{\rL'}\kappa_{\rL}'\otimes\sigma_{\rL}').$$
From Proposition \ref{propconjparabl} we deduce that $t(\mathbf{SC}(\mathcal{B}_{\sigma_{\rL}'}))=\mathbf{SC}(\mathcal{B}_t)$, hence
$$\mathrm{Irr}(\cP_{[\lambda_{\rL}',\sim],b,t})=\bigcup_{\sigma'\in\mathbf{SC}(\cP(\mathcal{B}_{\sigma_{\rL}'}))}\mathrm{Irr}(i_{\rL'}^{\rL_{\text{max}}'}\ind_{J_{\rL}'}^{\rL'}t(\kappa_{\rL}'\otimes\sigma_{\rL}')).$$
By Lemma \ref{lemconjBt} we have
$$\ind_{J_{\rL}'}^{\rL'}t(\kappa_{\rL}'\otimes\sigma_{\rL}')\cong\ind_{J_{\rL}'}^{\rL'}\kappa_{\rL}'\otimes\sigma_{\rL}'.$$

Now we allow variance on $b\in B$. An element $\pi'$ in $\mathrm{Irr}(\cP_{[\lambda_{\rL}',\sim],b,0})$ is an irreducible subquotient of $i_{\rL_{\text{max}}'}^{\rG'}\ind_{J_{max,b}'}^{\rL_{\text{max}}'}\kappa_{max,b}'\otimes\rho'$ where $\rho'\in\mathcal{B}_{\sigma_{\rL}'}$.
For each $\rL'$-conjugacy class in $\mathbf{SC}(\mathcal{B}_{\sigma_{\rL}'})$, we fix a representative, and denote their set = by $\mathrm{SC}(\mathcal{B}_{\sigma_{\rL}'})$. By Proposition \ref{finprop004}, we can write $\mathrm{SC}(\mathcal{B}_{\sigma_{\rL}'})=\{(\mathbb{L},\sigma_i),i\in I\}$ where $I$ is a finite index. Let $(\mathbb{L},\sigma_i')$ be an element in the supercuspidal support of $\rho'$. By Proposition \ref{iG'cover}, we have
$$i_{\rL'}^{\rG'}\ind_{J_{\rL}'}^{\rL'}\kappa_{\rL}'\otimes\sigma_i'\cong i_{\rL_{\text{max}}'}^{\rG'}\ind_{J_{\rP}'}^{\rL_{\text{max}}'}\kappa_{\rP}'\otimes\sigma_i'.$$
For the righthand side, by Proposition \ref{prop2.6}, Equation \ref{equationcompat10} and Remark \ref{remkappaconj}, we have
$$\ind_{J_{\rP}'}^{\rL_{\text{max}}'}\kappa_{\rP}'\otimes\sigma_i'\cong \ind_{J_{max,b}'}^{\rL_{\text{max}}'}\kappa_{max,b}'\otimes i_{\mathbb{L}'}^{\mathbb{L}_{\text{max}}'}\sigma_i',$$
which implies that $\pi'\in\mathrm{Irr}(\ind_{J_{\rP}'}^{\rL_{\text{max}}'}\kappa_{\rP}'\otimes\sigma_i')$. We conclude that
$$\mathrm{Irr}(\cP_{[\lambda_{\rL}',\sim],b,0})=\bigcup_{i\in I}\mathrm{Irr}(\ind_{J_{\rP}'}^{\rL_{\text{max}}'}\kappa_{\rP}'\otimes\sigma_i'),$$
hence the former is independent of $b\in B$.
\end{proof}

\begin{prop}
\label{formprop006}
Fix $b\in B$. Recall that $\rL$ is a Levi subgroup of $\rG$ such that $\rL'=\rL\cap\rG'$. Recall that $[J_{\rL}',\lambda_{\rL}']_{\rL}$ is the $\rL$-conjugacy class of $(J_{\rL}',\lambda_{\rL}')$. Let $\mathrm{Irr}[J_{\rL}',\lambda_{\rL}']_{\rL}$ be the set of isomorphism classes of irreducible representations of $\rG'$ of which the supercuspidal supports contain a supercuspidal type in the $\rL$-conjugacy class of $(J_{\rL}',\lambda_{\rL}')$. Then
\begin{itemize}
\item $\mathrm{Irr}[\lambda_{\rL}',\sim]\subset\mathrm{Irr}[J_{\rL}',\lambda_{\rL}']_{\rL}$.
\item Suppose for $g\in\mathbb{L}$ the intersection $\mathrm{Irr}[J_{\rL}',g(\lambda_{\rL}')]\cap\mathrm{Irr}[\lambda_{\rL}',\sim]$ is non-empty, then $\mathrm{Irr}[J_{\rL}',g(\lambda_{\rL}')]\subset\mathrm{Irr}[\lambda_{\rL}',\sim]$.
\end{itemize}
\end{prop}

\begin{proof}
For $\pi'\in\mathrm{Irr}[\lambda_{\rL}',\sim]$, there is an irreducible subquotient $\rho$ of $\cP(\mathcal{B}_{\sigma_{\rL}'})$ such that $\pi'$ is an irreducible subquotient of $i_{\rL_{\text{max}}'}^{\rG'}\ind_{J_{\text{max}}'}^{\rL_{\text{max}}'}\kappa_{\text{max}}'\otimes\rho$. By Proposition \ref{finprop004}, the supercuspidal support of $\rho$ is $[\mathbb{L}',g(\sigma_{\rL}')]$ for a $g\in\mathbb{L}_{\text{max}}$, and we can choose $g\in\mathbb{L}$ since $\mathbb{L}_{\max}\slash \mathbb{L}_{\max}'\cong\mathbb{L}\slash\mathbb{L}'$. Hence we can choose a representative of $g$ in $J_{\text{max}}$. By Corollary \ref{iG'cover} and Proposition \ref{prop2.6}, $i_{\rL_{\text{max}}'}^{\rG'}\ind_{J_{\text{max}}'}^{\rL_{\text{max}}'}\kappa_{\text{max}}'\otimes\rho$ is a subquotient of $i_{\rL'}^{\rG'}\ind_{J_{\rL}'}^{\rL'}\kappa_{\rL}'\otimes g(\sigma_{\rL}')$, which shows the first part. On the other hand, by the definition of $\mathcal{B}_{\sigma_{\rL}'}$, all irreducible subquotients of $i_{\mathbb{L}'}^{\mathbb{L}_{\text{max}}'}g(\sigma_{\rL}')$ are in $\mathcal{B}_{\sigma_{\rL}'}$. By Corollary \ref{iG'cover} and Proposition \ref{prop2.6} again, the irreducible subquotients of $i_{\rL'}^{\rG'}\ind_{J_{\rL}'}^{\rL'}\kappa_{\rL}'\otimes g(\sigma_{\rL}')$ are contained in $\mathrm{Irr}[\lambda_{\rL}',\sim]$, hence the result.
\end{proof}

\begin{lem}
\label{lemnocusp}
Suppose that $\rL_{\text{max}}'\neq\rG'$, then $\pi'$ is not cuspidal for $\pi'\in\mathrm{Irr}[\lambda_{\rL}',\sim]$.
\end{lem}

\begin{proof}
Suppose $\pi'\in\mathrm{Irr}[\lambda_{\rL}',\sim]$ is cuspidal, then it contains a cuspidal type $(J',\lambda')$, where $\lambda'\cong\kappa'\otimes\sigma'$ and $\sigma'$ is inflated from a cuspidal representation of $\mathbb{G}'\cong J'\slash J^{1'}$. Let $[\mathbb{M}',\tau']$ be its supercuspidal support. By a same argument in the proof of Proposition \ref{formprop0016}, there exists a standard Levi subgroup $\rM'$ and a wild pair $(J_{\rM}',\kappa_{\rM}')$ such that an element in the supercuspidal support of $\pi'$ contains $(J_{\rM}',\kappa_{\rM}'\otimes\tau')$, and there are wild-pairs $(J_{\rM},\kappa_{\rM})$ of $\rM$ (a Levi subgroup of $\rG$ such that $\rM'=\rM\cap\rG'$) and $(J,\kappa)$ of $\rG$ such that 
\begin{itemize}
\item $J\cap\rG'=J',J_{\rM}\cap\rM'=J_{\rM}'$;
\item $\kappa\vert_{J'}\cong\kappa',\kappa_{\rM}\vert_{J_{\rM}'}\cong\kappa_{\rM}'$;
\item The simple characters contained in $\kappa_{\rM}$ and in $\kappa$ are endo-equivalent (see Section \ref{Lmaxcovers}). 
\end{itemize}
However in Proposition \ref{formprop006}, we have proved that $\pi'\in\mathrm{Irr}[J_{\rL}',\lambda_{\rL}']_{\rL}$, hence $\rM'=\rL'$ and the simple character in $\kappa_{\rM}$ is conjugate $\theta_{\rL}$, hence the latter is endo-equivalent to the simple character in $\kappa$, which contradicts with the fact that the simple characters in $\rG$ is never endo-equivalent to $\theta_{\rL}$ when $\rL_{\text{max}}\neq\rG$, equivalently when $\rL_{\text{max}}'\neq\rG'$.
\end{proof}

\begin{lem}
\label{lemPsurjcusp}
Suppose $\rL_{\text{max}}'=\rG'$. For a cuspidal subquotient $\pi'\leq i_{\rL'}^{\rG'}\tau'$, there exist $b\in B,t\in T$ and cuspidal $\sigma_{b,t}'$ such that $\sigma_{b,t}'\leq i_{\mathbb{L}'}^{\mathbb{L}_{\text{max}}'}\sigma_{t}'$ and $\pi'$ contains $(J_{max,b}',\kappa_{max,b}'\otimes\sigma_{b,t}')$.
\end{lem}

\begin{proof}
Recall that $\tau$ is supercuspidal of $\rL$ such that $\tau'\leq\tau\vert_{\rL'}$, and $\tau$ contains $(J_{\rL},\lambda_{\rL})$, which is defined from $(\fA_{\rL},n,0,\beta)$.  We can find a cuspidal $\pi$ of $\rG$, such that: $\pi'\le\pi\vert_{\rG'}$, and $\tau$ belongs to the supercuspidal support of $\pi$. Let $(J_{max,\alpha},\lambda_{max,\alpha})$ be a $(\rG,\alpha)$-cover of $(J_{\rL},\lambda_{\rL})$, where $\lambda_{max,\alpha}\cong\kappa_{max,\alpha}\otimes\sigma_{\rL}$. Let $\mathbf{K}_{\rL}$ and $\mathbf{K}_{max,\alpha}$ be functors defined from $\kappa_{\rL}$ and $\kappa_{max,\alpha}$. By \cite[Proposition 5.6]{SS} we have
$$\mathbf{K}_{\text{max}}(\pi)\neq0.$$
Then there exists $x\in\rL$ such that 
$$x(\mathbf{K}_{\text{max}}')(\pi')\neq0,$$
where $x(\mathbf{K}_{\text{max}}')$ is defined with respect to $x(J_{\text{max}}',\kappa_{\text{max}}')$. The above inequality holds for all elements with the same determinant of $x$.
By Proposition \ref{propK'} we have
$$x(\mathbf{K}_{\rL}')(\tau')\neq 0,$$
where $x(\mathbf{K}_{\rL}')$ is defined from wild pair $(x(J_{\rL}'),x(\kappa_{\rL}'))$.
Hence $\tau'$ contains $xj(J_{\rL}',\kappa_{\rL}'\otimes\sigma_{\rL}')$ for a $j\in J_{\rL}$. Then $xj(J_{\rL}',\lambda_{\rL}')$ is conjugate to $(J_{\rL}',\lambda_{\rL}')$ by an element $x'\in\rL'$:
$$x'(J_{\rL},\lambda_{\rL}')=xj(J_{\rL}',\lambda_{\rL}'),$$
which is equivalent to say that $y:=(x')^{-1}xj$ belongs to the normaliser group of $\lambda_{\rL}'$ in $\rL$. Hence $y$ intertwines $\lambda_{\rL}$ to $\lambda_{\rL}\otimes\chi\circ\det$ for a $k$-character $\chi$ of $F^{\times}$. Since $\chi$ is trivial by the assumption that $p$ does not divide $\vert W_{\rG}\vert$, hence $y$ intertwines $\lambda_{\rL}$ to itself, then $y$ normalises $\lambda_{\rL}$, which implies that $y\in E_{\rL}^{\times}J_{\rL}$. Then $x\in\rL'E_{\rL}^{\times}J_{\rL}$, and $\det(x)\in\det(\mathrm{Stab}_{\rL}(\kappa_{\rL}'))$. There exists $b\in B$, such that $\det(b)=\det(x)$, which implies that $\mathbf{K}_{b}'(\pi')\neq0$. In particular, $\mathbf{K}_{b}'(\pi')\leq \mathbf{K}_{b}'(i_{\rL'}^{\rG'}\ind_{J_{\rL}'}^{\rL'}\lambda_{\rL}')$. Then  by Corollary \ref{corsupcuspK'}, $\pi'$ contains a cuspidal type $(J_{max,b}',\kappa_{max,b}'\otimes\sigma_{b,t}')$, where $\sigma_{b,t}'\leq i_{\mathbb{L}'}^{\mathbb{L}_{\text{max}}'}\sigma_{t}'$ is cuspidal.
\end{proof}

\begin{prop}
\label{formlem007}
Let $\pi'\in\mathrm{Irr}[\lambda_{\rL}',\sim]$. There is a surjective morphism from $\cP_{[\lambda_{\rL}',\sim],b}$ to $\pi'$ for $b\in B$.
\end{prop}

\begin{proof}
As in the proof of Proposition \ref{formprop006}, there is $g\in \mathbb{L}$ such that $\pi'$ is an irreducible subquotient of $i_{\rL'}^{\rG'}\ind_{J_{\rL}'}^{\rL'}\kappa_{\rL}'\otimes g(\sigma_{\rL}')$. Without loss of generality, we assume that $g=\mathds{1}$. Let $(\mathrm{M}',\rho')$ be in the cuspidal support of $\pi'$ where $\rM'$ is a standard Levi subgroup, and $(J_{0}',\lambda_0')$ is a cuspidal $k$-type contained in $\rho'$. Hence $\rho'$ is an irreducible subquotient of $r_{\rM'}^{\rG'}i_{\rL'}^{\rG'}\ind_{J_{\rL}'}^{\rL'}\kappa_{\rL}'\otimes \sigma_{\rL}'$. Since the irreducible subquotients of $\ind_{J_{\rL}'}^{\rL'}\kappa_{\rL}'\otimes \sigma_{\rL}'$ are supercuspidal, then by Bernstein-Zelevinsky geometric lemma we have $\rho'$ is a subquotient of $i_{\rL'}^{\rM'}
\ind_{x(J_{\rL}')}^{\rL'}x(\kappa_{\rL}'\otimes \sigma_{\rL}')$ for an $x\in\rG'$ that normalises $\rL'$. By applying $x^{-1}$ conjugation to $\rho'$, we may assume that $x=\mathds{1}$. By Lemma \ref{lemnocusp} we have the inclusion $\rM'\subset\rL_{\text{max}}'$. Replacing $\rL_{\text{max}}'$ by $\rM'$, by Proposition \ref{propdetinclus} and Lemma \ref{lemPsurjcusp}, there exists $b\in B,t\in T$, such that $\rho'$ contains $(J_{\rM,b}',\kappa_{\rM,b}'\otimes\sigma_{b,t}')$ where $\sigma_{b,t}'\leq i_{\mathbb{L}'}^{\mathbb{M}'}\sigma_{\rL}'$ is cuspidal of $\mathbb{M'}:=J_{\rM,b}'\slash J_{\rM,b}^{1'}$ and the wild pair $(J_{\rM,b}',\kappa_{\rM,b}')$ is defined in Section \ref{sectionKb'}. Denote by
$$\lambda_{\rM,b,t}':=\kappa_{\rM,b}'\otimes\sigma_{b,t}'.$$
We have 
$$\ind_{J_{\rM}'}^{\rM'}\lambda_{\rM,b,t}'\twoheadrightarrow \rho'.$$ 
Let $(J_{max,\alpha,b}',\lambda_{P,\alpha,b,t}')$ be the $(\rL_{\text{max}}',\alpha)$-cover of $(J_{\rM,b}',\lambda_{\rM,b,t}')$ with respect to $\kappa_{max,b}'$ (see definition in Section \ref{sectionG'alphacover}). Meanwhile by second adjunction of Bernstein we have $i_{\rM'}^{\rG'}\rho'\twoheadrightarrow\pi'$, hence
$$i_{\rL_{\text{max}}'}^{\rG'}\ind_{J_{max,\alpha,b}'}^{\rL_{\text{max}}'}\lambda_{P,\alpha,b,t}'\cong i_{\rM'}^{\rG'}\ind_{J_{\rM,b}'}^{\rM'}\lambda_{\rM,b,t}'\twoheadrightarrow\pi',$$
On the other hand, since $(\kappa_{\max,b},\kappa_{\rM,b},\kappa_{\rL})$ is a compatible system (Remark \ref{remkappaconj}), by Equation \ref{equationcompat10}, Proposition \ref{propcomparab} and Proposition \ref{prop2.6} (by replacing $\rL'$ by $\rM'$), we have 
$$\ind_{J_{max,\alpha,b}'}^{J_{max,b}'}\lambda_{P,\alpha,b,t}'\cong\kappa_{max,b}'\otimes i_{\mathbb{M'}}^{\mathbb{L}_{\text{max}}'}\sigma_{b,t}'.$$
There is an irreducible subquotient $\sigma_{\ast}'$ of $i_{\mathbb{M'}}^{\mathbb{L}_{\text{max}}'}\sigma_{b,t}'$, hence of $i_{\mathbb{L}'}^{\mathbb{L}_{\text{max}}'}\sigma_{\rL}'$, such that
$$i_{\rL_{\text{max}}'}^{\rG'}\ind_{J_{max,b}'}^{\rL_{\text{max}}'}\kappa_{max,b}'\otimes\sigma_{\ast}'\twoheadrightarrow \pi.$$
Hence
$$\cP_{[\lambda_{\rL}',\sim],b,t}=i_{\rL_{\text{max}}'}^{\rG'}\ind_{J_{max,b}'}^{\rL_{\text{max}}'}\kappa_{max,b}'\otimes\cP(\mathcal{B}_{t})\twoheadrightarrow i_{\rL_{\text{max}}'}^{\rG'}\ind_{J_{\text{max}}'}^{\rL_{\text{max}}'}\kappa_{max,b}'\otimes\sigma_{\ast}'\twoheadrightarrow \pi.$$
\end{proof}

\subsection{The blocks of $\mathrm{Rep}_k(\rG')$}
In this section, we establish the blocks of $\mathrm{Rep}_k(\rG')$. Each block is generated by finitely many supercuspidal classes, and we construct a projective generator for each block. 
\begin{defn}
We define an equivalence relation $\sim$ on the set of isomorphism classes of irreducible representations $\{g(\lambda_{\rL}'),g\in\mathbb{L}\}$ in the following way: $g_1(\lambda_{\rL}')\sim g_2(\lambda_{\rL}')$ if and only if $\mathrm{Irr}[J_{\rL}',g_1(\lambda_{\rL}')]\subset\mathrm{Irr}[g_2(\lambda_{\rL}'),\sim]$.
\end{defn}

Proposition \ref{formprop006} shows that 
$$\bigcup_{g(\lambda_{\rL}')\sim\lambda_{\rL}'}\mathrm{Irr}[J_{\rL}',\lambda_{\rL}']=\mathrm{Irr}[\lambda_{\rL}',\sim].$$ 
Denote by
\begin{equation}
\label{equaunion}
[\lambda_{\rL}',\sim]:=\bigcup_{g(\lambda_{\rL}')\sim \lambda_{\rL}'} [J_{\rL}',g(\lambda_{\rL}')].
\end{equation}

\begin{lem}
\label{lemunionlambda}
We have
$$[\lambda_{\rL}',\sim]=\bigcup_{(\mathbb{L}',\sigma')\in\mathbf{SC}(\mathcal{B}_{\sigma_{\rL}'})}[J_{\rL}',\kappa_{\rL}'\otimes\sigma'].$$
\end{lem}

\begin{proof}
From Proposition \ref{IrrPb} and the definition of $\mathcal{P}_{[\lambda_{\rL}',\sim],b,t}$, we know that 
$$\mathrm{Irr}[\lambda_{\rL}',\sim]=\bigcup_{(\mathbb{L}',\sigma')\in\mathbf{SC}(\mathcal{B}_{\sigma_{\rL}'})}\mathrm{Irr}[J_{\rL}',\kappa_{\rL}'\otimes\sigma'].$$
Hence the result.
\end{proof}

Notice that $\mathbf{SC}(\mathcal{B}_{g(\sigma_{\rL}')})=g(\mathbf{SC}(\mathcal{B}_{\sigma_{\rL}'}))$ for $g\in\mathbb{L}$. The connected components given by $\sim$ are transferred from one to another via taking conjugation in $\mathbb{L}$.

The righthand-side of Equation \ref{equaunion} is a finite union, and we can define the full-subcategory $\mathrm{Rep}_k(\rG')_{[\lambda_{\rL}',\sim]}$. Let $(\rL',\tau_g')$ be a supercuspidal pair containing $(J_{\rL}',g(\lambda_{\rL}'))$. Define 
$$[\tau',\sim]:=\bigcup_{g\in\mathbb{L},g(\lambda_{\rL}')\sim\lambda_{\rL}'}[\rL',\tau_g'].$$ 
Then by Remark \ref{remcattau'}, we can write $\mathrm{Rep}_k(\rG')_{[\lambda_{\rL}',\sim]}=\mathrm{Rep}_k(\rG')_{[\tau',\sim]}$. In other words, it contains $\Pi$ such that the supercuspidal supports of irreducible subquotients of $\Pi$ are contained in $[\tau',\sim]$. 

The equivalence relation $\sim$  gives a partition on $\{[J_{\rL}',g(\lambda_{\rL}')],g\in\mathbb{L}\}$. We take $D$ a subset of $\mathbb{L}$ as following: for each connected component of $\{g(\lambda_{\rL}'),g\in\mathbb{L}\}$ defining from $\sim$, there is a unique $d\in D$ such that $(J_{\rL}',\lambda_d')$ belongs to this component, where $\lambda_d':=d(\lambda_{\rL}')\cong\kappa_{\rL}'\otimes\sigma_d'$ and $\sigma_d':= d(\sigma_{\rL}')$. Notice that in general 
$$\bigcup_{g\in\mathbb{L}}[J_{\rL}',\kappa_{\rL}'\otimes g(\sigma_{\rL}')]\neq [J_{\rL}',\lambda_{\rL}']_{\rL}.$$ 
Now, let $\tilde{N}$ be the group of elements in $\rL$ that their conjugation stabilise the union $\bigcup_{g\in\mathbb{L}}[J_{\rL}',\kappa_{\rL}'\otimes g(\sigma_{\rL}')]$. The quotient $Q:=\rL\slash \tilde{N}\cong F^{\times}\slash \det(\tilde{N})$ is finite. We choose a representative of each $q\in Q$ in the fixed maximal split torus of $\rG$, that by identifying $q(J_{\rL}')\slash q(J_{\rL}^{1'})$ with $\mathbb{L}'$ via $q$-conjugation, we have $q(\lambda_{\rL}')\cong q(\kappa_{\rL}')\otimes\sigma_{\rL}'$.  We have a union
$$[J_{\rL}',\lambda_{\rL}']_{\rL}=\bigcup_{g\in \mathbb{L},q\in Q}q[J_{\rL}',g(\lambda_{\rL}')].$$
We denote by $(J_{\rL,q}',\lambda_{d,q}')$ the conjugation $q(J_{\rL}',\lambda_{d}')=(q(J_{\rL}'),q(\kappa_{\rL}')\otimes\sigma_d')$. Notice that by Remark \ref{remkappaconj}, all the results in Section \ref{sectiondecomp} for $(J_{\rL}',\lambda_{\rL}')$ can be applied to $q(J_{\rL}',\lambda_{d}')$ for each $q\in Q,d\in D$. Now for each $d\in D$, we consider a projective generator $\cP(\mathcal{B}_{\sigma_{d}'})=d(\cP(\mathcal{B}_{\sigma_{\rL}'}))$ of the $\ell$-parablock of $\mathbb{L}_{\text{max}}'$ containing $i_{\mathbb{L}'}^{\mathbb{L}_{\text{max}}'}\sigma_d'$. Notice that the sets $B,T,D$ are independent of the choice of $q\in Q$. Now we list the useful objects:
\begin{defn}
\begin{enumerate} 
\item
\begin{itemize}
\item $\cP_{[\lambda_{d,q}',\sim],b,t}:=i_{\rL_{\text{max}}'}^{\rG'}\ind_{q(J_{max,b}')}^{\rL_{\text{max}}'}q(\kappa_{max,b}')\otimes d(\cP(\mathcal{B}_{t})).$
\item $\cP_{[\lambda_{d,q}',\sim],b}:=\oplus_{t\in T}\cP_{[\lambda_{d,q}',\sim],b,t},$
\item  $\cP_{[\lambda_{d,q}',\sim]}:=\oplus_{b\in B}\cP_{[\lambda_{d,q}',\sim],b}.$
\end{itemize}

\item $\mathrm{Irr}[\lambda_{d,q}',\sim]:=\bigcup_{\sigma'\in\mathbf{SC}(\mathcal{B}_{\sigma_{d}'})}\mathrm{Irr}[J_{\rL,q}',q(\kappa_{\rL}')\otimes\sigma_d']=\mathrm{Irr}(\cP_{[\lambda_{d,q}',\sim],b,t})$ for any $b \in B,t\in T$.

\end{enumerate}
\end{defn} 

\begin{prop}
\begin{itemize}
\item $\mathrm{Irr}[\lambda_{d_1,q_1}',\sim]\cap\mathrm{Irr}[\lambda_{d_2,q_2}',\sim]\neq\emptyset$, if and only if $d_1=d_2$ and $q_1=q_2$.
\item $\bigsqcup_{d\in D,q\in Q}\mathrm{Irr}[\lambda_{d,q}',\sim]=\mathrm{Irr}[J_{\rL}',\lambda_{\rL}']_{\rL}$.
\item For each $d\in D,q\in Q$, there is a surjective morphism from $\cP_{[\lambda_{d,q}',\sim]}$ to any $\pi\in\mathrm{Irr}[\lambda_{d,q}',\sim]$.
\end{itemize}
\end{prop}
\begin{proof}
The second part can be read from the definition of $D$ and $Q$. The last part is obtained by applying Proposition \ref{formlem007} to each $d\in D,q\in Q$. For the first part, we notice a fact: for every $g\in\mathbb{L}$, a conjugation by $t\in T$ belongs to the same equivalence class as $g(\lambda_{\rL}')$, and
$$\mathrm{Irr}[J_{\rL}',g(\lambda_{\rL}')]=\mathrm{Irr}[J_{\rL}',\lambda_{\rL}'],$$
if and only if $g(\lambda_{\rL}')=t(\lambda_{\rL}')$ for a $t\in T$. On the other hand,  for any $d_1,d_2\in D$, $[J_{\rL,q_1}',\lambda_{d_1,q_1}']\neq [J_{\rL,q_2}',\lambda_{d_2,q_2}']$ when $q_1\neq q_2$, hence $\mathrm{Irr}[J_{\rL,q_1}',\lambda_{d_1,q_1}']\cap \mathrm{Irr}[J_{\rL,q_2}',\lambda_{d_2,q_2}']=\emptyset$ in this case, which gives the result.
\end{proof}

\begin{rem}
\label{remconjsim}
\begin{itemize}
\item For each $d\in D,q\in Q$, we have an equation
$$[\lambda_{d,q}',\sim]= qd([\lambda_{\rL}',\sim]).$$
\item For each $d\in D,q\in Q$, we take a supercuspidal pair $(\rL',\tau_{d,q}')$ containing $(J_{\rL,q}',\lambda_{d,q}')$. We define $[\tau_{d,q}',\sim]$ in a same manner as $[\tau',\sim]$, which is a union of finite supercuspidal classes containing a supercuspidal type in $[\lambda_{d,q}',\sim]$. By definition, we denote the full sub-category generated by this union in both manners:
$$\mathrm{Rep}_k(\rG')_{[\lambda_{d,q}',\sim]}=\mathrm{Rep}_k(\rG')_{[\tau_{d,q}',\sim]}.$$
\end{itemize}
\end{rem}

Now we conclude:
\begin{thm}
\label{thmblocksdecom}
\begin{itemize}
\item We have the block decomposition:
$$\mathrm{Rep}_k(\rG')_{[J_{\rL}',\lambda_{\rL}']_{\rL}}\cong\prod_{d\in D,q\in Q}\mathrm{Rep}_k(\rG')_{[\lambda_{d,q}',\sim]}.$$
\item In particular when $d=1,q=1$, the representation $\cP_{[\lambda_{\rL}',\sim]}$ is a projective generator of the block $\mathrm{Rep}_k(\rG')_{[\lambda_{\rL}',\sim]}$ that contains $\mathrm{Rep}_k(\rG')_{[J_{\rL}',\lambda_{\rL}']}=\mathrm{Rep}_k(\rG')_{[\rL',\tau']}$.
\end{itemize}
\end{thm}

\begin{proof}
The above proposition implies that the set of projective objects $\{\cP_{[\lambda_{d,q}',\sim]},d\in D,q\in Q\}$ verifies the conditions of Morita's equivalence (Theorem \ref{ThmMorita}). Hence 
$$\mathrm{Rep}_k(\rG')_{[J_{\rL}',\lambda_{\rL}']_{\rL}}\cong\prod_{d\in D,q\in Q}\mathrm{Rep}_k(\rG')_{[\lambda_{d,q}',\sim]}.$$
For the full-subcategory  $\mathrm{Rep}_k(\rG')_{[\lambda_{d,q}',\sim]}$, to show the projective object $\cP_{[\lambda_{d,q}',\sim]}$ is a projective generator is equivalent to show that for arbitrary index set $I$ and $\{\Pi_i\}_{i\in I}\subset\mathrm{Rep}_k(\rG')_{[\lambda_{d,q}',\sim]}$, we have
$$\oplus_{i\in I}\mathrm{Hom}(\cP_{[\lambda_{d,q}',\sim]},\Pi_i)\cong\mathrm{Hom}(\cP_{[\lambda_{d,q}',\sim]},\oplus_{i\in I}\Pi_i).$$
The above equivalence can be deduced from the facts that the opposite parabolic restriction $\bar{r}_{\rL'}^{\rG'}$ commutes with direct sums and $q(\kappa_{max,b}')\otimes d(\cP(\mathcal{B}_t))$ is finite dimensional. In particular, $\cP_{[\lambda_{\rL}',\sim]}$ is the projective generator of the full-subcategory $\mathrm{Rep}_k(\rG')_{[\lambda_{\rL}',\sim]}$.

It is left to show that $\mathrm{Rep}_k(\rG')_{[\lambda_{d,q}',\sim]}$ is non-split. Without loss of generality, we assume that $d=1,q=1$. It is sufficient to show that for each non-trivial partition on the finite set $\{[J_{\rL}',g(\lambda_{\rL}')]\}_{g\in\mathbb{L},g(\lambda_{\rL}')\sim\lambda_{\rL}'}=\mathrm{I}_1\cup\mathrm{I}_2$, we construct a representation $\cP$ of which a decomposition

$$\cP\cong \cP_1\oplus\cP_2,$$
such that $\cP_1\in\mathrm{Rep}_k(\rG')_{\mathrm{I}_1}$ and $\cP_2\in\mathrm{Rep}_k(\rG')_{\mathrm{I}_2}$ is impossible.

Assume an above decomposition exits. Let $(J_{\rL}',\lambda_1')\in I_1$ and $(J_{\rL}',\lambda_2')\in I_2$, where $\lambda_i'\cong\kappa_{\rL}'\otimes\sigma_i'$ for $i=1,2$. By Lemma \ref{lemunionlambda} and Remark \ref{remconjsim}, after conjugation by an element in $T$ (see Remark \ref{defnofT}) we may assume that $i_{\mathbb{L}'}^{\mathbb{L}_{\text{max}}'}\sigma_1'$ and $i_{\mathbb{L}'}^{\mathbb{L}_{\text{max}}'}\sigma_2'$ belong to a same $\ell$-parablock. By the definition of $\ell$-parablock (see Definition \ref{defn2.5}), we can find a family of indecomposable projective objects $\{P_{h},h\leq H\}$ of $\mathbb{L}_{\max}'$ for an $H\in\mathbb{N}$, such that:
\begin{equation}
\label{interSCcP}
\mathbf{SC}(P_{h})\cap\mathbf{SC}(P_{h+1})\neq\emptyset,
\end{equation}
plus $\mathrm{Irr}(P_1)\cap\mathrm{Irr}(i_{\mathbb{L}'}^{\mathbb{L}_{\text{max}}'}\sigma_1')\neq\emptyset$ and $\mathrm{Irr}(P_H)\cap\mathrm{Irr}(i_{\mathbb{L}'}^{\mathbb{L}_{\text{max}}'}\sigma_2')\neq\emptyset$. Denote by $[h]:=\bigcup_{\sigma'\in\mathbf{SC}(P_{h})}[J_{\rL}',\kappa_{\rL}'\otimes\sigma']$. There exists $h_0\leq H$ such that $[h_0]\cap I_1\neq\emptyset$ and $[h_0+1]\cap I_2\neq\emptyset$. We deduce from Equation \ref{interSCcP} that one of the two properties below is verified
\begin{enumerate}
\item $[h_0]\cap I_2\neq\emptyset$,
\item $[h_0+1]\cap I_1\neq \emptyset$.
\end{enumerate}
For an arbitrary $b\in B$. It implies that either $i_{\rL_{\text{max}}'}^{\rG'}\kappa_{max,b}'\otimes P_{h_0}$ or $i_{\rL_{\text{max}}'}^{\rG'}\kappa_{max,b}'\otimes P_{h_0+1}$ is decomposable with respect to 
$$\mathrm{Rep}_k(\rG')_{[\lambda_{\rL}',\sim]}\cong\mathrm{Rep}_k(\rG')_{I_1}\times\mathrm{Rep}_k(\rG')_{I_2}.$$
A contradiction arises from Proposition \ref{formprop0016}. Hence $\mathrm{Rep}_k(\rG')_{[\lambda_{\rL}',\sim]}$ is non-split, and we finish the proof.

\end{proof}


\begin{rem}
\begin{itemize}
\item The projective generator constructed above is highly correlated to the decomposition of depth zero subcategory given by Lanard in \cite{La18}, which is the reason that our result is relevant to the idea of reduction to depth zero.
\item While the author is completing this version, a work on integral blocks of classical groups came to fruition \cite{HKSS}.
\end{itemize}
\end{rem}

\end{document}